\documentclass[a4paper,12pt]{amsart}
\usepackage{amsmath,amsthm,amsfonts,amssymb,graphicx,float,wrapfig,calc,microtype,thmtools}
\usepackage[T1]{fontenc}
\usepackage[english]{babel}
\usepackage[usenames,dvipsnames,svgnames,table]{xcolor}

\usepackage{imakeidx}
\newcommand{\DefNoIndex}[1]{\textcolor{Maroon}{\emph{#1}}}
\newcommand{\defn}[1]{\textcolor{Maroon}{\emph{#1}}\index{#1}}
\newcommand{\hdefn}[2]{\textcolor{Maroon}{\emph{#1-#2}}\index{#2@#1-#2}}
\newcommand{\splitdefn}[2]{\textcolor{Maroon}{\emph{#1 #2}}\index{#2@#1 #2}}
\newcommand{\DefnIndex}[2]{\textcolor{Maroon}{\emph{#1}}\index{#2}}

\makeindex

\usepackage[unicode=true]{hyperref}
\hypersetup{
colorlinks,
linkcolor={black},
citecolor={black},
urlcolor={blue!60!black},
pdftitle={Universality in Minor-Closed Graph Classes},
pdfauthor={Huynh -- Mohar -- Samal -- Thomassen -- Wood}}

\usepackage[shortlabels]{enumitem}
\setlist[itemize]{topsep=0ex,itemsep=0ex,parsep=0ex}
\setlist[enumerate]{topsep=0ex,itemsep=0ex,parsep=0ex}

\usepackage[longnamesfirst,numbers,sort&compress]{natbib}
\makeatletter
\def\NAT@spacechar{~}
\makeatother

\usepackage[margin=30mm]{geometry}
\renewcommand{\baselinestretch}{1.1}
\allowdisplaybreaks

\usepackage[noabbrev,capitalise]{cleveref}
\crefname{lem}{Lemma}{Lemmas}
\crefname{thm}{Theorem}{Theorems}
\crefname{cor}{Corollary}{Corollaries}
\crefname{prop}{Proposition}{Propositions}
\crefname{conj}{Conjecture}{Conjectures}
\crefname{open}{Open Problem}{Open Problems}
\crefname{claim}{Claim}{Claims}
\crefname{enumi}{Item}{Items}
\crefformat{equation}{(#2#1#3)}
\Crefformat{equation}{Equation #2(#1)#3}
\crefformat{enumi}{#2#1#3}
\Crefformat{enumi}{Item (#2#1#3)}

\theoremstyle{plain}
\newtheorem{thm}{Theorem}[section]
\newtheorem{lem}[thm]{Lemma}

\newtheorem{prop}[thm]{Proposition}
\newtheorem{claim}{Claim}[thm]
\theoremstyle{definition}

\usepackage{scalerel,stackengine}
\stackMath
\newcommand\reallywidehat[1]{%
\savestack{\tmpbox}{\stretchto{%
  \scaleto{%
    \scalerel*[\widthof{\ensuremath{#1}}]{\kern.1pt\mathchar"0362\kern.1pt}%
    {\rule{0ex}{\textheight}}%
  }{\textheight}% 
}{2.4ex}}%
\stackon[-6.9pt]{#1}{\tmpbox}%
}

\newcommand{\CartProd}{\mathbin{\square}}

\newcommand{\T}{\ensuremath{\protect{\mathcal{T}}}}
\renewcommand{\ge}{\geqslant}

\renewcommand{\geq}{\geqslant}
\renewcommand{\leq}{\leqslant}

\DeclareMathOperator{\rad}{rad}
\DeclareMathOperator{\tw}{tw}
\DeclareMathOperator{\bw}{bw}
\DeclareMathOperator{\stw}{stw}
\DeclareMathOperator{\pw}{pw}
\DeclareMathOperator{\dist}{dist}

\DeclareMathOperator{\col}{col}
\newcommand{\sreach}{S}

\renewcommand{\thefootnote}{\fnsymbol{footnote}}	
\allowdisplaybreaks
\newcommand{\note}[1]{}
\newcommand{\PP}{P_{\hspace{-0.25ex}\infty}}
\newcommand{\PPP}{C_{\hspace{-0.15ex}\infty}}

\newcommand{\LL}{\mathcal{L}}
\newcommand{\FF}{\mathcal{F}}
\newcommand{\XX}{\mathcal{X}}

\newcommand{\CC}{\mathcal{C}}
\newcommand{\YY}{\mathcal{Y}}
\newcommand{\GG}{\mathcal{G}}

\newcommand{\HH}{\mathcal{H}}
\newcommand{\JJ}{\mathcal{J}}
\newcommand{\DD}{\mathcal{D}}
\newcommand{\TT}{\mathcal{T}}
\newcommand{\RR}{\mathcal{R}}

\newcommand{\NN}{\mathbb{N}}
\newcommand{\ZZ}{\mathbb{Z}}
\renewcommand{\SS}{\mathcal{S}}
\newcommand{\PART}{\mathcal{P}}
\newcommand{\GGG}[2]{#1\langle{#2}\rangle}

\begin{document}

\title{Universality in Minor-Closed Graph Classes\footnotemark[1]}

\footnotetext[1]{1 September 2021. Revised \today}

\author[T. Huynh]{Tony Huynh}
\address{
Discrete Mathematics Group\newline
Institute for Basic Science\newline
Daejeon, South Korea}
\email{tony@ibs.re.kr}
\thanks{Research of T. Huynh is supported by the Institute for Basic Science (IBS-R029-C1).}

\author[B. Mohar]{Bojan Mohar}
\address{Department  of Mathematics\newline 
Simon Fraser University\newline
Burnaby, BC, Canada}
\email{mohar@sfu.ca}
\thanks{Research of B.~Mohar is supported by the NSERC Discovery Grant R832714 (Canada) and in part by the ERC Synergy grant KARST (European Union, ERC, KARST, project number 101071836).}

\author[R. \v{S}\'amal]{Robert \v{S}\'amal}
\address{Computer Science Institute\newline 
Charles University\newline
Prague, Czech Republic}
\email{samal@iuuk.mff.cuni.cz}
\thanks{Research of R.~\v{S}\'amal is partially supported by grant 25-16627S of the Czech Science Foundation. This project has received funding from the European Union's Horizon 2020 research and innovation programme under the MSCA grant agreement No 823748 and also by ERC grant agreement No 810115.}

\author[C. Thomassen]{Carsten Thomassen}
\address{Department of Applied Mathematics and Computer Science\newline 
Technical University of Denmark\newline 
Lyngby, Denmark}
\email{ctho@dtu.dk}
\thanks{Research of C.~Thomassen supported by the Independent Research Fund Denmark, 8021-002498 AlgoGraph.}

\author[D. R. Wood]{David R. Wood}
\address{School of Mathematics\newline 
Monash University\newline
Melbourne, Australia}
\email{david.wood@monash.edu}
\thanks{Research of D.~R.~Wood is supported by the Australian Research Council and by NSERC.}

\subjclass{05C10 planar graphs, 05C83 graph minors, 05C63 infinite graphs}

\begin{abstract}
Stanis{\l}aw Ulam asked whether there exists a universal countable planar graph (that is, a countable planar graph that contains every countable planar graph as a subgraph). J\'anos Pach~(1981) answered this  question in the negative. We strengthen this result by showing that every countable graph that contains all countable planar graphs must contain (i) an infinite complete graph as a minor, and (ii) a subdivision of the complete graph $K_t$, for every finite $t$. 

On the other hand, we construct a countable graph that contains all countable planar graphs and has several key properties such as linear colouring numbers, linear expansion, and every finite $n$-vertex subgraph has a balanced separator of size $O(\sqrt{n})$. The graph is $\TT_6\boxtimes P_{\!\infty}$, where $\TT_k$ is the universal treewidth-$k$ countable graph (which we define explicitly), $P_{\!\infty}$ is the 1-way infinite path, and $\boxtimes$ denotes the strong product. 
More generally, for every $t\in\NN$ we construct a countable graph that contains every countable $K_t$-minor-free graph and has the above key properties. 

Our final contribution is a construction of a countable graph that contains every countable $K_t$-minor-free graph as an induced subgraph, has linear colouring numbers and linear expansion, and contains no subdivision of the countably infinite complete graph (implying (ii) above is best possible). 
\end{abstract}

\maketitle

\renewcommand{\thefootnote}{\arabic{footnote}}

\newpage
\tableofcontents

\setlength{\parindent}{0cm}
\setlength{\parskip}{1.4ex}
\newpage

\section{Introduction}
\label{Intro}

A graph\footnote{Graphs in this paper are simple and either finite or countably infinite, unless explicitly stated otherwise.} is \defn{planar} if it has a drawing in the Euclidean plane (or equivalently on the sphere) with no edge-crossings. Planar graphs are of broad importance in graph theory. The 4-Colour Theorem for colouring  maps~\citep{AH89,RSST97} can be restated as every planar graph is 4-colourable. The 4-Colour Conjecture (as it was) motivated the development of much of 20th century graph theory, and its proof was one of the first examples of a computer-based mathematical proof. The Kuratowski--Wagner Theorem~\citep{Kuratowski30,Wagner37} characterises planar graphs as those graphs not  containing the complete graph $K_5$ or the complete 
\begin{wrapfigure}{r}{59mm}
\centering
{\vspace*{-1.5ex}\includegraphics{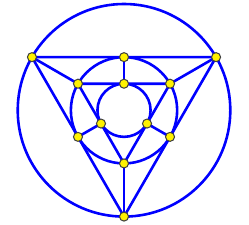}\vspace*{-1.3ex}}
\caption{The icosahedron graph.}
\vspace*{-1.3ex}
\end{wrapfigure}
 bipartite graph  $K_{3,3}$ as a minor. More generally, the Graph Minor Theorem of \citet{RS-GraphMinors} says that any proper minor-closed class of graphs can be characterised by a finite set of excluded minors. Planar graphs are foundational in Robertson and Seymour's Graph Minor Structure Theorem, which shows that graphs in any proper minor-closed class can be constructed using four types of ingredients, where planar graphs are the most basic building block.

Planar graphs also arise in many areas of mathematics outside of graph theory. The Koebe disc packing theorem \citep{Koebe36}, which says that every planar graph can be represented by touching discs in the plane, is the beginning of a rich theory of conformal mappings~\citep{HeSc93,Stephenson05}. Steinitz's Theorem~\citep{Steinitz22} connects planar graphs and classical geometry: the 1-skeletons of polyhedra are exactly the 3-connected planar graphs. Knot diagrams represent knots as planar graphs equipped with over/under relations~\citep{CG12,KK14}. Embeddings of graphs on surfaces, especially triangulations, are fundamental objects in hyperbolic geometry~\citep{STT-JAMS88,CKP19}. Planar graphs also have applications in topology. For example, \citet{Thomassen89a,Thomassen90} presented a simple proof of  the Jordan Curve Theorem based on planar graphs. Planar graphs also arise in physics. Quantum field theory studies Feynman diagrams, which are graphs describing particle interactions. Planar Feynman diagrams \citep{MM19,BBBD15} lead to notions such as the `planar limit' of quantum field theories. Indeed, there is a beautiful theory of `on-shell diagrams' which says that scattering amplitudes in $N=4$ supersymmetric Yang-Mills theory can be calculated using planar graphs~\citep{GK16}. Planar graphs also arise in computational geometry~\citep{GOT18}, for example as Delaunay triangulations of point sets, which leads to  practical applications such as computer graphics and 3D printing. They are central in graph drawing research \citep{HandbookGraphDrawing}, which underlies the field of network visualisation. 

%Curvature, geometry and spectral properties of planar graphs \citep{Keller11}

The focus of this paper is on infinite graphs, which are ubiquitous mathematical objects. For example, they are important in random walks and percolation~\citep{Grimmett99}, where it is desirable to avoid finite-size boundary effects. They also arise as group diagrams in combinatorial group theory \citep{LyndonSchupp77} and low-dimensional topology~\citep{Gersten83,Stallings83}. Bass--Serre theory analyses the algebraic structure of groups by considering the action of automorphisms on infinite trees~\citep{Serre77,Serre80,Bass93}. Several group-theoretic results can be proven by considering a group action on infinite graphs. For example, there is a simple proof that every subgroup of a free group is free using infinite graphs and covering spaces~\citep{NielsenSchreierTheorem}. Infinite Cayley graphs are central objects in geometric group theory \citep{delaHarpe00}. Infinite planar graphs play a key role in the theory of patterns and tilings~\citep{GS87}. Infinite graphs also arise in nanotechnology  \citep{FHI11} and as models of the world wide web~\citep{Bonato08}. In addition, numerous results for finite graphs have led to extensions for infinite graphs or to interesting open problems. See \citep{Thom83,Komjath11,NashWilliams67,Diestel90} for surveys on infinite graphs. 

%Terry Tao writes, ``In the converse direction, one can view infinite graphs as a discretisation of continuous spaces (and infinite Cayley graphs as a discretisation of homogeneous spaces). Gromov's original proof of his theorem relies on this perspective (or more precisely, the idea that homogeneous spaces can arise as limits of infinite Cayley graphs). So the discrete infinitary theory of infinite graphs form a nice bridge between the discrete finitary world and the continuous infinitary world'' \citep{https://mathoverflow.net/questions/39647/applications-of-infinite-graph-theory}. 

\subsection{Universality for Planar Graphs}

This paper addresses universality questions for planar graphs and other more general classes. A graph $G$ \defn{contains} a graph $H$ if $H$ is isomorphic to some subgraph of $G$. A graph $U$ is \defn{universal} for a graph class $\GG$ if $U\in\GG$ and $U$ contains every graph in $\GG$. The starting point for this work is the following question of Ulam: 

%\footnote{Note that the disjoint union of all finite planar graphs is a countable planar graph. So the question of universality of finite planar graphs is not interesting. Similarly, the disjoint union of all countable planar graphs is an uncountable graph that contains all countable planar graphs, which is again not interesting. This says it is important to only consider countable universal graphs.}: 

\smallskip
\emph{Is there a universal graph for the class of countable planar graphs?}

\smallskip 
This question was answered in the negative by \citet{Pach81a}, who proved that no countable planar graph contains every countable planar graph. Pach's result suggests the following question, which motivates the present paper. Recall that graphs in this paper are countable (that is, finite or countably infinite), unless explicitly stated otherwise. 

\smallskip
\emph{What is the `simplest' graph that contains every planar graph?}

\smallskip
Of course, the answer depends on one's measure of simplicity. First, it is desirable that a graph that  contains every planar graph has bounded chromatic number. By the 4-Colour Theorem~\citep{AH89,RSST97} and the de~Bruijn--Erd\H{o}s Theorem~\citep{dBE51}, the complete 4-partite graph\footnote{Let $K_n$ be the complete graph with $n$ vertices. Let $K_{m,n}$ be the complete bipartite graph with $m$ vertices in one colour class and $n$ vertices in the other colour class. Let $K_{n_1,\dots,n_k}$ be the complete $k$-partite graph with $n_i$ vertices in the $i$-th colour class.} with each colour class of cardinality $\aleph_0$  contains every planar graph. But this graph is very dense. At the other extreme, one might hope that some graph with bounded genus or excluding some fixed graph as a minor contains every planar graph. Our first theorem (proved in \cref{ForcingMinor}) dashes this hope and strengthens the above-mentioned result of \citet{Pach81a}.

\begin{thm}
\label{InfiniteCliqueMinor}
If a (countable) graph $U$ contains every planar graph, then 
\begin{enumerate}[(a)]
\item the infinite complete graph $K_{\aleph_0}$ is a minor of $U$, and
\item $U$ contains a subdivision of $K_t$ for every $t\in\NN$. 
\end{enumerate}
\end{thm}

If a graph $U$ contains every planar graph, then \cref{InfiniteCliqueMinor} says it is impossible for $U$ to  exclude a fixed minor or subdivision. However, it is desirable that $U$ satisfies many of the key properties of planar graphs. We contend that these properties include the following:

\smallskip
\begin{enumerate}[(K1)]
\item\label{KeyPropertyBoundedMinDegree} Every finite subgraph of $U$ has bounded minimum degree.
\item\label{KeyPropertySeparators}  Every finite $n$-vertex subgraph of $U$ has a balanced separator of order $O(\sqrt{n})$.
\item\label{KeyPropertyBoundedColouringNumbers} $U$ has bounded $r$-colouring numbers.
\item\label{KeyPropertyNoInfiniteSubdivision} $U$ contains no subdivision of $K_{\aleph_0}$.
\end{enumerate}

\smallskip
We now justify these key properties. \cref{KeyPropertyBoundedMinDegree} is the primary indicator of a sparse graph class, and implies several other desirable properties. In particular, if every finite subgraph of $U$ has minimum degree at most $d\in\NN$, then applying a greedy algorithm, every finite subgraph of $U$ is $(d+1)$-colourable, and in fact is $(d+1)$-list-colourable. 
By the de~Bruijn--Erd\H{o}s Theorem~\citep{dBE51}, $U$ itself is $(d+1)$-colourable, and is $(d+1)$-list-colourable by a similar argument. \cref{KeyPropertySeparators} implies the Lipton--Tarjan separator theorem \citep{LT79}, which is a seminal result about the structure of planar graphs. \cref{KeyPropertyBoundedColouringNumbers} is a robust measure of graph sparsity (defined and discussed in greater depth in \cref{GeneralisedColouringNumbers}). For now, note that \cref{KeyPropertyBoundedColouringNumbers} implies \cref{KeyPropertyBoundedMinDegree}, which in turn implies that $U$ has bounded chromatic number. Moreover, \cref{KeyPropertyBoundedColouringNumbers} implies that numerous other colouring parameters are bounded in $U$, including acyclic chromatic number, game chromatic number, etc. More generally, \cref{KeyPropertyBoundedColouringNumbers} implies that $U$ has bounded expansion, which is a key property in the Graph Sparsity Theory of \citet{Sparsity}. Finally, in light of \cref{InfiniteCliqueMinor}(b), \cref{KeyPropertyNoInfiniteSubdivision} is important since planar graphs have no $K_5$ or $K_{3,3}$ subdivision. 

\subsection{Product Constructions}

The second main contribution of this paper is to construct graphs that  contain every planar graph and satisfy \cref{KeyPropertyBoundedMinDegree}--\cref{KeyPropertyBoundedColouringNumbers}. To describe these graphs we need two ingredients. 

The first ingredient is the notion of treewidth, which is a parameter that measures how similar a given graph is to a tree (see \cref{SimplicialDecompositions} for the definition). For example, a connected graph with at least two vertices has treewidth 1 if and only if  it is a tree. Treewidth is of great importance in structural graph theory, especially in Robertson and Seymour's work on graph minors~\citep{RS-GraphMinors}. The Grid-Minor Theorem of \citet{RS-V} shows that treewidth is related to planarity in the sense that a minor-closed graph class $\GG$ has bounded treewidth if and only if some finite planar graph is not in $\GG$. In a sense detailed below, planar graphs are the smallest minor-closed class with a rich structure. Treewidth and planar graphs are also of great importance in algorithmic graph theory. In particular, many NP-complete problems remain NP-complete for finite planar graphs, but are polynomial-time solvable for classes of finite graphs with bounded treewidth (those excluding some planar graph as a minor). This shows that finite planar graphs often lie at the boundary of hard and easy instances of numerous algorithmic problems. 

It is well known that for each $k\in\NN$, there is a graph $\TT_k$ that is universal for the class of treewidth-$k$ graphs. \cref{Treewidth} gives an explicit definition of such a graph that is of independent interest and important for the main results that follow. 

The second ingredient is the notion of a graph product, as illustrated in \cref{ProductExample}. For graphs $G$ and $H$,  the \defn{Cartesian product} $G\CartProd H$ is the graph with vertex-set $V(G)\times V(H)$, where vertices $(v,w)$ and $(x,y)$ are adjacent if $v=x$ and $wy\in E(H)$, or $w=y$ and $vx \in E(G)$. The \defn{direct  product} $G\times H$ is the graph with vertex-set $V(G)\times V(H)$, where vertices $(v,w)$ and $(x,y)$ are adjacent if $vx\in E(G)$ and $wy\in E(H)$. Finally, the \defn{strong product} $G\boxtimes H$ is the union of $G\CartProd H$ and $G\times H$.

\begin{figure}[!ht]
\centering
\includegraphics{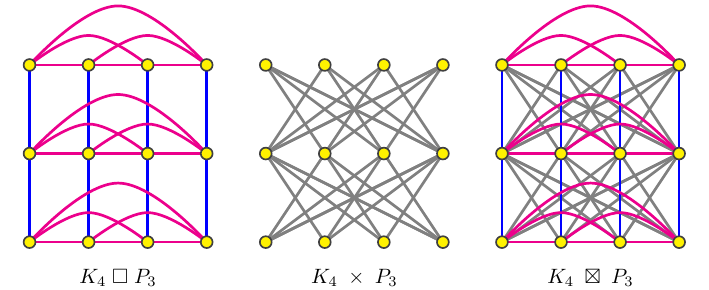}
\caption{Cartesian, direct and strong graph products.\label{ProductExample}}
\end{figure}

With these ingredients in hand, we prove the following, where $\PP$ is the 1-way infinite path.

\begin{thm}
\label{InfinitePlanarStructure}
Both $\TT_6\boxtimes\PP$ and $\TT_3\boxtimes\PP\boxtimes K_3$  contain every planar graph and satisfy \cref{KeyPropertyBoundedMinDegree}--\cref{KeyPropertyBoundedColouringNumbers}.
\end{thm}

We actually prove a strengthening of this result in terms of so-called `simple treewidth'. We define a graph $\SS_k$ that is universal for the class of graphs with simple treewidth $k$, and we show that $\TT$ can be replaced by $\SS$ in \cref{InfinitePlanarStructure}. This is noteworthy since $\SS_3$ is planar but $\TT_3$ is not. These results are proved in \cref{PlanarGraphs}. 

We also establish several results analogous to \cref{InfinitePlanarStructure} for more general graph classes. First, consider graphs embeddable on an arbitrary surface. The \defn{join} $A+B$ is the graph obtained from the disjoint union of graphs $A$ and $B$ by adding all edges $uv$ with $u\in V(A)$ and $v \in V(B)$. 
We show that for each $g\in\NN_0$, each of  
$\SS_3 \boxtimes \PP \boxtimes K_{\max\{2g,3\}}$ and
$(\SS_6 +K_{2g} )  \boxtimes \PP$ contain every graph of Euler genus $g$, and satisfy \cref{KeyPropertyBoundedMinDegree}--\cref{KeyPropertyBoundedColouringNumbers}. This result is proved in \cref{Genus}. We provide similar results for various non-minor-closed classes, such as graphs that can be drawn on a fixed surface with a bounded number of crossings per edge; see \cref{NonMinorClosed}. Moreover, for each finite graph $X$, we construct a graph that contains every $X$-minor-free graph. This graph is more complicated to describe than the above products, but it still satisfies \cref{KeyPropertyBoundedMinDegree}--\cref{KeyPropertyBoundedColouringNumbers}; see \cref{ExcludedMinors}. 

Graph products also provide a mechanism to construct 4-colourable  sparse graphs that contain every planar graph. To see this, note that $\chi(G\times H)\leq\min\{\chi(G),\chi(H)\}$ for all graphs $G$ and $H$. Also note that if $X$ is a $k$-colourable subgraph of $G$, then $X$ is contained in $G\times K_k$. Together, these observations imply that if $U$ is any graph that contains every planar graph, then $U \times K_4$ is 4-colourable and contains every planar graph. For example,  by \cref{InfinitePlanarStructure}, $(\SS_6\boxtimes\PP) \times K_4$ is a 4-colourable graph that contains every planar graph. Moreover, $U\times K_4$ inherits many of the desirable properties of $U$ with a small change in the constants.

\subsection{Avoiding an Infinite Complete Graph Subdivision}

Our final results are related to property \cref{KeyPropertyNoInfiniteSubdivision}, and are presented in \cref{NoInfiniteSubdivision}. Graphs containing no infinite complete graph subdivision were characterised in \citet{RST-TAMS92}. \cref{InfiniteCliqueMinor}(b) says that arbitrarily large (finite) complete graph subdivisions are unavoidable in any graph that contains all planar graphs. On the other hand, we prove that infinite complete graph subdivisions are avoidable. In fact, this result holds for induced subgraphs and in the setting of $K_t$-minor free graphs.

\begin{thm}
\label{Main3} 
For each $t\in\NN$, there is a graph $U_t$ that contains every $K_t$-minor-free graph as an induced subgraph, and satisfies 
\cref{KeyPropertyBoundedMinDegree}, \cref{KeyPropertyBoundedColouringNumbers}
and \cref{KeyPropertyNoInfiniteSubdivision}.
\end{thm}

Since planar graphs contain no $K_5$ minor, the case $t=5$ in \cref{Main3} includes planar graphs. See 
\cref{KtMinorFreeNoSubdiv} in \cref{ExcludedMinorNoSubdiv} for a precise version of \cref{Main3}.

It is open whether the graph $U_t$ in \cref{Main3} satisfies \cref{KeyPropertySeparators}, although $n$-vertex subgraphs of $U_t$ do have balanced separators of size $O(n^{3/4}\log n)$ by a result of \citet{PRS94}; see \cref{Expansion}. 

%%%%%%%%%%%%%%%%%%%%%%%%%%%%%%%%%%%%%%%%%%%%%%%%%%%%%%
\subsection{Related work}
\label{Related}

Before continuing, we survey related work on graph universality, focusing on connections to planar graphs. A graph $U$ is \defn{strongly universal} for a graph class $\GG$ if $U\in\GG$ and every graph in $\GG$ is isomorphic to some induced subgraph of $U$. 

\citet{Ackermann37}, \citet{ER63a} and \citet{Rado64} independently proved that there exists a graph, now called the \defn{Rado graph} or the \defn{random graph} \citep{Cameron01,Cameron97,Cameron84}, that contains every graph as an induced subgraph. \citet{Pach75} went further by showing that there exists a graph $U$ such that every graph $G$ can be isometrically embedded in $U$; that is, the distance between any two vertices in $G$ equals the distance between the corresponding vertices in $U$. See  \citep{Cameron01,Cameron97,Cameron84} for surveys on the Rado graph, and see 
\citep{HK88,KP91,BHM13a,BV14,BH13,Rotman71} for various related results. 

%Rotman71,

%\citet{KP91}: Summary: "A class of graphs has a universal element G0 if every other element of the class is isomorphic to an induced subgraph of G0. In Sections 1–4 we give a survey of some recent developments in the theory of universal graphs in the following areas: (1) graphs universal for isometric embeddings, (2) universal random graphs, (3) universal graphs with forbidden subgraphs, (4) universal graphs with forbidden topological subgraphs. Section 5 is devoted to the problem of deciding how far a class of graphs G is from having a universal element. We introduce a new measure of the complexity of the class G, denoted by cp(G). This is defined to be the minimum cardinal κ such that there exist κ elements in G with the property that any other element of G can be embedded into at least one of them as an induced subgraph. G has a universal element if and only if cp(G)=1. Among other theorems we prove that (i) the complexity of the class of all countable graphs without n≥2 independent edges is finite; (ii) for any cardinal κ, ω1≤κ≤2ω, it is consistent that the complexity of the class of all locally finite countable graphs is equal to κ. In Section 6 we consider some analogous questions for hypergraphs.''

%\citet{Higa04}: We consider embeddings between infinite graphs. In particular, we establish that there is no universal element in the class of countable graphs into which the random graph is not embeddable.

Universal graphs for classes defined by an excluded subgraph have been extensively studied. For example, \citet{Henson71} proved that for each  $t\in\NN$ there is a strongly universal graph for the class of $K_t$-free graphs. Several authors have addressed the question: for which (usually finite) graphs $X$, is there is a strongly universal $X$-free graph \citep{Komjath99,HP84,CS07a,CS16,FK97a,FK97b,CS01,KMP88,KP84,CT07,CSS99}?

\citet{DHV85} considered universality questions for classes excluding $K_t$ as a minor. For $t\leq 4$, they constructed a strongly universal graph, and for $t\geq 5$, they showed there is no universal graph. Note that this result (for $t\geq 5$) is an immediate consequence of our \cref{InfiniteCliqueMinor}(a). 
\citet{DHV85}  also proved that for each integer $t\geq 5$, there is no universal graph for the class of graphs containing no $K_t$-subdivision. This is implied and strengthened by \cref{InfiniteCliqueMinor}(b). 
Similarly, \citet{Diestel85} considered universality questions for classes excluding $K_{m,n}$ as a subdivision. For $m=2$ and $n=3$, Diestel showed there is a universal graph but no strongly universal graph. In the case $m=n=3$, the proof of \citet{Pach81a} can be extended to show that there is no universal graph. For $2\leq n\leq m$ and $m \geq 4$, Diestel showed that there is no universal graph (thus solving a conjecture of Halin).

\citet{BHM13} constructed, for each integer $k$, a strongly universal graph for the class of graphs with the property that every  subgraph is $k$-degenerate (that is, every finite subgraph has minimum degree at most $k$).  Similarly, for every finite graph $H$, \citet{MMS09} constructed a strongly universal graph for the class of graphs that admit a homomorphism to $H$.  By Euler's formula, every finite planar graph has minimum degree at most $5$, and  by the 4-Colour Theorem, every planar graph admits a homomorphism to $K_4$. However, there are 2-degenerate graphs (such as 1-subdivisions of complete graphs) and there are $4$-colourable graphs (such as complete 4-partite graphs) that satisfy none of \cref{KeyPropertySeparators}--\cref{KeyPropertyNoInfiniteSubdivision}. So degeneracy or homomorphisms to a fixed graph $H$ are too broad for our purposes.  

%%%%%%%%%%%%%%%%%%%%%%\subsubsection{Universal Graph for Finite Classes Defined by Excluded Subgraph --- Induced Subgraphs}

Universality questions for finite graphs are also widely studied (usually for induced subgraphs). Results are known for trees \citep{ADK17,CGP78,CG78,GL68,Bodini02}, cycles and paths \citep{AAHKS20}, bounded degree graphs \citep{AC07,AA02,ELO08,AAHKBS17,AN19,Butler09}, graphs with a given number of vertices \citep{Moon65,AKTZ19,Alon17,BT81}, graphs with a given number of vertices and edges \citep{BDS19,CE83}, and sparse  graphs \citep{BCEGS82,BCLR89}. The following question has been the focus of attention for finite planar graphs: what is the least number of vertices in a graph that contains every $n$-vertex planar graph as an induced subgraph~\citep{BGP22,KNR92,GL07,DEGJMM21,EJM23,DGJMMW}? Improving on a long sequence of results, \citet{DEGJMM21} recently obtained an optimal answer for this question, by constructing, for each value of $n$, a graph with $n^{1+o(1)}$ vertices that contains every $n$-vertex  planar graph as an  induced subgraph. \citet{EJM23} improved this result by constructing a graph with $n^{1+o(1)}$ vertices and edges that contains every $n$-vertex  planar graph as an induced subgraph. \citet{EJM23} also constructed a graph with $(1+o(1))n$ vertices and $n^{1+o(1)}$ edges that contains every $n$-vertex planar graph (as a subgraph). The proofs in \cite{BGP22,DEGJMM21,EJM23,DGJMMW} all use \cref{FinitePlanarStructure} below, which is also a key tool in our work. 

%%%%%%%%%%%%%%%\subsubsection{Universality under the Minor Relation}

One can also ask universality questions for the minor relation.  \citet{RST94} proved that every $n$-vertex planar graph is a minor of the planar grid $P_{14n} \CartProd P_{14n}$. The question becomes more challenging for infinite planar graphs. For example, \citet{DK99} observed that the planar graph obtained from $K_4$ by joining to each of its four vertices infinitely many new vertices of degree 1 is not a minor of the infinite grid. Nevertheless, \citet{DK99} constructed an infinite planar graph $U$ such that every planar graph is a minor of $U$. 

%\subsubsection{Uncountable Graphs}

Finally, we mention uncountable graphs. \citet{Wagner67} proved that a graph (of any cardinality) is planar if and only if it has at most continuumly many vertices, no $K_5$ or $K_{3,3}$ subdivision, and at most countably many vertices of degree at least 3. See \citep{Kojman98,KP84} for more on universality for uncountable graphs. 

\section{Preliminaries}
\label{Preliminaries}

This section contains preliminary material, including definitions~(\cref{Definitions}); 
an introduction to extendability~(\cref{Extendability}); 
universality for trees~(\cref{Trees}); and
useful results about simplicial decompositions, chordal graphs and tree-decompositions~(\cref{SimplicialDecompositions}). 
Then we provide more detail about some of the key properties of sparse graph classes that were introduced in \cref{Intro}. In particular, we look at the relationship between separators and treewidth~(\cref{SeparatorsTreewidth}), generalised colouring numbers~(\cref{GeneralisedColouringNumbers}), and bounded expansion~(\cref{Expansion}). 

%%%%%%%%%%%%%%%%%%%%%%%%%%%%%
\subsection{Definitions}
\label{Definitions}

We  use the following notation: $\NN:=\{1,2,\dots\}$ and $\NN_0:=\NN\cup\{0\}$. For $m,n\in\ZZ$ with $m\leq n$, let $[m,n]:=\{m,m+1,\dots,n\}$ and $[n]:=\{1,2,\dots,n\}$. 

We use standard graph-theoretic notation and terminology \citep{Diestel5}.

We consider undirected graphs $G$ with vertex-set $V(G)$ and edge-set $E(G)\subseteq \binom{V(G)}{2}$ (with no loops or parallel edges). A graph $G$ is \defn{trivial} if $E(G)=\emptyset$. A graph $G$ is \defn{finite} if $V(G)$ is finite, and is \defn{countable} if $V(G)$ is countable (that is, finite or countably infinite). Recall that, for brevity, we assume all graphs are countable. 

A graph $H$ is a \defn{subgraph} of a graph $G$ if $V(H)\subseteq V(G)$ and $E(H)\subseteq E(G)$. Two graphs $G_1$ and $G_2$ are \defn{isomorphic} if there is a bijection $f:V(G_1)\to V(G_2)$ such that for all $v,w\in V(G_1)$ we have $vw\in E(G_1)$ if and only if $f(v)f(w)\in E(G_2)$. 

The \defn{neighbourhood} of a vertex $v$ in a graph $G$ is $N_G(v):= \{w\in V(G): vw\in E(G)\}$; then $v$ has \defn{degree} $\deg_G(v):= |N_G(v)|$. For $S\subseteq V(G)$, let $N_G(S):=\bigcup_{v\in S}(N_G(v)\setminus S)$. A graph $G$ is \hdefn{$k$}{regular} if every vertex of $G$ has degree $k$. A graph $G$ is \defn{locally finite} if every vertex in $G$ has finite degree.

A \defn{path} in a graph $G$ is a sequence of distinct vertices, denoted $v_1,v_2, \ldots v_n$ or $v_1,v_2,\dots$, where $v_1v_2,v_2v_3,\dots,v_{n-1}v_n\in E(G)$. A \hdefn{$vw$}{path} in a graph $G$ is a path from $v$ to $w$. Let $\dist_G(v,w)$ be the length of a shortest $vw$-path in $G$. If there is no such path, then $\dist_G(v,w)=\infty$. A $vw$-path $P$ in $G$ is a \defn{geodesic} if the length of $P$ equals $\dist_G(v,w)$. For $A,B\subseteq V(G)$, an \hdefn{$AB$}{path} in $G$ is a $vw$-path $P$, such that $v\in A$ and $w\in B$ and no internal vertex of $P$ is in $A\cup B$.

A \defn{cycle} in a graph $G$ is a circular sequence of distinct vertices, denoted $(v_1,v_2,\dots,v_n)$, where $v_1v_2,v_2v_3,\dots,v_{n-1}v_n,v_nv_1\in E(G)$. A cycle $C$ in a graph $G$ is \defn{Hamiltonian} if  every vertex of $G$ is in $C$. Where it is convenient we consider a path or a cycle to be the corresponding subgraph. 

A graph $G$ is \defn{locally Hamiltonian} if for each vertex $v\in V(G)$, the subgraph $G[N_G(v)]$ has a Hamiltonian cycle. By definition, every locally Hamiltonian graph is locally finite. The following lemma will be useful.

\begin{lem}
\label{LocallyHam3Conn}
Every connected locally Hamiltonian graph is 3-connected. 
\end{lem}

\begin{proof}
Suppose for the sake of contradiction that $G$ is locally Hamiltonian but not 3-connected. Since $G$ is locally Hamiltonian, it has at least four vertices. If $G$ has a cut-vertex $v$, then $G[N_G(v)]$ is not connected, contradicting that $G$ is locally Hamiltonian. So  $G$ is 2-connected. Since $G$ is not 3-connected and has at least four vertices, $G$ has two vertices $u,v$ such that $G-u-v$ is disconnected. Since $G$ is 2-connected, $u$ has neighbours $u_1,u_2$ in distinct components of $G-u-v$. Now any path joining $u_1,u_2$ in $G[N_G(u)]$ must contain $v$. Hence $G[N_G(u)]$ has no Hamiltonian cycle, which is the desired contradiction. 
\end{proof}

For a vertex $v$ in a directed graph $G$, the \defn{in-neighbourhood} of $v$ is $N^-_G(v):= \{w\in V(G): wv\in E(G)\}$ and the \defn{out-neighbourhood} of $v$ is $N^+_G(v):= \{w\in V(G): vw\in E(G)\}$. Then $v$ has \defn{in-degree} $\deg^-_G(v):= |N^-_G(v)|$ and \defn{out-degree} $\deg^+_G(v):= |N^+_G(v)|$. 

A set $\GG$ of graphs is a \defn{graph class} if for every graph $G\in\GG$, every graph isomorphic to $G$ is also in $\GG$. A graph class $\GG$ is \defn{monotone} if for every $G\in\GG$ every subgraph of $G$ is in $\GG$. A graph class $\GG$ is \defn{hereditary} if for every $G\in\GG$ every induced subgraph of $G$ is in $\GG$. A graph class $\GG$ is \defn{countable} if $\GG$ has countably many equivalence classes under the isomorphism relation. 

A \defn{colouring} of a graph $G$ is a function that assigns one `colour' to each vertex of $G$, such that adjacent vertices of $G$ are assigned distinct colours. For $k\in\NN$, a \hdefn{$k$}{colouring} is a colouring with at most $k$ colours. The \defn{chromatic number}, $\chi(G)$, of $G$ is the minimum $k\in\NN$ such that $G$ is $k$-colourable. If there is no such $k$, then $G$ has chromatic number $\chi(G)=\infty$. 

A \defn{clique} in a graph $G$ is a set of pairwise adjacent vertices in $G$. The \defn{clique-number}, $\omega(G)$, of $G$ is the maximum $k\in\NN$ such that $G$ has a clique of cardinality $k$. If there is no such $k$, then $G$ has clique number $\omega(G)=\infty$. 

For $d\in\NN$, a graph $G$ is \hdefn{$d$}{degenerate} if every finite subgraph of $G$ has minimum degree at most $d$. The \defn{degeneracy} of $G$ is the minimum $d\in\NN$ such that $G$ is $d$-degenerate.

\index{infinite path} The \DefNoIndex{1-way infinite path} $\PP$ is the graph with vertex-set $\NN$ and edge-set $\{\{n,n+1\}:n\in\NN\}$. The \DefNoIndex{2-way infinite path} $\PPP$ has vertex-set $\ZZ$ and edge-set $\{\{n,n+1\}:n\in\ZZ\}$. 

A graph $H$ is a \defn{minor} of a graph $G$ if a graph isomorphic to $H$ can be obtained from a subgraph of $G$ by contracting edges. A graph class $\GG$ is \defn{minor-closed} if for every graph $G\in\GG$, every minor of $G$ is also in $\GG$. A minor-closed class $\GG$ is \defn{proper} if some graph is not in $\GG$. 

A \defn{model} of a graph $H$ in a graph $G$ is a collection $(X_v)_{v\in V(H)}$ of pairwise disjoint connected subgraphs of $G$ indexed by the vertices of $H$, such that for each  edge $vw\in E(H)$ there is an edge of $G$ between $X_v$ and $X_w$. Each subgraph $X_v$ is called a \defn{branch set} of the model. Note that $H$ is a minor of $G$ if and only if there is a model of $H$ in $G$. 

A \defn{subdivision} of a graph $H$ is any graph obtained from $H$ by repeatedly applying the following operation: delete an edge $vw$, introduce a new vertex $x$, and add new edges $vx$ and $xw$. A graph $G$ \defn{contains an $H$-subdivision} if $G$ contains a subgraph isomorphic to a subdivision of $H$. In this case, $H$ is a minor of $G$. 

Planar graphs form a proper minor-closed class. 
More generally, for any surface $\Sigma$, the class of graphs embeddable in $\Sigma$ form a proper minor-closed class. Here a \defn{surface} is a 2-dimensional manifold without boundary.  The \defn{Euler genus} of the orientable surface with $h$ handles is $2h$. The \defn{Euler genus} of  the non-orientable surface with $c$ cross-caps is $c$. The \defn{Euler genus} of a graph $G$ is the minimum Euler genus of a surface in which $G$ embeds (with no crossings). See~\citep{MoharThom} for background on embeddings of graphs on surfaces.

If $H$ is a subgraph of a graph $G$, then a \defn{chord} of $H$ (with respect to $G$) is an edge $vw\in E(G)\setminus E(H)$ with $v,w\in V(H)$. A graph $G$ is \defn{chordal} if every cycle $C$ of length at least 4 has a \defn{chord}.

A \defn{vertex-partition}, or simply \defn{partition}, of a graph $G$ is a collection $\PART$, where each element of $\PART$ is a subset of $V(G)$, and each vertex of $G$ is in exactly one element of $\PART$. Each element of $\PART$ is called a \defn{part}. A partition $\PART$ of a graph $G$ is \defn{finite} if each part of $\PART$ is finite. The \defn{quotient} of $\PART$ is the graph, denoted by $G/\PART$, with vertex-set $\{A\in\PART:A\neq\emptyset\}$ where distinct parts $A,B\in V(G/\PART)$ are adjacent if and only if some vertex in $A$ is adjacent in $G$ to some vertex in $B$. A partition of $G$ is \defn{connected} if the subgraph induced by each part is connected. In this case, the quotient is a minor of $G$ (obtained by contracting each part into a single vertex). \note{We slightly changed these definitions to allow empty parts in a partition $\PART$, and in the quotient $G/\PART$ we only take the non-empty parts.}

A partition $\PART$ of a graph $G$ is \defn{chordal} if the quotient $G/\PART$ is chordal. 

A set $S$ of vertices in a graph $G$ is \defn{separating} if $G-S$ is disconnected. Sometimes we also say $G[S]$ is separating. A separating set $S$ in $G$ is \defn{minimal} if no proper subset of $S$ is separating. If $S$ is a minimal separating set, then for each component $X$ of $G-S$, each vertex in $S$ has a neighbour in $X$.

An \defn{orientation} of a graph $G$ is a directed graph obtained from $G$ by directing each edge from one end to the other. An orientation is \defn{acyclic} if there is no directed cycle. In an acyclically oriented graph, if there is a directed path from a vertex $v$ to a vertex $w$, then $v$ is an \defn{ancestor} of $w$ and $w$ is a \defn{descendant} of $v$.

In an orientation of a graph $G$, a path $v_1,v_2,\dots$ in $G$ is \defn{backward} if $v_iv_{i+1}$ is oriented from $v_{i+1}$ to $v_i$, for each $i\geq 1$. 

\note{In places we previously used `anti-directed' to mean `backward'. We now use `backward' consistently.}

%%%%%%%%%%%%%%%%%%%%%%%%%%%%%%%%%%%%%%%%%%%%%%%
\subsection{Extendability}
\label{Extendability}

A graph class $\Gamma$ is \defn{extendable} if the following property holds for every graph $G$: if every finite subgraph of $G$ is in $\Gamma$, then $G$ is in $\Gamma$. (Recall that $G$ is assumed to be countable.) The following result, essentially due to Erd\H{o}s, provides an example of an extendable class. 
 
\begin{lem}[see \citep{DS54,Imrich75}] 
\label{PlanarityExtendable}
The class of planar graphs is extendable. 
\end{lem}

\begin{proof}
If a finite graph $H$ is a subdivision of a graph $G$, then $H$ is a subdivision of a finite subgraph of $G$. That is, if no finite subgraph of $G$ contains a subdivision of $H$, then $H$ is not a subdivision of $G$.  Kuratowski's Theorem says that a finite graph $G$ is planar if and only if $G$ contains no $K_5$ or $K_{3,3}$ subdivision~\citep{Kuratowski30}. 
Erd\H{o}s extended this result to the setting of countable graphs by an elegant argument based on K\"onig's Lemma below (see \cite{Thomassen83} for example). By the above observation, $G$ is planar if and only if no finite subgraph contains a $K_5$ or $K_{3,3}$ subdivision. That is, $G$ is planar if and only if every finite subgraph of $G$ is planar. Hence planarity is extendable.
\end{proof}

A function $f$ is a \defn{graph parameter} if $f(G)\in\NN$ for every graph $G$, and $f(G_1)=f(G_2)$ for all isomorphic graphs $G_1$ and $G_2$. A graph parameter is \defn{extendable} if for every graph $G$ and integer $k\in\NN$, the class $\{G : f(G) \leq k\}$ is extendable. 

The following two lemmas are useful for showing that certain graph parameters are extendable. For example, they each imply that chromatic number is extendable (the de~Bruijn--Erd\H{o}s~Theorem~\citep{dBE51}). 

\begin{lem}[K\"onig's Lemma \citep{Konig27}] 
\label{Konig} 
Let $V_1,V_2,\dots$ be an infinite sequence of disjoint non-empty finite sets. Let $G$ be a graph with vertex-set $\bigcup_{n\in\NN} V_n$, such that for all $n\in\NN$ every vertex in $V_{n+1}$ has a neighbour in $V_n$. Then $G$ contains a path $v_1,v_2,\dots$ with $v_n\in V_n$ for all $n\in\NN$. 
\end{lem}

\begin{lem}[Zorn's Lemma; see \citep{Konig27}] \label{Zorn} 
If a partially ordered set $P$ has the property that every chain in $P$ has an upper bound in $P$, then $P$ contains at least one maximal element.
\end{lem}

%%%%%%%%%%%%%%%%%%%%%%%%%%%%%
\subsection{Trees}
\label{Trees}

A \defn{tree} is a connected graph with no cycles. An orientation of a tree is a \hdefn{$1$}{orientation} if every vertex has in-degree at most 1. For every directed edge $vw$ in a 1-oriented tree, $v$ is the \defn{parent} of $w$ and $w$ is a \defn{child} of $v$. If there is a directed path from a vertex $v$ to a vertex $w$ in a 1-oriented tree, then $v$ is an \defn{ancestor} of $w$ and $w$ is a \defn{descendant} of $v$. 
If a vertex $r$ in a 1-oriented tree $T$ has in-degree 0, then every edge of $T$ is oriented away from $r$, implying that $r$ is the only vertex with in-degree 0. In this case, we say $T$ is \defn{rooted} at $r$. 
In a tree $T$ with a specified root vertex $r$, we implicitly consider the edges of $T$ to be oriented away from $r$.

\begin{lem}
\label{TreeOrientation}
For every vertex $r$ of a tree $T$ there is a 1-orientation of $T$ rooted at $r$. Moreover, a tree $T$ has an unrooted 1-orientation if and only if $T$ contains a 1-way infinite path. \end{lem}

\begin{proof}
For the first claim, simply orient every edge of $T$ away from $r$, to obtain a 1-orientation of $T$ rooted at $r$.  

Now suppose that $T$ contains a 1-way infinite path $P$ starting at some vertex $v$. Orient every edge of $P$ towards $v$, and orient every edge of $T-E(P)$ away from $P$. Then every vertex has in-degree exactly 1, and $T$ has an unrooted 1-orientation. 

Finally, suppose that $T$ has an unrooted 1-orientation. So every vertex has a parent. Let $v_1$ be any vertex in $T$. Suppose that $v_1,v_2,\dots,v_i$ is a backward path in $T$ (meaning that each edge is oriented from $v_j$ to $v_{j-1}$). Let $v_{i+1}$ be the parent of $v_i$. So $v_1,\dots,v_i$ are all descendants of $v_{i+1}$. Repeating this step, we obtain a 1-way infinite backward path $v_1,v_2,\dots$. 
\end{proof}

%%%%%%%%%%%%%%%%%%%%%%%%%%%%%
\subsection{Treewidth and Simplicial Decompositions}
\label{SimplicialDecompositions}

This section introduces graph treewidth and its connection to Halin's Simplicial Decomposition Theorem via chordal graphs. 

\note{Added more explanation here.}

A \defn{tree-decomposition} of a graph $G$ is a collection $(B_x\subseteq V(G):x\in V(T))$ of subsets of $V(G)$ (called \defn{bags}) indexed by the nodes of a tree $T$, such that:
\begin{enumerate}[label=(\alph*)]
\item for every edge $uv\in E(G)$, some bag $B_x$ contains both $u$ and $v$, and
\item for every vertex $v\in V(G)$, the set $\{x\in V(T):v\in B_x\}$ induces a non-empty subtree of $T$.
\end{enumerate}
We emphasise that the subtree in (b) must be connected.  Note that every countable graph has a tree-decomposition with finite bags\footnote{If $V(G)=\{v_1,v_2,\dots\}$ then the sequence $( \{ v_1 \}, \{v_1,v_2\}, \{v_1,v_2,v_3\},\dots )$ defines a tree-decomposition indexed by a path, in which every bag is finite.}.  A graph $G$ has \defn{bounded treewidth} if for some $k \in \NN_0$, $G$ has a tree-decomposition in which every bag has size at most $k+1$. In this case, the \defn{treewidth} of $G$, denoted by \defn{$\tw(G)$}, is the minimum such $k\in\NN_0$. In this paper, the $\tw(G)$ notation implicitly means that $G$ has bounded treewidth. Treewidth is recognised as the most important measure of how similar a given graph is to a tree. Indeed, a connected graph with at least two vertices has treewidth 1 if and only if it is a tree. 

See \citep{Diestel90,KT90,KT90a,KT91} for work on tree-decompositions of infinite graphs. 
\citet{RST-TAMS92} defined a graph $G$ to have \defn{treewidth $<\aleph_0$} if $G$ has a tree-decomposition $(B_x\subseteq V(G):x\in V(T))$ such that $B_x$ is finite for all $x \in V(T)$, and
\begin{enumerate}[(c)]
\item For each one-way infinite path $x_1,x_2, \ldots$ in $T$, 
$\bigcup\limits_{i=1}^\infty \bigcap\limits_{j\geq i} B_{x_j}$ is finite.
\end{enumerate}
For example, the disjoint union of all finite complete graphs has treewidth $<\aleph_0$, whereas the countably infinite complete graph $K_{\aleph_0}$ does not. 

A \defn{path-decomposition} is a tree-decomposition in which the underlying tree is a 1-way infinite path. We denote a path-decomposition by the corresponding sequence of bags $(B_1,B_2,\dots)$. A graph $G$ has \defn{bounded pathwidth} if, for some $k\in\NN_0$, $G$ has a path-decomposition in which every bag has size at most $k+1$. In this case, the \defn{pathwidth} of $G$, denoted by \defn{$\pw(G)$}, is the minimum such $k\in\NN_0$.
A graph $G$ has \defn{bounded bandwidth} if, for some $k\in\NN_0$, there is a linear ordering $v_1,v_2,\dots$ of $V(G)$, such that $|i-j|\leq k$ for every edge $v_iv_j\in E(G)$. In this case, the \defn{bandwidth} of $G$, denoted by \defn{$\bw(G)$}, is the minimum such $k\in\NN_0$. It is easily seen that $\tw(G)\leq\pw(G)\leq \bw(G)$.

The general definition of a simplicial decomposition can be found in \cite{Thom83}. Here, we give the definition in the special case of countable graphs containing no $K_{\aleph_0}$. Let $H_1,H_2,\ldots$ be pairwise disjoint (finite or infinite) graphs that contain no $K_{\aleph_0}$ and no finite separating complete subgraph. (We consider the empty graph to be complete, so each $H_i$ is connected.)\ Form a sequence of graphs $G_1,G_2, \ldots$ as follows. Let $G_1:=H_1$. Having defined $G_n$, define $G_{n+1}$ by selecting, for some $p\in\NN_0$,  a $K_p$ subgraph in $G_n$ and a $K_p$ subgraph in $H_{n+1}$, and identifying the two copies of $K_p$. Let $G := \bigcup_{n\in\NN} G_n$. If $G$ contains no   $K_{\aleph_0}$, then $H_1,H_2, \ldots$ are said to form a \defn{simplicial decomposition} of $G$, where each $H_i$ is called a \defn{simplicial summand}. \citet{Halin64a} proved the following (where, as always, we assume graphs are countable):

\begin{thm}[\citep{Halin64a}]
\label{HalinsTheorem}
Every graph containing no $K_{\aleph_0}$ has a simplicial decomposition. 
\end{thm}

% \note{We inserted the next lemma to address the referee's comment at the beginning of Section 6.1.}

% \begin{lem}
% \label{subchordal}
% A graph $H$ has treewidth $k$ (respectively treewidth $<\aleph_0$) if and only if $H$ is a spanning subgraph of a chordal graph $G$ with treewidth $k$ (treewidth $<\aleph_0$).
% \end{lem}

% \begin{proof}
% The if-part is trivial. So assume $H$ has a tree-decomposition showing that $H$ has treewidth $k$ (respectively treewidth $<\aleph_0$). For each bag (which is finite) in the tree-decomposition,
% add all possible edges and call the resulting graph $G$. Clearly, $G$ has treewidth $k$ (respectively, treewidth $<\aleph_0$). It is an easy exercise to prove that $G$ is chordal.
% \end{proof}

% As pointed out above, $K_{\aleph_0}$ does not have treewidth $<\aleph_0$, so the chordal graph $G$ in \cref{subchordal} contains no $K_{\aleph_0}$. Conversely, if a chordal graph $G$ contains no  $K_{\aleph_0}$, then Halin's theory of simplicial decompositions (to be discussed later) implies that $G$ has treewidth $<\aleph_0$.

We frequently use the following lemma, which easily follows from the Helly property for trees.  

\note{Throughout the paper we have added lemma numbers for results cited from the literature, as requested by the referee.}

\begin{lem}[{\protect\citep[Corollary~12.3.5]{Diestel5}}] 
\label{TreeDecompositionClique}
For every graph $G$ and every finite clique $X$ of $G$, every tree-decomposition of $G$ has a bag containing $X$.
\end{lem}

The next lemmas show the equivalence between simplicial decompositions and tree-decompositions, and their relationship with chordal graphs. We include the proofs for completeness.

\begin{lem}
\label{SimpDecTreeDec}
Let $H_1,H_2, \ldots$ be pairwise disjoint graphs, none containing an infinite complete subgraph or a finite separating complete subgraph. Then $H_1,H_2, \ldots$  form a simplicial decomposition of a graph $G$ if and only if there is a tree-decomposition $(B_x:x\in V(T))$ of $G$ for some rooted tree $T$, and a bijection $f:V(T)\to\NN$ such that $f(v)<f(w)$ whenever $v$ is the parent of $w$, and $G[B_x]\cong H_{f(x)}$ for every $x\in V(T)$, and $G[B_x \cap B_y]$ is a finite complete graph for each $xy\in E(T)$.
\end{lem}

\begin{proof}
Suppose $H_1,H_2, \ldots$ form a simplicial decomposition of $G$.
We claim that for each $n\in\NN$, the graph $G_n$ defined above has a tree-decomposition $(B_x:x\in V(T_n))$ for some  $n$-vertex tree $T_n$ rooted at node $r$, and there is a bijection $f:V(T)\to[n]$ such that $f(r)=1$ and $f(v)<f(w)$ whenever $v$ is the parent of $w$, and
$G[B_x]\cong H_{f(x)}$ for every $x\in V(T)$. 
For $n=1$, let $T_1$ be a 1-node tree $T_1$ with vertex-set $\{r\}$, let $f:V(T_1)\to[1]$ be a bijection, and let $B_x:=V(H_1)$. Consider $T_1$ to be rooted at $r$. Assume we have the claimed tree-decomposition and bijection for some $n\in\NN$. 
Then $G_{n+1}$ is obtained from $G_n$ by identifying, for some $p\in\NN$, a $K_p$ subgraph in $G_n$ with a $K_p$ subgraph in $H_{n+1}$. The first $K_p$ subgraph is in some bag $B_x$ by \cref{TreeDecompositionClique}. Let $T_{n+1}$ be the tree obtained from $T_n$ by adding one new node $y$ adjacent to $x$, where $f(y):=n+1$ and $B_y:=V(H_{n+1})$. Consider $T_{n+1}$ to be rooted at $r$. We obtain the claimed tree-decomposition and bijection for $n+1$. Let $T:=\bigcup_{n\in\NN} T_n$. Then $(B_x:x\in V(T))$ is the desired tree-decomposition of $G$ and $f$ is the desired bijection.

Conversely, suppose there is a tree-decomposition $(B_x:x\in V(T))$ of $G$ for some rooted tree $T$, and a bijection $f:V(T)\to\NN$ such that 
$f(v)<f(w)$ whenever $v$ is the parent of $w$, and
$G[B_x] \cong H_{f(x)}$ for every $x\in V(T)$, and
$G[B_x \cap B_y]$ is a finite complete graph for every $xy\in E(T)$.
We prove by induction on $n\in\NN$ that $H_1,\dots,H_n$ form a simplicial decomposition of $G_n:=G[ \cup_{x\in V(T),f(x)\in[n]} B_x ]$. This trivially holds for $n=1$. Assume that $H_1,\dots,H_n$ form a simplicial decomposition of $G_n$. Let $y:= f^{-1}(n+1)$. Let $x$ be the parent of $y$ in $T$. So $f(x)<f(y)$, implying $f(x)\in[n]$. By assumption, $G[B_y]\cong H_{n+1}$. Let $S:=B_x \cap B_y$. By assumption, $G[S]\cong K_p$ for some $p\in\NN$. Thus $G_{n+1}$ is obtained from $G_n$ and $H_{n+1}$ by identifying $S$ (in $G_n$) with the image of $S$ under the assumed isomorphism from $G[B_y]$ to $H_{n+1}$. Thus $H_1,\dots,H_n,H_{n+1}$ form a simplicial decomposition of $G_{n+1}$. And $H_1,H_2,\dots$ form a simplicial decomposition of $G$. 
\end{proof}

Let $G$ be a graph.  A \defn{simplicial orientation} of $G$ is an acyclic orientation of $G$ such that the in-neighbourhood of every vertex is a clique.  An orientation of $G$ is \defn{well-founded} if $G$ contains no infinite 1-way backward path. This terminology is inspired by the analogous concept from posets.

\note{Added the well-founded definition to deal with an error found by the referee.}

\begin{lem}
\label{ChordalCharacterisation}
The following are equivalent for a graph $G$ containing no $K_{\aleph_0}$:
\begin{enumerate}[(a)] 
\item\label{ChordalCharacterisationChordal} $G$ is chordal;
\item\label{ChordalCharacterisationSimpDec} $G$ has a simplicial decomposition in which each simplicial summand is a finite clique;  
\item\label{ChordalCharacterisationTreeDec} $G$ has a tree-decomposition in which each bag is a finite clique; 
\item\label{ChordalCharacterisationSimplicialOrientation} $G$ has a simplicial and well-founded orientation; 
\item\label{ChordalCharacterisationSimplicial} $G$ has a simplicial orientation; 
\item\label{ChordalCharacterisationMinimalSep} every minimal separating set in $G$ is a clique.
\end{enumerate}
\end{lem}

\begin{proof}
\cref{ChordalCharacterisationChordal} $\Longrightarrow$ \cref{ChordalCharacterisationMinimalSep}: Assume that $G$ is chordal. Let $S$ be a minimal separating set in $G$. Let $X$ and $Y$ be distinct components of $G-S$. Suppose that $v$ and $w$ are non-adjacent vertices in $S$. Since $S$ is minimal, 
there is a shortest $vw$-path $P$ in $G$ with every internal vertex in $X$, and
there is a shortest $vw$-path $Q$ in $G$ with every internal vertex in $Y$. 
Thus $P\cup Q$ is a chordless cycle in $G$. This contradiction shows that $S$ is a clique. 

\cref{ChordalCharacterisationChordal} $\Longrightarrow$ \cref{ChordalCharacterisationSimpDec}: Assume that $G$ is chordal. By \cref{HalinsTheorem}, there is a simplicial decomposition $H_1,H_2,\dots$ of $G$. Every induced subgraph of a chordal graph is chordal, so each $H_i$ is chordal.  Since we have already established \cref{ChordalCharacterisationChordal} $\Longrightarrow$ \cref{ChordalCharacterisationMinimalSep}, 
every minimal separator in $H_i$ induces a complete subgraph. By the definition of a simplicial decomposition, each $H_i$ has no separating complete subgraph. Thus, $H_i$ has no minimal separator, implying $H_i$ is complete. Finally, $H_i$ is finite since $G$ contains no $K_{\aleph_0}$.

\cref{ChordalCharacterisationSimpDec}
$\Longleftrightarrow$ \cref{ChordalCharacterisationTreeDec} follows immediately from \cref{SimpDecTreeDec}.

\note{The following proof has been added.} 

\cref{ChordalCharacterisationTreeDec} $\Longrightarrow$ \cref{ChordalCharacterisationSimplicialOrientation}: 
Assume that $G$ has a tree-decomposition $(B_x:x\in V(T))$ in which each bag is a clique. Root $T$ at an arbitrary node $r$. For each vertex $v\in V(G)$, let $x_v$ be the node of $T$ closest to $r$ such that $v\in B_{x_v}$. For each node $x\in V(T)$, let $B'_x:=\{v\in V(G):x_v=x\}$.  Note that $B'_{x}$ and $B'_y$ are disjoint for all distinct $x,y$.  For each $x\in V(T)$, acyclically orient the clique on $B'_x$. Consider an edge $e=vw$ of $G$. 
If $x_v=x_w$ then we have already oriented $e$. If $x_v$ is closer to $r$ than $x_w$, then orient $e$ from $v$ to $w$. If $x_w$ is closer to $r$ than $x_v$, then orient $e$ from $w$ to $v$. Now $G$ is acyclically oriented, and for every vertex $v\in V(G)$, the in-neighbourhood of $v$ is a subset of $B_{x_v}$ and is therefore a clique.  We now show that this orientation is well-founded.  Suppose not.  For each infinite 1-way backward path $P$, let $r(P)$ be the unique vertex of $P$ with out-degree 0 in $P$.  Among all infinite 1-way backward paths $P$ of the orientation, choose $P$ to minimise the distance from $r$ to $x_{r(P)}$ in $T$.  Let $v$ be the vertex closest to $r(P)$ in $P$ such that $v \notin B_{r(P)}$ (which exists since $|B_{r(P)}|$ is finite).  By the choice of $v$, we have that $w \in B_{r(P)}$ where $w$ is the unique out-neighbour of $v$ in $P$. By the definition of our orientation, $x_v$ is closer to $r$ than $x_w$.  But now the subpath of $P$ ending at $v$ contradicts the choice of $P$.  Thus the orientation of $G$ is well-founded.

\cref{ChordalCharacterisationSimplicialOrientation} $\Longrightarrow$ 
\cref{ChordalCharacterisationSimplicial} is immediate.

\cref{ChordalCharacterisationSimplicial} $\Longrightarrow$ \cref{ChordalCharacterisationChordal}: Assume that $G$ has a simplicial orientation. Let $C$ be a cycle of length at least 4 in $G$. Since the orientation is acyclic, $C$ has a vertex $v$ with in-degree 2 in $C$. Thus, the neighbours of $v$ in $C$ are adjacent in $G$. So $C$ has a chord. Thus $G$ is chordal.

\cref{ChordalCharacterisationMinimalSep} $\Longrightarrow$ 
\cref{ChordalCharacterisationChordal}: Assume that every minimal separating set in $G$ is a clique. Let $C$ be a cycle in $G$ of length at least 4. Let $v,w$ be non-adjacent vertices in $C$. Let $S$ be a minimal separating set in $G$ such that $v$ and $w$ are in distinct components of $G-S$. So $S$ is a clique. Each of the two $vw$-paths in $C$ have a vertex in $S$, implying $C$ has a chord. Hence $G$ is chordal. 
\end{proof}

A \defn{simplicial $k$-orientation} of a chordal graph $G$ is an acyclic orientation of $G$ such that the in-neighbourhood of every vertex is a clique of size at most $k$. For brevity, a simplicial $k$-orientation is henceforth called a \hdefn{$k$}{orientation} (which matches the definition of 1-orientation when $k=1$). A $k$-orientation is \defn{rooted} if each connected component of $G$ has a vertex of in-degree 0. 

\note{The next lemma has been added.}  

\begin{lem}
\label{WellFoundedRooted}
For every well-founded acyclic orientation of a graph $G$, there is a root vertex. Moreover, for every rooted simplicial orientation of a connected chordal graph, there is a unique root. 
\end{lem}

\begin{proof}
Let $G$ be an acyclically oriented graph such that each vertex has indegree at least 1. Let $v_1,v_2,\dots $ be a maximal backward oriented path in $G$. Given any vertex $v_i$ let $v_{i+1}$ be any neighbour of $v_i$ where the edge $v_iv_{i+1}$ is directed from $v_{i+1}$ to $v_i$. Since the orientation is acyclic, $v_{i+1}\neq v_j$ for all $j\in[i]$. So $v_1,v_2,\dots$ is a 1-way infinite backward oriented path. This proves the first claim. 

Now assume that $G$ is a connected chordal graph with a rooted simplicial orientation. Say $r$ is a vertex of in-degree 0. Let $v$ be any vertex of $G$. Let $P$ be a shortest $rv$-path in $G$ (ignoring the edge orientation). If there is a vertex $x$ in $P$ incident with two incoming edges $yx,zx$ in $P$, then $yz$ is an edge, implying that $P$ is not a shortest $rv$-path. So no vertex in $P$ has two incoming edges in $P$. The edge incident to $r$ in $P$ is outgoing at $r$. So $P$ is oriented from $r$ to $v$. In particular, $v$ has in-degree at least 1. Thus $G$ has exactly one vertex of in-degree 0.
\end{proof}

The next lemma follows from \cref{ChordalCharacterisation,TreeDecompositionClique}. 

\begin{lem}
\label{ChordalCharacterisation2}
The following are equivalent for a graph $G$ and $k\in\NN$:
\begin{itemize}
\item $G$ is chordal containing no $K_{k+2}$ subgraph;
\item $G$ is chordal and is $(k+1)$-colourable; 
\item $G$ has a tree-decomposition in which each bag is a clique of size at most $k+1$;
\item $G$ has a $k$-orientation. 
\end{itemize}
\end{lem}

The next lemma (which is well known when $G$ is finite) says that every graph has a `normalised' tree-decomposition.

\begin{lem}
\label{StandardTreewidth}
Let $G$ be a graph with treewidth at most $k\in\NN$. Then 
for every vertex $r$ in $G$, there is a tree $T$ with $V(T)=V(G)$ and there is a tree-decomposition $(B_x:x\in V(T))$ of $G$ with width at most $k$ such that:
\begin{itemize}
\item $T$ is rooted at $r$ and $B_r=\{r\}$,
\item $v\in B_v$ for each $v\in V(G)$, 
\item for every edge $vw$ of $G$, $v$ is an ancestor of $w$ or $w$ is an ancestor of $v$ in $T$, and
\item for every non-root node $y\in V(T)$, if $x$ is the parent of $y$ in $T$, then $|B_y\setminus B_x|=1$. 
\end{itemize}
Moreover, $G$ has a $(k+1)$-colouring such that distinct vertices in the same bag are assigned distinct colours. 
\end{lem}

\note{This proof has been re-written to avoid the ``repeat this operation forever'' step.} 

\begin{proof}
If $E(G)=\emptyset$ then the result is trivial. Now assume that $E(G)\neq\emptyset$. Let $(B_x:x\in V(T))$ be an arbitrary tree-decomposition of $G$ with width at most~$k$. If $r$ is isolated in $G$, then let $x$ be any node in $T$ with $|B_x|\geq 2$, which exists since $E(G)\neq\emptyset$. If $r$ is not isolated in $G$, then let $s$ be any neighbour of $r$ in $G$, and let $x$ be a node of $T$ with $r,s\in B_x$.  

Add a new node $\alpha$ to $T$ only adjacent to $x$, and let $B_\alpha:=\{r\}$. Consider $T$ to be rooted at $\alpha$. Note that $|B_x\setminus B_\alpha|\geq 1$. Consider each subtree $X$ of $T$ to be rooted at the vertex in $X$ closest to $\alpha$. If $x$ is the parent of $y$ in $T$, then denote the edge $xy$ by $\overrightarrow{xy}$.

Partition $V(T)$ into subtrees $T_1,T_2,\dots$ such that:
\begin{itemize}
    \item for each $i\in\NN$ and for each edge $\overrightarrow{xy}\in E(T_i)$, we have $B_y\subseteq B_x$, and 
    \item for each edge $\overrightarrow{xy}\in E(T)$, if $x\in V(T_i)$ and $y\in V(T_j)$ with $i\neq j$, then $B_y\not\subseteq B_x$. 
\end{itemize}
Note that $V(T_i)=\{\alpha\}$ for some $i\in\NN$. 
The first property implies that if $r_i$ is the root of $T_i$, then $B_x\subseteq B_{r_i}$ for each $x\in V(T_i)$. 

Let $T'$ be the tree obtained from $T$ as follows. For each $i\in\NN$, contract each $T_i$ into its root vertex $r_i$. Observe that $(B_x:x\in V(T'))$ is a tree-decomposition of $G$ with width at most $k$, and for every edge $\overrightarrow{xy}\in E(T')$, we have $B_y\not\subseteq B_x$ (implying $|B_y\setminus B_x|\geq 1$). 

Let $T''$ be the tree obtained from $T'$ as follows. For each edge $\overrightarrow{xy}\in E(T')$, if $|B_y\setminus B_x|=\ell\geq 2$ and $B_y\setminus B_x=\{w_1,\dots,w_\ell\}$, then replace the edge $xy$ in $T'$ by a path $x,z_1,z_2,\dots,z_{\ell-1},y$ in $T''$ (where $z_1,z_2,\dots,z_{\ell-1}$ are new vertices), and 
let $B_{z_i}:= (B_x\cap B_y)\cup\{w_1,w_2,\dots,w_i\}$ for each $i\in [\ell-1]$. So $B_{z_i}\subseteq B_y$ and thus $|B_{z_i}|\leq k+1$. 
In $T''$ the parent of $z_1$ is $x$, and $B_{z_1}\setminus B_x=\{w_1\}$. For $i\in[2,\ell-1]$, in $T''$ the parent of $z_i$ is $z_{i-1}$, and $B_{z_i}\setminus B_{z_{i-1}}=\{w_i\}$. In $T''$ the parent of $y$ is $z_{\ell-1}$, and $B_y\setminus B_{z_{\ell-1}}=\{w_\ell\}$. 

Now $(B_x:x\in V(T''))$ is a tree-decomposition of $G$ with width at most $k$, and for each edge $\overrightarrow{xy}\in E(T'')$, we have $|B_y\setminus B_x|=1$. Thus there is a bijection from $V(G)$ to $V(T'')$, where each vertex $w$ of $G$ is mapped to the node $x$ in $T''$ closest to $r$ with $w\in B_x$. Rename each vertex of $T''$ by the corresponding vertex of $G$. So $\alpha$ is renamed $r$, and $v\in B_v$ for each $v\in V(T'')=V(G)$.

For every edge $vw$ of $G$, since $v$ and $w$ appear in a common bag, $v$ is an ancestor of $w$ or $w$ is an ancestor of $v$ in $T''$. 

We now $(k+1)$-colour the vertices $w$ of $T''$ in order of their distance from $r$ (breaking ties arbitrarily). First, let $c(r):=1$. When we come to colour vertex $w$, if $v$ is the parent of $w$, then $B_w\setminus B_v=\{w\}$, implying all the vertices in $B_w\setminus \{w\}$ are already coloured. Let $c(w)$ be an element of $[k+1]$ not assigned to any vertex in $B_w\setminus \{w\}$. This is possible since $|B_w|\leq k+1$. We now prove that distinct vertices in the same bag are assigned distinct colours. This is true for the root bag since it only contains one vertex. Consider two vertices $v$ and $w$ in the same bag. Let $x$ be the node of $T''$ closest to the root, such that $v,w\in B_x$. Then $v=x$ or $w=x$. Without loss of generality, $w=x$, implying $v\in B_w\setminus\{w\}$. By construction, $c(w)\neq c(v)$. Hence distinct vertices in the same bag are assigned distinct colours. 
\end{proof}

%%%%%%%%%%%%%%%%%%%%%%%%%%%%%%%%

\subsection{Balanced Separations}
\label{SeparatorsTreewidth}

A \defn{separation} in a graph $G$ is a pair $(G_1,G_2)$ of subgraphs of $G$ such that $G=G_1\cup G_2$, $E(G_1)\cap E(G_2)=\emptyset$, $V(G_1)\setminus V(G_2)\neq\emptyset$, and $V(G_2)\setminus V(G_1)\neq\emptyset$. The \defn{order} of $(G_1,G_2)$ is $|V(G_1) \cap V(G_2)|$. A \hdefn{$k$}{separation} is a separation of order $k$. A separation $(G_1,G_2)$ is \defn{balanced} if $|V(G_1)\setminus V(G_2)| \leq \frac{2}{3} |V(G)|$ and $|V(G_2)\setminus V(G_1)| \leq \frac{2}{3} |V(G)|$. 

A graph class $\GG$  admits \defn{strongly sublinear separators} if there exists $ c \in\mathbb{R}^+$ and $\beta\in[0,1)$ such that for every graph $G\in\GG$, every subgraph $H$ of $G$ has a balanced separation of order at most $c|V(H)|^\beta$. For example, \citet{LT79} proved that the class of planar graphs admits strongly sublinear separators (with $\beta=\frac12$). More generally, \citet{AST90} proved that every proper minor-closed class admits strongly sublinear separators (again with $\beta=\frac12$). 

\citet[(2.6)]{RS-II} established the following connection between treewidth and balanced separations: Every graph $G$ has a balanced separation of order at most $\tw(G)+1$. \citet{DN19} proved the following converse: If every finite subgraph of a graph $G$ has a balanced separation of order at most $s$, then $\tw(G) \leq 15s$.

\citet{DHJLW21} proved the following strongly sublinear bound on the treewidth of graph products, which is used in \cref{TreewidthPathStructure}.

\begin{lem}[{\protect\citep[Lemma~18]{DHJLW21}}]
\label{ProductTreewidth}
Let $G$ be an $n$-vertex subgraph of $H \boxtimes P$ for some graph $H$ and path $P$. Then $\tw(G) \leq 2  \sqrt{ (\tw(H)+1) n } -  1$, 
and $G$ has a balanced separation of order at most $2  \sqrt{ (\tw(H)+1) n }$.
\end{lem}

For $k\in\NN$, a tree-decomposition $(B_x:x\in V(T))$ of a graph $G$ has \defn{adhesion} $k$ if $|B_x\cap B_y|\leq k$ for every edge $xy\in E(T)$. Given a tree-decomposition $(B_x:x\in V(T))$ of a graph $G$, the \defn{torso} of $x$ is the graph obtained from $G[B_x]$ by adding an edge $vw$ whenever $v,w\in B_x\cap B_y$ for some edge $xy\in E(T)$ incident to $x$. If $\Gamma$ is a class of graphs, then a tree-decomposition is said to be \defn{over} $\Gamma$ if every torso is in $\Gamma$. If $\Gamma$ is a graph, then a tree-decomposition is said to be \defn{over $\Gamma$} if every torso is contained in $\Gamma$. For a graph $U$, let $\DD(U)$ be the class of graphs that have a tree-decomposition over $U$. For $k\in\NN$, let $\DD_k(U)$ be the class of graphs that have a tree-decomposition over $U$ with adhesion $k$. 

The following simple lemma will be used in the proof of \cref{MinorUniversal}.

\begin{lem}
\label{TreewidthOverTreeDecomposition}
For every graph class $\Gamma$ and $n\in\NN$, the maximum treewidth of a graph in $\DD(\Gamma)$ with at most $n$ vertices equals the maximum treewidth of a graph in $\Gamma$ with at most $n$ vertices. 
\end{lem}

\begin{proof}
Let $f(n)$ be the maximum treewidth of a graph in $\Gamma$ with at most $n$ vertices. It suffices to prove that if $G$ is a graph in $\DD(\Gamma)$ with at most $n$ vertices, then $\tw(G)\leq f(n)$. Let $(B_x:x\in V(T))$ be a tree-decomposition of $G$ such that each torso is in $\Gamma$. For each node $x\in V(T)$, since $|B_x|\leq n$, the torso of $x$ has a tree-decomposition $(C^x_z:z\in V(T^x))$ of width at most $f(n)$. Let $T^*$ be the tree obtained from the disjoint union of $T^x$ taken over all $x\in V(T)$, where for each edge $xy\in T$, we add an edge between a node $\alpha\in V(T^x)$ and a node $\beta\in V(T^y)$ such that $B_x\cap B_y \subseteq C^x_\alpha\cap C^y_\beta$, which exist since $B_x\cap B_y$ is a clique in the torso of $x$ and in the torso of $y$. We obtain a tree-decomposition of $G$ with width at most $f(n)$. 
\end{proof}

%%%%%%%%%%%%%%%%%%%%%%%%%%%%%%%%%%%%%%%%%%%%%%%%%%%%%
\subsection{Generalised Colouring Numbers}
\label{GeneralisedColouringNumbers}

\citet{KY03} introduced the following definition. For a graph $G$, total order $\preceq$ of $V(G)$, vertex $v\in V(G)$, and  $r\in\NN$, let $\sreach_r(G,\preceq,v)$ be the set of vertices $w\in V(G)$ for which there is a path $v=w_0,w_1,\dots,w_{r'}=w$ of length $r'\in[0,r]$ such that $w\preceq v$ and $v\prec w_i$ for all $i\in[r'-1]$. For a graph $G$ and integer $r\in\NN$, the \hdefn{$r$}{colouring number} $\col_r(G)$ is the minimum integer $k$ such that there is a total order~$\preceq$ of $V(G)$ with $|\sreach_r(G,\preceq,v)|\leq k$ for every vertex $v$ of $G$. If no such $k$ exists then we say $\col_r(G)=\infty$. 

An attractive aspect of generalised colouring numbers is that they interpolate between degeneracy and treewidth \citep{KPRS16}. Indeed, it follows from the definition that $\col_1(G)$ equals the degeneracy of $G$ plus 1, implying $\chi(G)\leq \col_1(G)$. At the other extreme, \citet{GKRSS18} showed that $\col_r(G)\leq \tw(G)+1$ for all $r\in\NN$, and indeed $$\lim_{r\to\infty}\col_r(G)=\tw(G)+1.$$

Generalised colouring numbers provide upper bounds on several graph parameters of interest. For example, a graph colouring is \defn{acyclic} if the union of any two colour classes induces a forest; that is, every cycle is assigned at least three colours. The \defn{acyclic chromatic number} $\chi_\text{a}(G)$ of a graph $G$ is the minimum integer $k$ such that $G$ has an acyclic $k$-colouring. Acyclic colourings are qualitatively different from colourings, since every finite graph with bounded acyclic chromatic number has bounded average degree. \citet{KY03} proved that $\chi_\text{a}(G)\leq \col_2(G)$ for every graph $G$. Other examples include game chromatic number \citep{KT94,KY03}, Ramsey numbers \citep{CS93}, oriented chromatic number \citep{KSZ-JGT97}, arrangeability~\citep{CS93}, etc. 

Generalised colouring numbers are important also because they characterise bounded expansion classes \citep{Zhu09}, they characterise nowhere dense classes \citep{GKRSS18}, and have several algorithmic applications such as the constant-factor approximation algorithm for domination number by \citet{Dvorak13}, and the almost linear-time model-checking algorithm of \citet{GKS17}. 

A graph class $\GG$ has \defn{linear colouring numbers} if there is a constant $c$ such that $\col_r(G)\leq cr$  for every $G\in\GG$ and for every $r\in \NN$. An infinite graph $G$ has \defn{linear colouring numbers} if there is a constant $c$ such that $\col_r(G)\leq cr$ for every $r\in \NN$. 

For example, van~den~Heuvel~et~al.~\citep{HOQRS17} proved the following results (always for every $r\in\NN$): Every finite planar graph $G$ satisfies $\col_r(G) \leq 5r+1$. More generally,  every finite graph $G$ with Euler genus $g$ satisfies $\col_r(G) \leq (4g+5)r + 2g+1$. Even more generally, for $t\geq 4$, every finite $K_t$-minor-free graph $G$ satisfies $\col_r(G) \leq \binom{t-1}{2} (2r+1)$. The proofs of these results depend on the following lemma, which we use in the proof of \cref{KtMinorFreeNoSubdiv}.

\begin{lem}[{\protect\citep[Lemmas~3.2~and~3.4]{HOQRS17}}]
\label{PartitionGenColNum}
Let $G$ be a finite graph that has a connected partition $\PART=\{A_1,A_2,\dots,A_n\}$ such that for each part $A_i\in\PART$, there are at most $d$ neighbouring parts $A_j\in N_{G/\PART}(A_i)$ with $j<i$, and $V(A_i)$ is the union of the vertex-sets of $p$ geodesic paths in $G-(A_1\cup\dots\cup A_{i-1})$. Then $\col_r(G) \leq p(d+1)(2r+1)$.
\end{lem}

We need the following rooted variant of $r$-colouring number. For a graph $G$ and $r\in\NN$, the \defn{rooted $r$-colouring number} $\col^*_r(G)$ of $G$ is the minimum integer $k$ such that for every ordered clique $C$ of $G$ there is a total order $\preceq$ of $V(G)$ with $C$ at the start, such that $|\sreach_r(G,\preceq,v)| \leq k$ for every vertex $v$ of $G$. If no such $k$ exists then we say $\col^*_r(G)=\infty$. 
The next lemma is the motivation for this definition. 

\begin{lem}
\label{GenColourTreeDecomp}
For every graph $U$ and graph $G\in\DD(U)$ and $r\in\NN$, 
$$\col^*_r(G)\leq \col^*_r(U).$$
\end{lem}

\begin{proof}
If $G$ has an infinite clique, then 
$\col^*_r(G)= \col^*_r(U)=\infty$ and we are done. Now assume that $G$ has no infinite clique. By assumption, $G$ has a tree-decomposition $(B_x:x\in V(T))$ such that for each node $x\in V(T)$, the torso $G_x$ of $x$ is contained in $U$. Let $C$ be a clique of $G$, which is finite. Root $T$ at a node $s$ such that $C\subseteq B_s$. Such a node exists by \cref{TreeDecompositionClique}. For each vertex $v\in V(G)$, let $h(v)$ be the node of $T$ closest to $s$ and with $v\in B_{h(v)}$. Let $\preceq_T$ be any total order on $V(T)$, where $x\prec_T y$ whenever $y$ is a descendant of $x$. Since $G_x\subseteq U$, we have $\col^*_r(G_x)\leq\col^*_r(U)$. 
Let $\preceq_s$ be a total order on $V(G_s)$ with $C$ at the start, and $|\sreach_r(G_x,\preceq_s,v)| \leq \col^*_r(G_x) \leq \col^*_r(U)$ for every vertex $v\in V(G_s)$. We now define a total order $\preceq_x$ on $V(G_x)$ for each non-root node $x\in V(T)$. Consider each non-root node $x$ of $T$ in non-decreasing order of distance from $s$. Let $y$ be the parent of $x$. So we may assume that $\preceq_y$ is already defined. Since $G_x$ is contained in $U$ and $B_x\cap B_y$ is a clique of $G_x$, there is a total order $\preceq_x$ of $V(G_x)$ with $B_x\cap B_y$ at the start of $\preceq_x$ and ordered according to $\preceq_y$, such that $|\sreach_r(G_x,\preceq_x,v)| \leq \col^*_r(G_x) \leq \col^*_r(U)$ for every vertex $v$ of $G_x$. Finally, define a total order $\preceq$ on $V(G)$, where $v \preceq w$ if and only if $h(v) \prec_T h(w)$, or $h(v)=h(w)$ and $v\preceq_{h(v)} w$. Clearly $\preceq$ is a total order on $V(G)$.

Consider $w\in\sreach_r(G,\preceq,v)$ for some $v\in V(G)$. Thus $G$ contains a path $P=(v,w_1,\dots,w_k,w)$ of length at most~$r$ with $w\preceq v\preceq w_i$ for each $i\in[k]$. Let $x:=h(v)$. Since $v\preceq w_i$ we have $x \preceq_T h(w_i)$, and $h(w_i)$ is not an ancestor of $x$. Since $v,w_1,\dots,w_k$ is a path in $G$, we have $h(w_1),\dots,h(w_k)$ are all in the subtree of $T$ rooted at $x$. 
By the definition of tree-decomposition, $w$ and $w_k$ appear in a common bag. Thus $h(w)$ is an ancestor of $h(w_k)$ or vice versa. Since $w\preceq w_k$, we have $h(w) \preceq_T h(w_k)$. That is, $h(w)$ is an ancestor of $h(w_k)$, or $h(w)=h(w_k)$. Thus $w\in B_{h(w_k)}$. Since $w\preceq v$, we have $h(w) \preceq_T x$ and $x$ is not an ancestor of $h(w)$. Hence, $x$ is on the path from $h(w)$ to $h(w_k)$ in $T$. Since $w\in B_{h(w_k)}$, by the definition of tree-decomposition,  $w$ is in every bag in the path from $h(w)$ to $h(w_k)$ in~$T$. In particular, $w$ is in $B_x$. This says that both ends of $P$ are in $B_x$. If some vertex in $P$ is not in $B_x$, then for some child node $y$ of $x$, $P$ contains a subpath $w_i,\dots,w_j$, where $w_i,w_j$ are distinct vertices in $B_x\cap B_y$. Replace the subpath $w_i,\dots,w_j$ of $P$ by the edge $w_iw_j$, which is in the torso $G_x$. Do this whenever $P$ contains a vertex not in $B_x$. We obtain a path from $v$ to $w$ of length at most $r$ in $G_x$. Since $G_x$ is ordered by $\preceq_x$ in $\preceq$, we have $w\in \sreach_r(G_x,\preceq_x,v)$. Hence
$\sreach_r(G,\preceq,v) \subseteq  \sreach_r(G_x,\preceq_x,v)$ and
$|\sreach_r(G,\preceq,v)| \leq |\sreach_r(G_x,\preceq_x,v)| \leq \col^*_r(G_x) \leq \col^*_r(U)$. Therefore $\col^*_r(G)\leq \col^*_r(U)$.
\end{proof}

Note that \cref{GenColourTreeDecomp} with $U=K_{\tw(G)+1}$ shows 
that $\col^*_r(G) \leq \tw(G)+1$. This implies that $\col_r(G)\leq \tw(G)+1$, as proved by \citet{GKRSS18}. 

\begin{lem}
\label{MakeRootedColNum}
For every graph $G$ and $r\in\NN$, 
$$\col^*_r(G) \leq \col_r(G)+ \omega(G).$$ 
\end{lem}
 
\begin{proof}
Let $\preceq$ be a total order of $V(G)$ such that $|\sreach_r(G,\preceq,v)| \leq \col_r(G)$ for every vertex $v$ of $G$. Let $C$ be an ordered clique of $G$. Let $\preceq'$ be the total order of $V(G)$ obtained from $\preceq$ by placing $C$ at the start in the given order. For every vertex $v$ of $G$, we have $\sreach_r(G,\preceq',v) \subseteq \sreach_r(G,\preceq,v) \cup C$, which has cardinality at most $\col_r(G)+\omega(G)$.
\end{proof}

We now apply K\"onig's Lemma to show that $\col_r$ is extendable.

\begin{lem}
\label{colExtendable}
$\col_r$ is extendable for each $r\in\NN$.
\end{lem}

\begin{proof}
Let $G$ be a graph, such that for some $c\in\NN$ for every finite subgraph $H$ of $G$, we have $\col_r(H)\leq c$. Since $G$ is countable, we may assume that $V(G)=\{v_1,v_2,\dots\}$. Let $H_n:=G[\{v_1,\dots,v_n\}]$ for $n\in\NN$. For $n\in\NN$, let $V_n$ be the set of total orders $\preceq$ of $V(H_n)$ such that $|\sreach_r(H_n,\preceq,v)|\leq c$ for every $v\in V(H_n)$. By assumption, $V_n\neq\emptyset$. By construction, $V_n$ is finite. Let $Q$ be the graph with vertex-set $\bigcup_{n\in\NN} V_n$ where $\preceq_n$ in $V_n$ is adjacent to $\preceq_{n-1}$ in $V_{n-1}$ if $x\preceq_{n-1} y$ if and only if $x\preceq_n y$ for all vertices $x,y\in V(H_{n-1})$. So every vertex in $V_n$ has a neighbour in $V_{n-1}$. By \cref{Konig}, there exist $\preceq_1,\preceq_2,\dots$ where $\preceq_n$ is in $V_n$ for each $n\in \NN$, and $\preceq_n$ is adjacent to $\preceq_{n-1}$ for each $n\geq 2$. Define the relation $\preceq$ on $V(G)$, where $x\preceq y$ whenever $x\preceq_n y$ for some $n\in\NN$. By the definition of adjacency in $Q$, we have $x\preceq y$ if and only if $x\preceq_n y$ for all $n\in \NN$ with $x,y\in V(H_n)$. Call this property $(\star)$. 

We now show that $\preceq$ is a total order of $V(G)$. 
Say $x\preceq y$ and $y\preceq z$, where $x=v_i$ and $y=v_j$ and $z=v_k$. %
Let $n:=\max\{i,j,k\}$. So $x,y,z\in V(H_n)$. 
By property $(\star)$, we have $x\preceq_n y$ and $y\preceq_n z$.
By the transitivity of $\preceq_n$, we have $x \preceq_n z$. 
By property $(\star)$, we have $x \preceq z$.
Thus $\preceq$ is transitive. 
Say $x\preceq y$ and $y\preceq x$, where $x=v_i$ and $y=v_j$. 
Let $n:=\max\{i,j\}$. So $x,y\in V(H_n)$. 
By property $(\star)$, we have $x\preceq_n y$ and $y\preceq_n x$.
By the antisymmetry of $\preceq_n$, we have $x =y$. 
Hence $\preceq$ is antisymmetric. 
Consider vertices $x,y\in V(G)$, where $x=v_i$ and $y=v_j$. 
Let $n:=\max\{i,j\}$. So $x,y\in V(H_n)$. 
By the connexity of $\preceq_n$, we have $x\preceq_n y$ or $y\preceq_n x$.
By property $(\star)$, we have $x\preceq y$ or $y\preceq x$.
Hence $\preceq$ is connex. 
Therefore $\preceq$ is a total order of $V(G)$.

\note{Revised here in response to referee's comment.}

It remains to show that $|\sreach_r(G,\preceq,v)| \leq c$ for each vertex $v\in V(G)$. For the sake of contradiction, suppose that $|\sreach_r(G,\preceq,v)|>c$ for some $v\in V(G)$. 
Let $W$ be a set of $c+1$ vertices in $\sreach_r(G,\preceq,v)$. 
For each $w\in W$ let $P_w$ be a path of length at most $r$ witnessing that $w\in \sreach_r(G,\preceq,v)$. Let $n$ be minimum such that $\bigcup_{w\in W} V(P_w) \subseteq \{v_1,\dots,v_n\}$, which is well-defined since $\bigcup_{w\in W} V(P_w)$ is finite. Thus $\sreach_r(G,\preceq,v)= \sreach_r(G,\preceq_n,v)$, and $|\sreach_r(G,\preceq,v)|= |\sreach_r(G,\preceq_n,v)| \leq c$, which is the desired contradiction. 
\end{proof}
 
\begin{lem}
\label{GenColourProduct}
For every graph $H$ and path $P$, and for every $r\in\NN$ and $a\in\NN_0$, 
\begin{equation*}
\col_r(( H\boxtimes P) + K_a) \leq (\tw(H)+1)(2r+1) + a.
\end{equation*}
\end{lem}

\begin{proof}
A result of van den Heuvel and Wood~\citep[Lemma~30, arXiv~version]{vdHW18} implies $\col_r(H\boxtimes P) \leq (\tw(H)+1)(2r+1)$ for finite $H$ and $P$. 
By \cref{colExtendable}, the result holds for infinite graphs. The claimed result follows by placing a copy of $K_a$ at the start of the corresponding vertex ordering of $H\boxtimes P$. 
\end{proof}

The next lemma will be used in \cref{ExcludedMinors} to show that our graph that contains every $X$-minor-free graph has linear colouring numbers. 

\begin{lem}
\label{GenColourProductTreeDecomp}
Fix $r\in\NN$ and $a\in\NN_0$. For every graph $H$ and path $P$, and for every graph $G$ in $\DD( (H\boxtimes P)+K_a )$, 
$$\col^*_r(G) \leq  (\tw(H)+1)(2r+3)+2a.$$
\end{lem}

\begin{proof}
Note that $\omega((H\boxtimes P)+K_a) = a + 2\omega(H) \leq a + 2(\tw(H)+1)$. By \cref{MakeRootedColNum,GenColourProduct}, $\col^*_r( (H\boxtimes P)+K_a) \leq (\tw(H)+1)(2r+3)+2a$. The claim then follows from \cref{GenColourTreeDecomp}.
\end{proof}

%%%%%%%%%%%%%
\subsection{Bounded Expansion}
\label{Expansion}

For $r\in \NN$, a graph $H$ is an \hdefn{$r$}{shallow minor} of a graph $G$ if there is a model $(X_v)_{v\in V(H)}$ of $H$ in $G$ such that each subgraph $X_v$ has radius at most $r$. \citet{Sparsity} introduced the following definition for finite graphs. \note{Revised here in response to referee's comment.} For any (possibly infinite) graph $G$, let 
$$\nabla_r(G):=\sup \left\{\frac{2|E(H)|}{|V(H)|} : \text{$H$ is a finite $r$-shallow minor of $G$ with $V(H)\neq\emptyset$}\right\}.$$
We only consider finite $H$ in this definition, since average degree is not well-defined for infinite graphs.

A graph class $\GG$ has \defn{bounded expansion} with \defn{bounding function} $f$ if $\nabla_r(G)\leq f(r)$ for each $G\in\GG$ and $r\in\NN$. We say $\GG$ has \defn{linear expansion} if, for some constant $c$, for all $r\in\NN$, every graph $G\in\GG$ satisfies $\nabla_r(G) \leq cr$. Similarly, $\GG$ has \defn{polynomial expansion} if, for some constant $c$, for all $r\in\NN$, every graph $G\in\GG$ satisfies $\nabla_r(G) \leq cr^c$. For example, when $f(r)$ is a constant, $\GG$ is contained in a proper minor-closed class. As $f(r)$ is allowed to grow with $r$ we obtain larger and larger graph classes.

An infinite graph $G$ has \defn{linear expansion} if, for some constant $c$, for all $r\in\NN$, every subgraph $H$ of $G$ satisfies $\nabla_r(H) \leq c\,\rad(H)$.

\citet{DN16} noted that a result of \citet{PRS94} implies that graph classes with polynomial expansion admit strongly sublinear separators. \citet{DN16} proved the converse:  A hereditary class of graphs admits strongly sublinear separators if and only if it has  polynomial expansion. See \citep{Dvorak16,Dvorak18,ER18} for more results on this theme. 

Sergey Norin observed the following connection between colouring numbers and expansion in finite graphs; see~\citep{ER18}. Since every finite $r$-shallow minor of a graph $G$ is a minor of a finite subgraph of $G$, the result for infinite graphs immediately follows. 

\begin{lem}
\label{SergeyCorollary}
For every graph $G$ and $r \in \NN$, 
$$\nabla_r(G) \leq 2 \col_{4r+1}(G).$$
\end{lem}

\cref{SergeyCorollary} says that an infinite graph that has linear colouring numbers also has linear expansion. 

%%%%%%%%%%%%%%%%%%%%%%%%%%%%%%%%%%%%%%%%%%%%%
\subsection{Planar Triangulations}
\label{Triangulations}

This subsection introduces some elementary notions regarding planar graphs that form a foundation for the results in \cref{ForcingMinor}.

A \defn{plane graph} is a graph embedded in the plane (that is, drawn without crossings).  See \cite{Mo88} for a more formal treatment of embeddings of possibly infinite graphs on surfaces. A \defn{face} of a plane graph $G$ is an arcwise  connected component of the set of points in the plane not in the embedding of $G$. A plane graph $G$ is \defn{outerplane} if there is a face $f$ of $G$ such that every vertex of $G$ is on the boundary of $f$. A graph is \defn{outerplanar} if it is isomorphic to an outerplane graph. 

 A \defn{plane triangulation} is a connected locally Hamiltonian plane graph $G$ such that for each vertex $v\in V(G)$, there is a Hamiltonian cycle $C_v$ of $G[N_G(v)]$ where $v$ is the only vertex of $G$ inside $C_v$ or $v$ is the only vertex of $G$ outside $C_v$. The subgraph of $G$ consisting of $v$, $C_v$, and the edges from $v$ to $C_v$ form a \defn{wheel centred at $v$}. A \defn{planar triangulation} is a planar graph that can be embedded in the plane as a plane triangulation.

\begin{lem}
\label{LocallyHamiltonianTriangulation}
Every connected locally Hamiltonian plane graph $G$ is a plane triangulation.
\end{lem}

\begin{proof}
For each vertex $v\in V(G)$, let $C_v$ be a Hamiltonian cycle through $N_G(v)$. Let $v\in V(G)$ and let $pq\in E(C_v)$. So $vpq$ is a triangle in $G$. Since $C_v-p-q$ is connected, every vertex in $C_v-p-q$ is inside $vpq$ or every vertex in $C_v-p-q$ is outside $vpq$. In both cases, $vp$ and $vq$ are consecutive in the cyclic ordering of edges incident to $v$ defined by the embedding. Thus the cyclic ordering of $N_G(v)$ defined by $C_v$ coincides with the cyclic ordering of $N_G(v)$ defined by the embedding. 
Without loss of generality, $v$ is inside $C_v$. Suppose for a contradiction that some other vertex $x$ is inside $C_v$. So $x$ is inside the triangle $vpq$ for some $pq\in E(C_v)$. Choose such a vertex $x$ to minimise the distance in $G$ from $x$ to $\{v,p,q\}$. Since $vp$ and $vq$ are consecutive in the cyclic ordering of edges incident to $v$, $v$ has no neighbour inside $vpq$. In particular, $x\not\in N_G(v)$. Since $G$ is connected and by the choice of $x$, without loss of generality, $px\in E(G)$. Let $y$ be the neighbour of $p$ inside $vpq$, such that $pv$ and $py$ are consecutive in the cyclic ordering of edges incident to $p$ (which is well-defined since $x$ is a candidate). Since $C_p$ coincides with the cyclic ordering of neighbours of $p$, we have $vy\in E(G)$, which contradicts the fact that $v$ has no neighbour in $vpq$. Hence, $v$ is the only vertex inside $C_v$. And $G$ is a plane triangulation. 
\end{proof}

By definition, a planar triangulation is locally Hamiltonian and is thus 3-connected  by \cref{LocallyHam3Conn}. This can be strengthened as follows. 

\begin{lem}
\label{Extended3Connected}
Every finite subgraph of a planar triangulation $G$ can be extended to a finite 3-connected subgraph of $G$.
\end{lem}

\begin{proof}
Let $H_0$ be a finite subgraph of $G$. Add shortest paths between the components of $H_0$ to obtain a finite connected subgraph $H_1$ of $G$. Let $H$ be obtained from $H_1$ by adding the wheel in $G$ centred at each vertex in $H_1$. Then $H$ is a finite 3-connected subgraph of $G$.
\end{proof}

This lemma is in sharp contrast to the fact that there exist planar graphs of arbitrarily large (finite) connectivity such that no finite subgraph is 3-connected. For example, let $G_k$ be a planar graph containing cycles $C_0,C_1, \ldots$ drawn as concentric circles in the plane, such that for $i\in\NN_0$ each vertex in $C_i$ is adjacent to at least $k$ vertices in $C_{i+1}$, and for $i\in\NN$ each vertex in $C_i$ is adjacent to at most one vertex in $C_{i-1}$ and some vertex in $C_i$ is adjacent to no vertex in $C_{i-1}$. Then $G_k$ is $(k+2)$-connected but contains no 3-connected finite subgraph.

%Every 1-ended planar triangulation has the following extension property: 
%
% \begin{lem}
% \label{ExtendToNearTriang}
% Every finite subgraph of a 1-ended planar triangulation $G$ can be extended to a finite near-triangulation in $G$. 
% \end{lem}
%
% \begin{proof}
% Let $H_0$ be a finite subgraph of $G$. First extend $H_0$ to a finite 3-connected subgraph $H$ using \cref{Extended3Connected}. Since $G$ is 1-ended, precisely one face of $H$ contains infinitely many vertices of $G$. Add all the other vertices of $G$ (and their incident edges) to $H$ to obtain a finite near-triangulation in $G$.
% \end{proof}

\note{The following claim was added in response to a comment from the referee.} 

\begin{lem}
\label{TriangularFaces}
Each face $F$ of a plane triangulation $G$ either has no vertices of $G$ in the boundary of $F$, or $F$ is bounded by a triangle of $G$.
\end{lem}

\begin{proof}
Let $p$ be a point in $F$. We may assume there is a vertex $v$ on the boundary of $F$. The point $p$ is also in a face $F'$ of the wheel centred at $v$, and clearly $F'$ is bounded by a triangle. Also, $F \subseteq F'$. By the definition of plane triangulation, $F=F'$.  
\end{proof}

\begin{lem}
\label{4ConnPlaneTriangProperSubgraph}
No 4-connected plane triangulation contains a plane triangulation as a proper subgraph.
\end{lem}

\begin{proof}
Suppose for contradiction that some 4-connected plane triangulation $T$ contains a  plane triangulation $T'$ as a proper subgraph. First suppose that $V(T')\subsetneq V(T)$. So $T$ has a vertex $v$ that is a point in a face $F$ of $T'$. Let $S$ be the set of vertices of $T'$ on the boundary of $F$. By \cref{TriangularFaces}, $|S|\leq 3$. Since $T'$ is a plane triangulation, there is a vertex $w\in V(T')\setminus S \subseteq V(T)\setminus S$. So $S$ separates $v$ and $w$ in $T$, which contradicts the 4-connectivity of $T$. Now assume that $V(T')= V(T)$. Since $T'$ is a proper subgraph of $T$, there is an edge $vw\in E(T)\setminus E(T')$. Since $T'$ is a plane triangulation, there is a  Hamiltonian cycle $C_v$ of $T'[N_{T'}(v)]$ such that $v$ is the only vertex of $G'$ inside $C_v$ or $v$ is the only vertex of $T'$ outside $C_v$. Since $vw\not\in E(T')$, in both cases, $C_v$ separates $v$ and $w$ in the embedding of $T'$ and thus in the embedding of $T$, which is a contradiction since $vw\in E(T)$. 
\end{proof}

\begin{lem}
\label{4ConnTriangProperSubgraph}
No 4-connected planar triangulation contains a planar triangulation as a proper subgraph.
\end{lem}

\begin{proof}
Let $G$ be a 4-connected planar triangulation. Let $H$ be a proper subgraph of $G$. Fix a plane triangulation embedding of $G$. This determines a plane embedding of $H$. By \cref{4ConnPlaneTriangProperSubgraph}, $H$ is not a plane triangulation. By the uniqueness of embeddings of planar triangulations~\citep{Imrich75,RT97}, every plane embedding of a planar triangulation is a plane triangulation. So $H$ is not a planar triangulation. 
\end{proof}

\note{The next result was previously a footnote. For precision we made it a lemma.}

\begin{lem}
\label{MakeUnbounded}
Every infinite planar graph $G$ has an embedding in the plane so that $V(G)$ is unbounded (that is, no disc contains $V(G)$). 
\end{lem}

\begin{proof}
Consider a plane embedding of $G$. We may assume that $V(G)$ is bounded. So $G$ has an accumulation point $p$. If $p$ is a vertex or is a point on an edge of $G$, then $G$ can be redrawn vertex-by-vertex so that $p$ is not used and each vertex in the new embedding is close to the corresponding vertex in the old embedding, while maintaining the vertex-accumulation points. So we may assume that $p$ is not a vertex or a point on an edge of $G$. Consider $G$ to be drawn on the sphere with $p$ at the north pole. Deleting $p$ from the sphere results in $G$ being embedded in the plane. A sequence of vertices converging to $p$ on the sphere now becomes a sequence of vertices tending to infinity in the plane. So the embedding is unbounded. 
\end{proof}

%%%%%%%%%%%%%%%%%%%%%%%%%%%%%%%%%%%%%%%%%%%%%%%%%%%%%%%%%%%%%%%
\section{Forcing a Minor or Subdivision}
\label{ForcingMinor}

\note{Small details have been added throughout this section.}

This section proves our first main result, \cref{InfiniteCliqueMinor}, which says that if a (countable) graph $U$  contains every planar graph, then (a) the complete graph $K_{\aleph_0}$ is a minor of $U$, and (b) $U$ contains a subdivision of $K_t$ for every $t\in\NN$. The proof is split across three subsections. In \cref{Limits} we introduce the notion of a `limit', which may be of independent interest, and we prove the `Limit Lemma' (\cref{LimitLemma}), which shows that any graph that contains uncountably many planar graphs of a certain type contains a planar triangulation along with an infinite number of `jump' paths. Then \cref{Routing} presents a number of lemmas about routing paths in graphs obtained from planar triangulations by adding  jumps. All of these results are then combined in \cref{ProofInfiniteCliqueMinor}, where the proof of \cref{InfiniteCliqueMinor} is completed. In fact, we prove significant strengthenings of both parts of \cref{InfiniteCliqueMinor}. Finally, in \cref{ExcludingSubdivision} we show that numerous graph classes do \emph{not} have a universal element, including chordal graphs containing no $K_{\aleph_0}$, graphs with no $K_{\aleph_0}$ minor (which was proved in \cite{DHV85}), and graphs containing no $K_{\aleph_0}$ subdivision (which was left open in \cite{DHV85}).

%%%%%%%%%%%%%%%%%%%%%%%%%%%%%%
\subsection{Limit Lemma}
\label{Limits}

Let $G$ be a (countable) graph. Let $\GG$ be an uncountable set of infinite, locally finite, connected subgraphs of $G$. A \hdefn{$\GG$}{limit} in $G$ is any subgraph of $G$ that can be obtained as follows.  Let $v_0$ be any vertex in $G$ that is contained in uncountably many distinct subgraphs in $\GG$ (which exists since $\GG$ is uncountable and $G$ is countable). Denote this subset of $\GG$ by $\GG_0$. Define an equivalence relation on $\GG_0$, where two graphs in $\GG_0$ are equivalent if they contain the same set of edges incident to $v_0$. By the choice of $v_0$ and since every graph in $\GG$ is connected, this set of edges is nonempty. The number of equivalence classes is countable since every graph in $\GG$ is locally finite. Since $\GG_0$ is uncountable, some equivalence class is uncountable. Call this equivalence class $\GG_1$. Let $v_1,v_2, \ldots$ be the remaining vertices of $G$. We now define a sequence of uncountable families $\GG_0\supseteq \GG_1\supseteq \GG_2\supseteq\dots$. Suppose that $\GG_n$ is defined for some $n\in\NN$. Define an equivalence relation on $\GG_n$, where two graphs in $\GG_n$ are equivalent if they contain the same set of edges incident to $v_n$ (which may be empty). Again, the number of equivalence classes is countable since every graph in $\GG$ is locally finite. Let $\GG_{n+1}$ be an uncountable equivalence class. Since every graph in $\GG$ is connected, $v_n$ is contained in all or none of the graphs in $\GG_{n+1}$. Let $H$ be the graph with vertex-set $V(G)$ such that for every $n\ge0$, the set of edges incident with $v_n$ in $H$ is the set of edges incident with $v_n$ in each graph in $\GG_{n+1}$. We call $H$ a \DefNoIndex{$\GG$-limit} in $G$. 

By construction of $\GG$-limits, each finite induced subgraph of $H$ is an induced subgraph of every graph in $\GG_n$ for some sufficiently large $n$. This implies that $H$ inherits many properties of $\GG$ (in the spirit of extendable properties introduced in \cref{Definitions}). For example, by \cref{PlanarityExtendable}, if every graph in $\GG$ is planar, then $H$ is planar. Every degree of a vertex in $H$ is also the degree of some vertex of a graph in $\GG$. However, there is one important property that need not be preserved, namely connectedness. To see this, let $G$ be obtained from the infinite ladder $K_2\CartProd \PPP$ by replacing each edge by two internally disjoint paths of length 2, and let $\GG$ be the set of 2-way infinite paths of $G$. Then the $\GG$-limits in $G$ are precisely those spanning subgraphs that are either a 2-way infinite path plus infinitely many isolated vertices, or the union of two disjoint 2-way infinite paths containing no rungs of the ladder. Therefore we shall focus on a connected component of a limit. Some components may consist of isolated vertices. It is easy to see that every finite component of a limit graph (if any) is trivial (that is, an isolated vertex). But, the component containing $v_0$ does contain edges and is therefore infinite. 

%Is every non-trivial connected component of a $\mathcal{F}'$-limit isomorphic to a graph in $\mathcal{F'}$? No. e.g. $G$ is doubled 2-way infinite path, and $\FF'$  is the 1-way infinite path, then the limits are the 2-way infinite paths in $G$, which are not in $\FF'$ }

A locally finite graph $G$ is \hdefn{$k$}{ended} if $k$ is the maximum integer such that $G-S$ has $k$ infinite components for some finite set $S\subseteq V(G)$. The property of being 1-ended need not be preserved under taking limits. For example, if $G$ is the 2-way infinite path with all edges doubled and subdivided 
\note{Added subdivided edges to address referee comment that we only consider simple graphs.} and $\GG$ is  
the set of all 1-way infinite paths in $G$, then the $\GG$-limits in~$G$ are all 1-way infinite paths (together with isolated vertices) and all 2-way infinite paths, and the latter are 2-ended. Therefore we introduce the following stronger properties. 

A plane graph $G$ is \defn{isoperimetric} if every cycle $C$ in $G$ has finitely many vertices in the closed disc bounded by $C$.  For a function $f:\NN\to\NN$, a plane graph $G$ is \hdefn{$f$}{isoperimetric} if for every cycle $C$ of length $n$ there are at most $f(n)$ vertices of $G$ in the closed disc bounded by $C$. For example, the infinite grid $\PPP\CartProd \PPP$ is $f$-isoperimetric for some function $f$ with $f(n) \in  O(n^2)$; see \citep{Thomassen17}. Here is another example that will be useful later. A \defn{near-triangulation} is a 2-connected finite plane graph, in which each face, except possibly one, is bounded by a 3-cycle. The exceptional face is bounded by the \defn{outer cycle}. 

\begin{lem}
\label{IsoperimetricMinDegree}
Every isoperimetric plane triangulation $G$ with minimum degree at least 7 is $f$-isoperimetric with $f(n) := 3n$. 
\end{lem}

\begin{proof}
Let $C$ be a cycle in $G$ with $n$ vertices. Let $Q$ be the subgraph of $G$ induced by $C$ and the vertices in the interior of $C$. Since $G$ is isoperimetric, $Q$ is finite. In particular, $Q$ is a near-triangulation with outer cycle $C$. Say $Q$ has $n'$ internal vertices, each of which has degree at least 7. Let $Q'$ be the planar triangulation obtained from $Q$ by adding one vertex adjacent to all the vertices on the unbounded face of $Q$. So $Q'$ has $n+n'+1$ vertices and $3(n+n'+1)-6$ edges. Now 
$$6(n+n'+1)-12 \geq 2|E(Q')| = \sum_{v\in V(Q')} \deg(v) \geq n + 3n + 7n'.$$ 
So $n'< 2n$ and $|V(Q)|=n+n'<3n$. 
Hence $G$ is $f$-isoperimetric.
\end{proof}

A class $\GG$ of plane graphs is \DefNoIndex{$f$-isoperimetric} if every graph in $\GG$ is $f$-isoperimetric. A planar triangulation $G$ is \DefNoIndex{$f$-isoperimetric} if $G$ has an $f$-isoperimetric embedding in the plane. 

Note that every $f$-isoperimetric infinite planar triangulation $G$ is 1-ended. To see this, suppose for the sake of contradiction that $G-S$ has at least two infinite components for some finite $S\subseteq V(G)$. By \cref{Extended3Connected}, $G$ contains a finite 3-connected subgraph $H$ containing $S$. We may assume that $H$ is an induced subgraph. Since $H$ is induced, and $G$ is a planar triangulation (and hence every edge is contained in two triangles each bounding a face), it follows that distinct components of $G-V(H)$ belong to distinct faces of $H$. Since $G-V(H)$ has at least two infinite components, at least two faces of $H$ contain infinitely many vertices of $G$.  One of these faces is the interior of a cycle $C$ of length, say $n$, in $H$. But this contradicts that $G$ has at most $f(n)$ vertices inside $C$.

% \begin{lem}[Limit Lemma with Some Fixed Vertices]
% \label{LimitLemma}
% Fix $n\in\NN$ and an arbitrary function $f:\NN\to\NN$. Let $\FF$ be any (possibly finite) class of 4-connected $f$-isoperimetric planar triangulations, such that each $F\in\FF$ has specified vertices $y^F_1,\dots,y^F_n$, and for all $F,F'\in\FF$ we have $y^{F}_iy^{F}_j\in E(F)$ if and only if $y^{F'}_iy^{F'}_j\in E(F')$. Let $G$ be a countable graph and let $x_1,\dots,x_n$ be vertices in $G$, such that there is an uncountable set $\GG$ of subgraphs of $G$, such that for each $G'\in\GG$, there is some $F\in \FF$ and there is an isomorphism from $G'$ to $F$ that maps $x_i$ to  $y^F_i$ for each $i\in[n]$. Let $L$ be a non-trivial connected component of an $\FF$-limit in $G$ with respect to $\GG$. Then $L$ is a 4-connected $f$-isoperimetric planar triangulation with infinitely many pairwise disjoint jumps in $G$. Moreover, every finite induced subgraph of $L$ is an induced subgraph of some graph in $\FF$, and $x_1,\dots,x_n\in V(L)$ and for all  $F\in\FF$ we have $x_ix_j\in E(L)$ if and only if $y^{F}_iy^{F}_j\in E(F)$. 
% \end{lem}

Let $H$ be a subgraph of a graph $G$. A \defn{jump} of $H$ in $G$ is a path $P$ in $G$  such that the ends of $P$ are non-adjacent distinct vertices in $H$ and no internal vertex of $P$ is in $H$. (A jump can be a single edge.)

\begin{lem}[Limit Lemma]
\label{LimitLemma}
Fix an arbitrary function $f:\NN\to\NN$. Let $G$ be a countably infinite graph. Let $\GG$ be an uncountable set of subgraphs of $G$, each of which is a 4-connected $f$-isoperimetric planar triangulation. Let $L$ be any non-trivial connected component of some $\GG$-limit in $G$.  Then $L$ is a 4-connected $f$-isoperimetric planar triangulation, and there are infinitely many pairwise disjoint jumps of $L$ in $G$. 
\end{lem}

\begin{proof}
Let $L_0$ be a $\GG$-limit in $G$ defined with respect to a sequence of vertices $v_0,v_1,\dots$ and a sequence of families $\GG_0 \supseteq \GG_1\supseteq \dots$ such that $L$ is a non-trivial connected component of~$L_0$. By construction, every finite induced subgraph of $L$ is an induced subgraph of uncountably many graphs in $\GG$. Since every graph in $\GG$ is planar, every finite subgraph of $L$ is planar, and $L$ is planar by \cref{PlanarityExtendable}. By the construction of a $\GG$-limit, for every vertex $v$ of $L$, for some $G'\in\GG$, the subgraph of $L$ induced by $N_L[v]$ is equal to the subgraph of $G'$ induced by $N_{G'}[v]$. So $L$ is locally Hamiltonian. 

By \cref{MakeUnbounded}, there is a plane embedding of $L$ with $V(L)$ unbounded. By \cref{LocallyHamiltonianTriangulation}, $L$ is a plane triangulation. We now show that $L$ is $f$-isoperimetric. Suppose that $C$ is a cycle of length $k$ in $L$, and $L$ has more than $f(k)$ vertices inside $C$. Let $C'$ consist of $C$ along with $f(k)+1$ vertices inside $C$ and also $f(k)+1$ vertices outside $C$ (which exist since $L$ is drawn so that the vertex-set is unbounded). By \cref{Extended3Connected}, $C'$ is contained in a finite 3-connected subgraph $C''$ of $L$, which is a subgraph of some planar triangulation $T$ in $\GG$. 
Since $C''$ is 3-connected, it has a unique embedding in the plane up to the choice of the unbounded face. 
Regardless of the choice of the unbounded face of $C''$,  every plane embedding of $T$ that respects $C''$ is a plane triangulation 
with more than $f(k)$ vertices inside $C$, which contradicts that $T$ is $f$-isoperimetric. Hence $L$ is $f$-isoperimetric. 
Since every graph in $\GG$ is 4-connected, we may assume that $f(3)=3$. Since $L$ is locally Hamiltonian, every separating set with three vertices in $L$ induces a triangle. Since $L$ is $f$-isoperimetric and $f(3)=3$, every triangle in $L$ is a face, implying that $L$ is 4-connected. 
\note{The referee made a comment about using the connectedness of the things inside any face. Actually  we only use the uniqueness of the embedding of $H$. We have added a remark to emphasize this.}

It remains to prove that $L$ has infinitely many pairwise disjoint jumps in $G$.  It suffices to show that for any integer $p\geq 0$ for all  pairwise disjoint jumps $J_1, \dots ,J_p$ of $L$ in $G$, there exists $J_{p+1}$ such that $J_1, \dots, J_{p+1}$ are pairwise disjoint jumps of $L$ in $G$. Let $H$ be the subgraph of $L$ induced by the ends of $J_1,\dots,J_p$. 
Let $W$ be the set of internal vertices in $J_1,\dots,J_p$. So $W\subseteq V(G)\setminus V(L)$. 

Define the following finite subgraph $H'$ of $G$ disjoint from $L$. Consider each component $T'$ of $L_0$ that intersects $W$. If 
$|V(T')|=1$, then add this vertex as an isolated vertex to $H'$.  Otherwise, $T'$ is a planar triangulation. \note{More details added here.}
By \cref{MakeUnbounded}, there is a plane embedding of $T'$ with $V(T')$ unbounded.  By the argument proving that $L$ is a 4-connected $f$-isoperimetric plane triangulation, $T'$ is also a 4-connected $f$-isoperimetric plane triangulation.   
Let $C_0$ and $C_1$ be cycles of $T'$ such that the open disk bounded by $C_0$ contains $W\cap V(T')$, and the open disk bounded by $C_1$ contains $C_0$.  Let $N$ be the subgraph of $T'$ contained inside the closed disk bounded by $C_1$.  Note that $N$ is a near-triangulation with outer cycle $C_1$.  We refer to $N$ as the \defn{protecting near-triangulation} of $T'$ (as it, intuitively, ``protects'' those vertices of $W$ that are in $T'$). Add each such protecting near-triangulation as a component of $H'$.

Since $H\cup H'$ is finite, $V(H\cup H') \subseteq \{v_1,v_2, \dots ,v_n\}$ for some sufficiently large $n$. Since $\GG_{n+1}$ has uncountably many planar triangulations, by \cref{4ConnTriangProperSubgraph}, $\GG_{n+1}$ contains at least one planar triangulation, say $T''$, distinct from $L$. 

We claim that $T''$ contains the required additional jump $J_{p+1}$. To prove the claim, note that $T''$ contains an edge not in $L$, since by \cref{4ConnTriangProperSubgraph}, no 4-connected planar triangulation contains a planar triangulation as a proper subgraph. Since $T''$ is connected and intersects $L$, $T''$ has an edge $uw$ not in $L$ such that $u$ is in $L$. Since $T''$ and $L$ agree on the set of edges in $G$ incident with each vertex in $H$, we have that $u$ is not in $H$. Also, $w$ is not in $H$ because $L$ and $T''$ agree on the edges incident to  vertices in $H$. 

If  $w$ is in $L$, then we may take $J_{p+1}=uw$ (since neither $u$ nor $w$ is in $H$), as claimed. Now assume that $w$ is not in $L$. 

Note that $|V(L \cap T'')| \geq 4$, since $L$ and $T''$ agree on the closed neighbourhood of some vertex in $T''$. 
Since $T''$ is 4-connected, by Menger's theorem, $T''-u$ has three paths $P_1,P_2,P_3$ from $w$ to $L$ which are pairwise disjoint except that they all contain $w$.  Let $P_4$ be the edge $wu$. Each of $P_1,\dots,P_4$ have exactly one end in $L$. Observe that $P_1 \cup P_2 \cup P_3 \cup P_4$ does not intersect $H$ by the same argument that $u$ is not in $H$.  In particular, no end of $J_1, \dots, J_p$ is in $P_1 \cup P_2 \cup P_3 \cup P_4$.

\begin{claim}
\label{OnOuterCycle}
If $xy$ is an edge in $P_1\cup P_2 \cup P_3 \cup P_4$ and $x$ is in some protecting near-triangulation $N$ and $y$ is not in $N$, then $x$ is on the outer cycle of $N$.
\end{claim}

\begin{proof}
Towards a contradiction suppose $x$ is not on the outer cycle of $N$. Then the wheel centred at $x$ in $N$ is also the wheel in $T''$ centred at $x$, by the choice of $n$. In particular, $y$ is also a neighbour of $x$ in $N$, which is a contradiction.   
\end{proof}

Note that if $x$ is an isolated vertex in $H'$, then $x$ is isolated in 
$L_0$ and in $T''$ (since $L_0$ and $T''$ agree on edges incident to vertices in $H'$). Since $P_1,P_2,P_3,P_4$ are in $T''$, no vertex in $P_1\cup P_2 \cup P_3 \cup P_4$ is isolated in $T''$. Thus no vertex in $P_1\cup P_2 \cup P_3 \cup P_4$ is isolated in $H'$. In particular, $w$ is not isolated in $H'$. It is possible that $w$ is in $H'$, in which case $w$ is on the outer cycle of one of the protecting near-triangulations by \cref{OnOuterCycle}.

Let $N$ be a protecting near-triangulation such that $V(N) \cap W \cap V(P_1 \cup P_2 \cup P_3) \neq \emptyset$. For each $i \in [3]$, let $x_i$ be the first vertex of $P_i$ (beginning from $w$) in $N$, and let $y_i$ be the last vertex of $P_i$ that is in $N$.
Note that $x_i \neq y_j$ whenever $i \neq j$ because $P_1,P_2,P_3,P_4$ are pairwise disjoint except that they have $w$ in common. Also note that if $|V(P_i) \cap V(N)|=1$, then $x_i=y_i$, and if $V(P_i) \cap V(N)=\emptyset$, then both $x_i$ and $y_i$ are undefined.  Let $X:=\{x_1, x_2, x_3\}$ and $Y:=\{y_1, y_2, y_3\}$.  Let $C$ be the outer cycle of $N$, $D$ be the outer cycle of $N -C$, and $N':=N[V(C) \cup V(D)]$.   By \cref{OnOuterCycle}, $X \cup Y \subseteq V(C)$.  Moreover, by construction, $W \cap V(N)$ is disjoint from $V(N')$. 

We claim that $N'$ is $3$-connected. To prove this, observe that $N'$ consist of two disjoint cycles $C,D$ that are both face boundaries of $N'$ and all other faces are bounded by triangles. Moreover, by construction, each vertex on one of the cycles has a neighbour on the other cycle. If we delete two vertices on the same cycle, the resulting graph is clearly connected. Also, if we delete one vertex of each cycle, then the resulting paths are joined by at least one edge, proving that $N'$ is $3$-connected. By Menger's Theorem, there are three vertex-disjoint paths from $X$ to $Y$ in $N'$. Therefore, by rerouting, we may assume that $V(N) \cap W \cap V(P_1 \cup P_2 \cup P_3) = \emptyset$. 

Repeating the above argument for each protecting near-triangulation $N$ such that $V(N) \cap W \cap V(P_1 \cup P_2 \cup P_3) \neq \emptyset$, we may assume that $P_1 \cup P_2 \cup P_3 \cup P_4$ is disjoint from $W$. Since $L$ is 4-connected, it contains no $K_4$. So there are distinct $i,j\in [4]$ such that the ends of $P_i$ and $P_j$ in $L$ are non-adjacent in $L$. Thus, we may take $J_{p+1}=P_i\cup P_j$, which completes the proof.
\end{proof}

\note{The referee writes, ``3-connectedness would be enough for this argument. Instead of no $K_4$, it would be enough to say...'' Possibly some of our arguments are valid under the weaker assumption of 3-connectedness. However, the Limit Lemma itself becomes false if "4-connected" is replaced by "3-connected" even if we also exclude $K_4$.}

%%%%%%%%%%%%%%%%
\subsection{Routing with Jumps}
\label{Routing}

This subsection proves a series of lemmas about routing paths in graphs obtained by adding jumps to planar triangulations. These results are used in the proofs of \cref{InfiniteCliqueMinorStrongerStronger,thm:finitesubdivision} below. 

Consider the grid graph $G=P_\ell\CartProd P_m$, where $P_\ell$ is the path $x_1,\dots,x_\ell$ and $P_m$ is the path $y_1,\dots,y_m$. For $i\in[\ell]$, the subpath $C^i:=(x_i,y_1),\dots,(x_i,y_m)$ is called the \DefNoIndex{$i$-th column}\index{column} of $G$. 

\begin{lem} \label{lem:gridrouting}
Let $G=P_\ell\CartProd P_m$ for some $\ell,m\in\NN$ and $k\in\NN_0$ with $\ell\geq k+2$ and $m\geq k$. Let $a_1, \dots, a_k$ be vertices in the first column $C^1$ of $G$, and let $b_1, \dots, b_k$ be vertices in the last column $C^\ell$ of $G$, where $a_1,\dots,a_k$ and $b_1,\dots,b_k$ appear in the same order with respect to $P_m$. Then there exist pairwise vertex-disjoint paths $P^1, \dots, P^k$ in $G$ such that each $P^i$ is an $a_ib_i$-path and $V(P^i)\cap V(C^1)=\{a_i\}$ and $V(P^i)\cap V(C^\ell)=\{b_i\}$. 
\end{lem}

\begin{proof}
We proceed by induction on $k\geq 0$.  The case $k=0$ is vacuous. 
Assume $P_\ell$ is the path $(1,2,\dots,\ell)$, and $P_m$ is the path $(1,2,\dots,m)$.
Refer to \cref{GridCylinderRouting}(a). 
Say $a_i=(1,y_i)$ and $b_i=(\ell,z_i)$ for each $i\in[k]$. 
By symmetry, we may assume that $y_1\leq\dots \leq y_k$ and $z_1\leq\dots\leq z_k$ and $y_1\leq z_1$. 
Let $P^1$ be the path 
\[a_1=(1,y_1),(2,y_1),\dots,(\ell-1,y_1),(\ell-1,y_1+1),\dots,(\ell-1,z_1),(\ell,z_1)=b_1.\]
Let $c_i:=(\ell-1,z_i)$ for each $i\in[2,k]$. Note that $c_ib_i$ is an edge. 
Consider the subgraph $G'$ of $G$ induced by $\{1,\dots,\ell-1\}\times\{y_1+1,\dots,m\}$, which is isomorphic to $P_{\ell-1}\square P_{m-y_1}$. 
Note that 
$a_2,\dots,a_k$ are in the first column of $G'$ and
$c_2,\dots,c_k$ are in the last column of $G'$, both in increasing order. 
Thus $m-y_1\geq k-1$. 
By induction, 
there are pairwise disjoint paths $Q^2,\dots,Q^k$ in $G'$, where for $i\in[2,k]$, 
$Q^i$ is an $a_ic_i$-path, such that $a_i$ is the only vertex of $Q^i$ in the first column of $G'$, and  $c_i$ is the only vertex of $Q^i$ in the last column of $G'$.
Let $P^i$ be obtained from $Q^i$ by appending the edge $(\ell-1,z_i)(\ell,z_i)$. 
So $P^1,\dots,P^k$ are pairwise disjoint paths in $G$, where for each $i\in[k]$, 
$P^i$ is an $a_ib_i$-path, such that $a_i$ is the only vertex of $P^i$ in the first column of $G$, and  $b_i$ is the only vertex of $P^i$ in the last column of $G$.
\end{proof}

\begin{figure}
    (a)\includegraphics{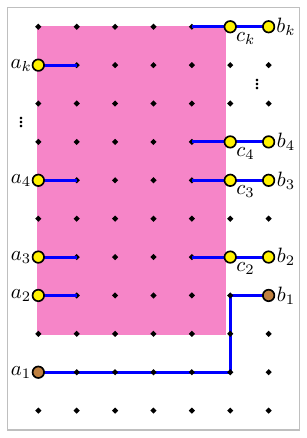}\quad
    (b)\includegraphics{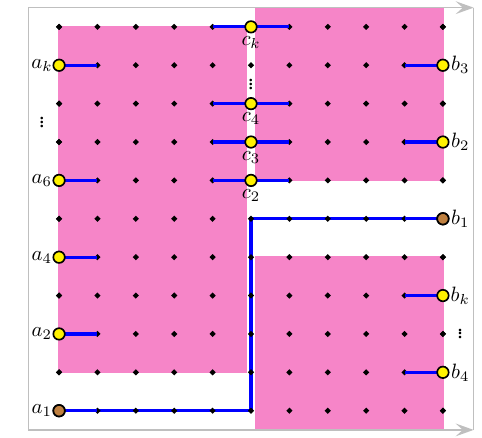}
    \caption{Routing disjoint paths in  (a) a grid (\cref{lem:gridrouting})  and (b)  a cylindrical grid (\cref{lem:cylinderrouting}).}
    \label{GridCylinderRouting}
\end{figure}

\note{The next lemma has been added to address a referee comment in the proof of~\cref{claim:switchpaths}}

Consider the \defn{cylindrical grid} graph $P_\ell\CartProd C_m$, where $P_\ell$ is the path $x_1,\dots,x_\ell$ and $C_m$ is the cycle $(y_1,\dots,y_m)$. For $i\in[\ell]$, the cycle $C^i:=((x_i,y_1),\dots,(x_i,y_m))$ is called the \splitdefn{$i$-th}{cycle} of $P_\ell\CartProd C_m$. 

\begin{lem} 
\label{lem:cylinderrouting}
Let $G:=P_\ell\CartProd C_m$ for some $\ell,m\in\NN$ and $k\in\NN_0$ with $\ell\geq 2k+1$ and $m\geq 2k-1$. Let $a_1, \dots, a_k$ be distinct vertices in the first cycle $C^1$ of $G$, and let $b_1, \dots, b_k$ be distinct vertices in the last cycle $C^\ell$ of $G$, where $a_1,\dots,a_k$ and $b_1,\dots,b_k$ appear in the cyclic order. Then there exist pairwise vertex-disjoint paths $P^1, \dots, P^k$ in $G$ such that each $P^i$ is an $a_ib_i$-path and $V(P^i)\cap V(C^1)=\{a_i\}$ and $V(P^i)\cap V(C^\ell)=\{b_i\}$. 
\end{lem}

\begin{proof}
Refer to \cref{GridCylinderRouting}(b). 
Say $P_\ell$ is the path $(1,2,\dots,\ell)$ and $C_m$ is the cycle $(1,2,\dots,m)$. 
By symmetry, we may assume that $a_1=(1,1)$ and $b_1=(\ell,y_1)$ where $y_1\in[1,\lfloor{\frac{m}{2}}\rfloor+1]$, implying $m-y_1\geq k-1$ (since $m\geq 2k-1$). Let $P^1$ be the path 
\[
a_1=(1,1),(2,1),\dots,(k+1,1),(k+1,2),\dots,(k+1,y_1),(k+2,y_1),\dots,(\ell,y_1)=b_1.
\]
So $P^1$ is the desired $a_1b_1$-path.
For each $i\in[2,k]$, let $c_i:=(k+1,y_1+i-1)$. 
Note that this y-coordinate is at least $y_1+1$ and at most $y_1+k-1\leq m$. 
Let $G_1$ be the cylindrical grid induced by $\{1,\dots,k+1\}\times V(C_m)$. 
Let $G_2$ be the cylindrical grid induced by $\{k+1,\dots,\ell\}\times V(C_m)$. 
Note that $P^1$ only intersects $G_1$ in the horizontal path $(1,1),(2,1),\dots,(k+1,1)$, 
and $P^1$ only intersects $G_2$ in the horizontal path $(k+1,y_1),(k+2,y_1),\dots,(\ell,y_1)$. 
Let $G'_1:=G_1-V(P^1)$, which is isomorphic to $P_{k+1} \square P_{m-1}$. 
Let $G'_2:=G_2-V(P^1)$, which is isomorphic to $P_{\ell-k} \square P_{m-1}$. 
The vertices $a_2,\dots,a_k$ are in the first column of $G'_1$.
The vertices $c_2,\dots,c_k$ are in the last column of $G'_1$ and the first column of $G'_2$.
The vertices $b_2,\dots,b_k$ are in the last column of $G'_2$.
By \cref{lem:gridrouting}, 
there exist pairwise vertex-disjoint paths $Q^2, \dots, Q^k$ in $G'_1$ 
such that for each $i\in[2,k]$, 
$Q^i$ is an $a_ic_i$-path, $a_i$ is the only vertex in $Q^i$ in the first column of $G'_1$,
and $c_i$ is the only vertex in $Q^i$ in the last column of $G'_1$. 
Similarly, since $\ell-k\geq (k-1)+2$, by \cref{lem:gridrouting}, 
there exist pairwise vertex-disjoint paths $R^2, \dots, R^k$ in $G'_2$ 
such that for each $i\in[2,k]$, 
$R^i$ is a $c_ib_i$-path, $c_i$ is the only vertex in $R^i$ in the first column of $G'_2$,
and $b_i$ is the only vertex in $R^i$ in the last column of $G'_2$. For $i\in[2,k]$, let $P^i$ be the $a_ib_i$-path obtained by concatenating $Q^i$ and $R^i$. 
By construction, $P^1,\dots,P^k$ are pairwise disjoint, and for each $i\in[k]$, 
 $a_i$ is the only vertex in $P^i$ in the first column of $G$, 
and $b_i$ is the only vertex in $P^i$ in the last column of $G$. 
\end{proof}

The next lemma is a special case of the 2-linkage theorem, independently due to \citet{Thomassen80} and \citet{Seymour80}. For a graph $G$ and $u,v \in V(G)$,  let \defn{$G+uv$} be the graph obtained from $G$ by adding the edge $uv$ if it does not already exist. 

\begin{lem} \label{lem:special2linkage}
Let $G$ be a near-triangulation with outer cycle $C$, let $p_1,q_1,p_2,q_2$ be distinct vertices on $C$, and let $u,v$ be non-adjacent vertices of $G$ such that neither $u$ nor $v$ is in $C$.  If $G$ does not contain a $(\leq 3)$-separation $(G_1, G_2)$ such that $\{p_1, q_1, p_2, q_2\} \subseteq V(G_1)$ and $G_1$ is planar, then there are vertex-disjoint paths $P$ and $Q$ in $G + uv$ such that the ends of $P$ are $p_1$ and $p_2$ and the ends of $Q$ are $q_1$ and $q_2$. 
\end{lem}

For a cycle $C$ in a plane graph $G$, let $\Delta(C) \subseteq \mathbb{R}^2$ be the closed disk bounded by $C$.  We say that $C$ is \defn{tight} if there is no cycle $D$ in $G$ such that $\Delta(C) \subseteq \Delta(D)$ and $|V(D)| < |V(C)|$. A cycle $D$ in $G$ \defn{surrounds} a cycle $C$ in $G$ if $V(C) \cap V(D)=\emptyset$ and $\Delta(C) \subseteq \Delta(D)$.

\begin{lem} \label{FindDisjointTightCycle}
For every cycle $C$ of an infinite isoperimetric plane triangulation $G$, there is a tight cycle that surrounds $C$.
\end{lem}

\begin{proof}
Suppose for the sake of contradiction that no tight cycle surrounds $C$. Let $\CC$ be the set of facial cycles of $G$ that intersect $\Delta(C)$ (recall that $\Delta(C)$ is a \emph{closed} disk).  Since $G$ is locally finite and isoperimetric, $\CC$ is finite.  Let $C^1$ be the symmetric difference of all cycles in $\CC$.  We claim that $C^1$ is a cycle surrounding $C$.
To see this, let $ p \in \Delta(C)$ and consider any curve from $p$ to infinity containing no vertex of $G$. The last point on the curve which intersects a cycle in $\CC$ is a point on an edge belonging to precisely one cycle in $\CC$, and that edge is in $C^1$. 
Since $C^1$ is not tight, there is a cycle $C^2$ with $\Delta(C^1)\subseteq \Delta(C^2)$ and $|V(C^2)|<|V(C^1)|$. So $C^2$ surrounds $C$.  Repeating this argument we obtain an infinite sequence $C^1,C^2,\dots$ of non-tight cycles with $\Delta(C^i) \subseteq \Delta(C^{i+1})$ and $|V(C^{i+1})|<|V(C^i)|$, where each $C^i$ surrounds $C$. This is a contradiction since $|V(C^1)|$ is finite.
\end{proof}

\begin{lem} \label{claim:tightcycle}
%let $C$ be a tight cycle in $G$, and let $D$ be a cycle in $G$ such that $\Delta(C) \subseteq \Delta(D)$
Let $G$ be an isoperimetric plane triangulation, and  let $D$ be a cycle in $G$ that surrounds a tight cycle $C$ in $G$.
Then there are $|V(C)|$ vertex-disjoint paths in $G$ from $V(C)$ to $V(D)$.
\end{lem}

\begin{proof}
\note{The referee said to prove it. We have added a proof, and added the isoperimetric assumption to simplify the proof.}
Suppose for the sake of contradiction that the desired paths do not exist. By Menger's Theorem, there exists a minimal set $X \subseteq V(G)$ with $|X| < |V(C)|$ such that there is no path in $G-X$ between $V(C)$ and $V(D)$, and there are vertex-disjoint $CD$-paths $P_1,\dots,P_{|X|}$ in $G$, each containing precisely one vertex of $X$. 
Each path $P_i$ has exactly one vertex in $C$ and exactly one vertex in $D$, so $P_i$ avoids the interior of $C$ and the exterior of $D$.
Let $G'$ be the plane triangulation obtained from $G$ by deleting all vertices in the interior of $C$, deleting all vertices in the exterior of $D$, adding a vertex $v$ in the interior of $C$ adjacent to every vertex in $C$, and adding a vertex $w$ in the exterior of $D$ adjacent to every vertex in $D$. 
Since $G$ is isoperimetric, $G'$ is finite. 
By construction, $P_i$ is a path of $G'-v-w$, and $X\subseteq V(G'-v-w)$. 

We claim that $G'-X$ has exactly two components, one containing $v$ and one containing $w$. Suppose for the sake of contradiction that some component $A$ of $G'-X$ avoids $v$ and $w$. If $V(A \cap P_i)\neq\emptyset$ for some $i$, say $x\in V(A\cap P_i)$, then since $|V(P_i)\cap X|=1$, there is a path in $G'-X$ from $x$ to $v$ or $w$, and thus $v$ or $w$ is in $A$. Now assume that $V(A\cap P_i)=\emptyset$ for each $i$. Thus $A$ lies between two paths $P_i$ and $P_j$ in the plane embedding of $G'$. Thus $V(P_i\cup P_j)\cap X$ is a set of two vertices that separates $A$ from $\{v,w\}$ in $G'$, contradicting the 3-connectivity of $G'$. So $G'-X$ has exactly two components, one containing $v$ and one containing $w$. 

For each vertex $x\in X$ there is a $CD$-path in $G-(X\setminus\{x\})$, implying $G'-(X\setminus\{x\})$ is connected. Thus, $X$ is a minimal separating set in $G'$. (Note that in general, minimality with respect to two specific vertex sets does not necessarily imply that the set is a minimal separating set.) Since $G'$ is a finite plane triangulation, by Proposition 8.2.3 in \citep{MoharThom}, $X$ induces a cycle in $G'$ such that the two components of $G'-X$ are the interior and exterior of $G'[X]$, respectively. Hence $G[X]$ is a cycle such that $\Delta(C)\subseteq \Delta(G[X])$ and $|X|<|V(C)|$. 
This contradicts the assumption that $C$ is tight.
\end{proof}

Repeatedly applying \cref{claim:tightcycle} we obtain the following.

\begin{lem} 
\label{claim:cylindricalgrid}
\note{Added the isoperimetric assumption here.}
Let $C^1, \dots, C^\ell$ be tight cycles in an isoperimetric plane triangulation $G$ such that $|V(C^1)| \geq m$, and $C^j$ surrounds $C^i$ for all $1 \leq i < j \leq \ell$. Then $G$ contains a subdivision of $P_\ell \CartProd C_m$, where the branch vertices of the $i$-th cycle of $P_\ell \CartProd C_m$ are in $C^i$.
\end{lem}

% \begin{lem} 
% \label{claim:cylindricalgrid}
% Let $C^1, \dots, C^m$ be tight cycles in a plane graph $G$ such that $|V(C^1)|=\ell$, and $C^j$ surrounds $C^i$ for all $1 \leq i < j \leq m$. Then $G$ contains a subdivision of $C_\ell \CartProd P_m$, where the branch vertices of the $i$-th cycle of $C_\ell \CartProd P_m$ are in $C^i$.
% \end{lem}

\begin{lem} \label{claim:switchpaths}
Let $G':=G \cup M$, where $G$ is a $4$-connected $f$-isoperimetric plane triangulation and $M$ is an infinite matching of edges not in $G$. Let $C$ be a tight cycle in $G$ of length $m \geq 2k+4$, and $x_1, \dots, x_k$ be vertices in clockwise order around $C$.  Then there is a tight cycle $D$ surrounding $C$, there are vertices $w_1,\dots,w_k$ in $D$, and there are vertex-disjoint paths $Q^1, \dots, Q^k$ from $C$ to $D$ such that:
\begin{itemize}
\item $Q^i$ is an $x_iw_i$-path for each $i\in[k]$,  
\item $V(Q^i)\cap V(C)=\{x_i\}$ for each $i \in [k]$,
\item $V(Q^i) \cap V(D):=\{w_i\}$ for each $i \in [k]$, 
\item $w_2, w_1, w_3, \dots, w_k$ appear in this clockwise order on $D$, 
\item $Q^1 \cup \dots \cup Q^k$ is contained in the subgraph of $G$ between $C$ and $D$ together with one edge $e\in M$.
\end{itemize}
\end{lem}

\begin{proof}
By \cref{FindDisjointTightCycle}, there are tight cycles $C=C^1, C^2, \dots, C^{2k+2}$ such that $C^j$ surrounds $C^i$ for all $1 \leq i < j \leq 2k+2$; see \cref{SwitchPaths}.  Since $G$ is $f$-isoperimetric, there are only finitely many vertices of $G$ inside $C^{2k+2}$.  By \cref{FindDisjointTightCycle} and since $M$ is infinite, there is a tight cycle $C^{2k+3}$ surrounding $C^{2k+2}$ and there is an edge $vw \in M$ such that $v \in \Delta(C^{2k+3})$, and $\{v,w\} \cap \Delta(C^{2k+3})=\emptyset$. By \cref{FindDisjointTightCycle}, there exists a tight cycle $C^{2k+4}$ surrounding $C^{2k+3}$ such that $w \in \Delta(C^{2k+4})$ and $w \notin V(C^{2k+4})$.  Let $C^{2k+5}$ be a tight cycle surrounding $C^{2k+4}$ with $\Delta(C^{2k+5})$ minimal (under inclusion). Finally, let $C^{2k+6}$ be a tight cycle surrounding $C^{2k+5}$.

\begin{figure}[!ht]
\centering
\includegraphics[width=\textwidth]{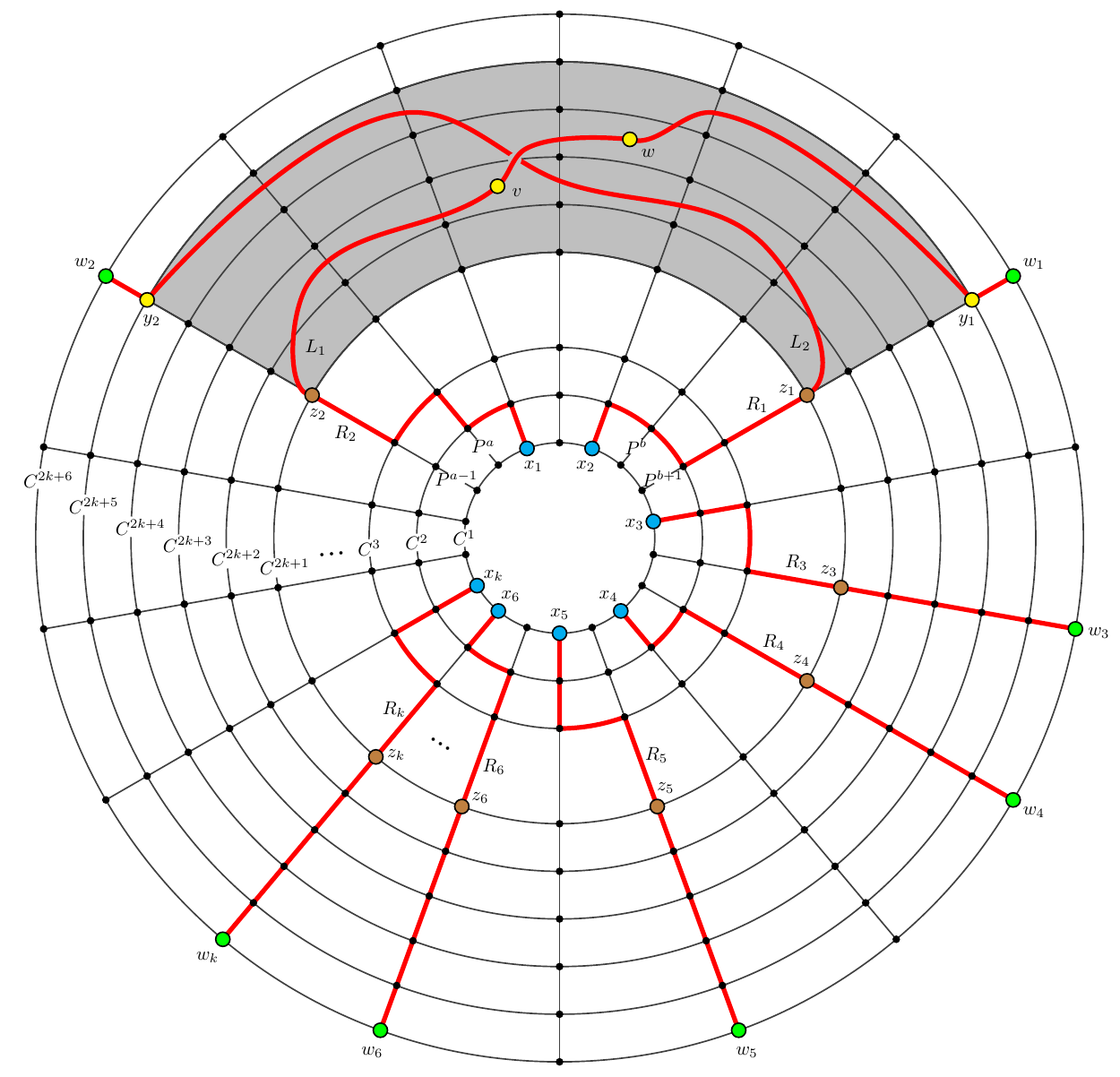} 
\caption{Illustration for the proof of \cref{claim:switchpaths}.
\label{SwitchPaths}}
\end{figure}

By \cref{claim:cylindricalgrid}, $G$ contains a subgraph $H$ that is a subdivision of $P_{2k+6} \CartProd C_m$, where the branch vertices of the $i$-th cycle of $P_{2k+6} \CartProd C_m$ are contained in $C^i$. Choose such an $H$ with $|V(H)|$ minimum. Let $P^1, \dots, P^m$ be the $C^1C^{2k+6}$-paths of the subdivision (ordered clockwise). For $i,j \in [m]$, let $G(i,j)$ be the subgraph of $G$ induced by the vertices between $C^{2k+1}$ and $C^{2k+5}$ (inclusive) and between $P^i$ and $P^j$ (inclusive and clockwise from $P^i$ to $P^j$). Since $m \geq 2k+4$, there exist $a,b \in [m]$ such that $v$ and $w$ are both in $G(a,b)$, and there are at least $k$  branch vertices of $H$ in $V(C^{2k+1}) \setminus V(G(a,b))$. By symmetry, we may assume $1<a<b<m$.  For each $i \in [m]$, let $s_i$ and $t_i$ be the unique vertices of $P^i$ on $C^{2k+1}$ and $C^{2k+5}$, respectively.  Let $z_1, z_2, \dots, z_k$ be ordered clockwise on $C^{2k+1}$ where $z_1:=s_{a-1}$, $z_2:=s_{b+1}$, and $z_3, \dots, z_k$ are branch vertices of $H$ contained in $V(C^{2k+1}) \setminus V(G(a-1,b+1))$. 
Let $H'$ be the subdivision of $P_{2k+1} \CartProd C_m$ contained in $H-\bigcup_{i=2k+2}^{2k+6} V(C^{i})$.  Therefore,  by \cref{lem:cylinderrouting} there exist vertex-disjoint $C^1C^{2k+1}$-paths $R^1, \dots, R^k$ in $H'$ such that the ends of $R^i$ are $x_i$ and $z_i$ for each $i \in [k]$.

\note{The referee wrote ``For this we need that the paths $P_j$ containing $x_1,\dots,x_k$ are in increasing order, and that the same is true of the $z_i$. But that doesn't follow from the construction.'' The addition of \cref{lem:cylinderrouting} fixes this.} 

Let $y_1:=t_{b+1}$ and $y_2:=t_{a-1}$. We now prove that there exist vertex-disjoint paths $L^1$ and $L^2$ in $G(a-1,b+1) +vw$ such that the ends of $L^1$ are $z_1$ and $y_1$ and the ends of $L^2$ are $z_2$ and $y_2$.

By construction, neither $v$ nor $w$ is on the outer cycle $O$ of $G(a-1,b+1)$.  We claim that $G(a-1, b+1)+vw$ does not contain a $(\leq 3)$-separation $(G_1, G_2)$, such that $\{s_{a-1}, t_{a-1}, s_{b+1}, t_{b+1}\} \subseteq V(G_1)$ and $G_1$ is planar.  Suppose for the sake of contradiction that $(G_1, G_2)$ is such a separation.  Since $G$ is $4$-connected, $(G_1, G_2)$ is a $2$- or $3$-separation.  First suppose it is a $2$-separation, say $V(G_1) \cap V(G_2)=:\{x,y\}$.  Since $G$ is a plane triangulation, $xy$ is a chord of $O$. Since $\{s_{a-1}, t_{a-1}, s_{b+1}, t_{b+1}\} \subseteq V(G_1)$, we have $\{x,y\} \subseteq V(C^{2k+5}), \{x,y\} \subseteq V(P^{a-1}), \{x,y\} \subseteq V(P^{b+1})$, or $\{x,y\} \subseteq V(C^{2k+1})$.  Since $C^{2k+1}$ is tight, $\{x,y\} \not\subseteq V(C^{2k+1})$.  By the minimality of $|V(H)|$, $\{x,y\} \not \subseteq V(P^{a-1})$ and $\{x,y\} \not \subseteq V(P^{b+1})$.  By the minimality of $\Delta(C^{2k+5})$, $\{x,y\} \not \subseteq V(C^{2k+5})$.
Thus, we may assume that $(G_1, G_2)$ is a $3$-separation, say $V(G_1) \cap V(G_2)=:\{x,y,z\}$.  Since $G$ is $4$-connected, every component of $(G(a-1, b+1)+vw)-\{x,y,z\}$ must contain a vertex of $O$.  It follows that $|\{x,y,z\} \cap V(O)| \geq 2$.  By symmetry, we may assume that $\{x,z\} \subseteq V(O)$.  If $xz$ is a chord of $O$, then $G_1$ cannot contain all of $s_{a-1}, t_{a-1}, s_{b+1}, t_{b+1}$ by the previous case. Thus, $xz$ is not a chord of $O$.  If $y$ is not adjacent to both $x$ and $z$, then $G-\{x,y,z\}$ is connected. Thus, $xyz$ is a path of $G$. 

Since $\{s_{a-1}, t_{a-1}, s_{b+1}, t_{b+1}\} \subseteq V(G_1)$, there is an $xz$-subpath $P$ of $O$ such that no internal vertex of $P$ is in $\{s_{a-1}, t_{a-1}, s_{b+1}, t_{b+1}\}$.  Let $T$ be the near-triangulation bounded by $P \cup \{xy, yz\}$.  If $\{x,z\} \subseteq V(C^{2k+5})$, then $\{v,w\} \cap V(T)=\emptyset$ since $\{v,w\} \subseteq \Delta(C_{2k+7}) \setminus V(C_{2k+4})$.  It follows that $G_2=T$, and $G_1$ is non-planar since it is of the form $G_1'+vw$, where $G_1'$ is a near-triangulation and neither  $v$ nor $w$ is on the outer cycle of $G_1'$. If $\{x,z\} \subseteq V(C^{2k+1})$, then $\{v,w\} \cap V(T) =\emptyset$, since $\{v,w\} \cap \Delta(C^{2k+2})=\emptyset$.  Thus, as above, $G_1$ is non-planar.  The remaining cases are  
$\{x,z\} \subseteq V(P^{a-1})$ or $\{x,z\} \subseteq V(P^{b+1})$.  In either case, $\{v,w\} \cap V(T) \subseteq \{y\}$, since $\{v,w\} \subseteq V(G(a, b))$.  It follows that $G_2=T$, and $G_1$ is non-planar since it is of the form $G_1'+vw$, where $G_1'$ is a near-triangulation and at least one of $v$ or $w$ is not on the outer cycle of $G_1'$.    
\note{The referee said to flesh out this argument. We have added more details.}

Therefore, by \cref{lem:special2linkage}, there are vertex-disjoint paths $L^1$ and $L^2$ in $G(a-1,b+1) +vw$ such that the ends of $L^1$ are $z_1$ and $y_1$ and the ends of $L^2$ are $z_2$ and $y_2$, as claimed. 

Use $L^1$ to extend $R^1$ to a path $S^1$ with $y_1$ as an end, and use $L^2$ to extend $R^2$ to a path $S^2$ with $y_2$ as an end. Then use $H$ to extend $S^1$ and $S^2$ to paths $Q^1$ and $Q^2$ such that for each $i \in [2]$, $Q^i$ has an end $w_{i}$ in $D$, and $w_{i}$ is the only vertex of $Q^i$ on $D$. For $i \in [3,k]$, use $H$ to extend $R^i$ to a path $Q^i$ with an end $w_i$ in $C^{2k+6}$, where $w_i$ is the only vertex of $Q^i$ on $C^{2k+6}$. This completes the proof, with $D:=C^{2k+6}$ and $e:=vw$.
\end{proof}

The next lemma is a key ingredient in the proof of \cref{InfiniteCliqueMinor}(a).

\begin{lem}
\label{StrongerConstructMinor}
Let $G':=G \cup M$, where $G$ is a $4$-connected $f$-isoperimetric plane triangulation and $M$ is an infinite matching of edges not in $G$. Then $K_{\aleph_0}$ is a minor of $G'$. 
\end{lem}

\begin{proof}
\note{This proof has been revised according to the referee's suggestions.}
Using \cref{claim:switchpaths} it is straightforward, although tedious, to introduce and extend  more and more paths and join any pair of them by paths so that we obtain the desired minor. We present a formal proof below. For this we use the lexicographic ordering of pairs. More precisely, say $(a,b)\leq (c,d)$ if $a<c$, or $a=c$ and $b\leq d$.  We claim that for all $1 \leq a \leq b$,  there exists a path $P_a^b$ in $G'$ and a tight cycle $D_b^a$ in $G$ satisfying the following:
\begin{itemize}
    \item $D_b^a$ surrounds $D_d^c$  if $(d,c) \leq (b,a)$,
    \item $P_a^b$ is a $D_a^1$-$D_b^b$ path,
    \item $|V(P_a^b) \cap V(D_d^c)|=1$ for all $a \leq d \leq b$ and $c \leq d$,  
    \item $P_a^c$ is a subpath of $P_a^b$ for all $a \leq c \leq b$.  
    \item The paths $P_1^b, P_2^b, \dots, P_b^b$ are vertex-disjoint,
    \item 
    For all $a < b$, there exists $\ell \leq b$ such that the two vertices of $P_a^b \cup P_b^b$ on $D_b^\ell$ are consecutive in the cyclic ordering of the $b$ vertices of $P_1^b \cup \dots \cup P_b^b$ on $D_b^\ell$.
\end{itemize}

We proceed by induction on $b$.  For the base case, let $D_1^1$ be any tight cycle in $G$, and $P_1^1=v$, where $v$ is any vertex of $D_1^1$.  Fix $b \geq 2$, and suppose that for all $1 \leq i \leq j \leq b-1$, $P_i^j$ and $D_j^i$ have already been defined to satisfy the above properties.

Let $D_b^1$ be any tight cycle such that $D_b^1$ surrounds $D_{b-1}^{b-1}$ and $|V(D_b^1)| \geq 2b+4$.  Note that $D_b^1$ exists by repeatedly applying \cref{FindDisjointTightCycle} and using the fact that $G$ is $f$-isoperimetric. 

By \cref{claim:tightcycle}, we can extend $P_1^{b-1}, \dots, P_{b-1}^{b-1}$ to vertex-disjoint paths
$Q_1^{b-1}, \dots, Q_{b-1}^{b-1}$ such that each path has one end in $D_b^1$. Let $X$ be the set of ends of $Q_1^{b-1}, \dots, Q_{b-1}^{b-1}$ on $D_b^1$. Let $y$ be an arbitrary vertex of $V(D_b^1) \setminus X$.  Suppose the cyclic ordering of $X \cup \{y\}$ on $D_b^1$ is $y, x_1, x_2, \dots, x_{b-1}$.

By \cref{claim:switchpaths}, there is a tight cycle $D$ such that $D$ surrounds $D_b^1$, and we can extend $y, Q_1^{b-1}, \dots, Q_{b-1}^{b-1}$ so that each path has one end in $D$, and the 
%cyclic ordering of $X \cup \{a\}$ on $D$ is $x_1, \dots, x_{\ell-1}, a, x_{\ell}, \dots, x_k$. 
cyclic ordering of these ends on $D$ is $x_1', y', x_2', \dots, x_{b-1}'$ (where $y'$ is end of the path containing $y$, and $x_i'$ is end of the path containing $x_i$).
Define $D_b^2:=D$.  Repeat this argument to define $D_b^3, \dots, D_b^b$.  We let $P_1^b, \dots, P_b^b$ be the set of extended paths. 
By construction, for all $1 \leq i \leq j \leq b$, $P_i^j$ and $D_j^i$ satisfy all of the above properties, as required. 

For each $i \in \NN$ let $P_i=\bigcup_{j \geq i} P_i^j$.  Note that each $P_i$ is an infinite one-way path in $G$ and the paths $P_1, P_2, \dots$ are all vertex-disjoint.  
Moreover, for all $i < j$ there is a $P_iP_j$-path $P_{i,j}$ in $G$ such that all these paths are pairwise vertex-disjoint.  Therefore, we obtain a $K_{\aleph_0}$ minor in $G'$ by contracting each $P_i$ to a vertex and contracting all but one edge from each $P_{i,j}$. 
\end{proof}

%%%%%%%%%%%%%%%%%%%%%%%%%%%%%%%%%%%%%%%%%%%%%%%%%
\subsection{\texorpdfstring{Proof of \cref{InfiniteCliqueMinor}}{Proof of Theorem~1.1}}
\label{ProofInfiniteCliqueMinor}

We now combine the lemmas from the previous subsections to prove our first main result, \cref{InfiniteCliqueMinor}.

\begin{lem}
\label{InfiniteCliqueMinorStrongerStronger}
Let $U$ be a countable graph containing an uncountable family $\GG$ of subgraphs, each of which is a 4-connected $f$-isoperimetric planar triangulation, for an arbitrary function $f:\NN\to\NN$. Then the complete graph $K_{\aleph_0}$ is a minor of $U$.
\end{lem}

\begin{proof}
Let $L$ be any non-trivial connected component of any $\GG$-limit in $U$. By \cref{LimitLemma}, $L$ is a 4-connected $f$-isoperimetric planar triangulation, and there exists an infinite set $\PART$ of pairwise-disjoint jumps of $L$ in $U$.  \note{Revised here according to referee suggestion.} Let $Q$ be the graph obtained from $L$ together with all paths in $\PART$ by contracting each path in $\PART$ to an edge. Observe that $Q$ is isomorphic to a graph obtained from $L$ by adding an infinite matching of edges not in $L$. By \cref{StrongerConstructMinor}, $K_{\aleph_0}$ is a minor of $Q$. The result follows since $Q$ is a minor of $U$.
\end{proof}

\note{We originally claimed  \cref{thm:uncountabledegree7} with 7 replaced by 6. The referee pointed out an error in the proof, now fixed.}

\begin{lem} 
\label{thm:uncountabledegree7}
There exist uncountably many pairwise non-isomorphic planar triangulations of maximum degree 7 obtained by adding edges to the infinite grid.
\end{lem}

\begin{proof}
Let $G$ be the infinite grid with vertex set $\ZZ \times \ZZ$.  Let $\mathcal{T}$ be the family of graphs obtained from $G$ as follows.  For each face $f$ of $G$, let $b(f)$ be the bottom-left vertex of $f$.  For each face $f$ of $G$, add a diagonal edge across $f$ with slope $1$ or $-1$ if $b(f) \in \{(9,0), (15,0), (21,0), \dots \}$, with slope $-1$ if $b(f) \in \{(0,0), (6,0), (12,0), \dots \}$, and with slope $1$ otherwise.  Observe that $\mathcal{T}$ is an uncountable family of planar  triangulations, and each $T \in \mathcal{T}$ has maximum degree 7.  Thus, it suffices to show that the graphs in $\mathcal{T}$ are pairwise nonisomorphic.  Let $T_1, T_2 \in \mathcal{T}$ and let $\phi: \ZZ \times \ZZ \to \ZZ \times \ZZ$ be an isomorphism from $T_1$ to $T_2$. Since $\phi$ is an isomorphism, we also regard $\phi$ as a bijection from $E(T_1)$ to $E(T_2)$.

For each $j \in [2]$, let $D_j$ be the subgraph of $T_j$ consisting of the diagonal edges of $T_j$ with slope $-1$ and their ends. Observe that $V(D_j)$ is the set of degree-$7$ vertices of $T_j$, and $D_j$ is an induced subgraph of $T_j$.
Thus, $\phi(E(D_1))=E(D_2)$.   Order $E(D_1)$ as $d_1, d_2, \dots, $ according to the end $(x_i, 1)$ of $d_i$ on the line $y=1$.   

We claim that $\phi((x_i,1))=(x_i,1)$ for all $i$.  Note that $x_1=0$ and for each $j \in [2]$, $(0,1)$ is the unique vertex $v_j$ of $T_j$ such that there are at most two other vertices in $V(D_j)$ at distance at most $6$ from $v_j$ in $T_j$. Thus, $\phi((x_1,1))=(x_1,1)$. Now, inductively assume that $\phi((x_i,1))=(x_i,1)$ for all $i \in [k]$. For each $j \in [2]$, $(x_{k+1}, 1)$ is the unique vertex $v_j \in V(D_j)$ such that $\dist_{T_j}(u,v_j)=\dist_{T_j}(u,(x_{k+1}, 1))$ for all $u \in \{(x_1,1), \dots, (x_k, 1)\}$.  Thus, $\phi((x_{k+1},1))=(x_{k+1},1)$.  By induction, $\phi((x_i,1))=(x_i,1)$ for all $i \in \NN$.  Since $\phi(E(D_1))=E(D_2)$, we also have $\phi((x_i+1,0))=(x_i+1,0)$ for all $i \in \NN$.  That is $\phi(v)=v$ for all $v \in V(D_1)$.  Let $V_i:=\{v \in V(T_1) \mid \dist_{T_1}(v, V(D_1)) \leq i\}$.   We have already established that $\phi(v)=v$ for all $v \in V_0$.  So, inductively assume that $\phi(v)=v$ for all $v \in V_k$ for some $k \geq 0$.  Let $v' \in V_{k+1} \setminus V_k$.  For each $j \in [2]$, $v'$ is the unique vertex $v \in V_{k+1} \setminus V_k$ such that $\dist_{T_j}(v, u)=\dist_{T_j}(v',u)$ for all $u \in V_k$.  Thus, $\phi(v')=v'$. By induction, for each $i \in \NN$, we have $\phi(v)=v$ for all $v \in V_i$.  Thus, $\phi(v)=v$ for all $v \in V(T_1)$, and so $T_1=T_2$, as required.  
\end{proof}

The next result implies and strengthens \cref{InfiniteCliqueMinor}(a).  

\begin{thm}
\label{InfiniteCliqueMinorDegree7}
Let $\GG$ be the class of planar triangulations of maximum degree 7 obtained by adding edges to the infinite grid. If a graph $U$ contains every graph in $\GG$,  then 
 $K_{\aleph_0}$  is a minor of $U$.
\end{thm}

\begin{proof}
Each planar triangulation in $\GG$ is $4$-connected and $f$-isoperimetric for $f(n) \in O(n^2)$.  By~\Cref{thm:uncountabledegree7}, there are uncountably many pairwise non-isomorphic planar triangulations in $\GG$. By~\Cref{InfiniteCliqueMinorStrongerStronger}, $K_{\aleph_0}$  is a minor of $U$.
\end{proof}

We now prove \cref{InfiniteCliqueMinor}(b).

\begin{thm} 
\label{thm:finitesubdivision}
If a graph $U$ contains every planar graph, then $U$ contains a subdivision of $K_t$ for every $t \in \NN$.
\end{thm}

\note{This proof has been substantially rewritten to address the referee comments.} 

\begin{proof}
For a plane graph $H$, let $\FF_H$ be the collection of plane triangulations that contain $H$ as a spanning subgraph (preserving the given embedding of $H$). We begin by constructing a plane graph $H$ with vertices of arbitrarily large degree, such that there are uncountably many plane triangulations in $\mathcal{F}_H$ that are pairwise non-isomorphic, 4-connected, and $f$-isoperimetric, where $f(k):=4k$.

Let $d_0, d_1,d_2,\dots$ and $\ell_0,\ell_1,\ell_2,\dots$ be recursively defined sequences where $d_0:=1$, $\ell_0:=1$ and for $i\geq 0$, 
\[d_{i+1}:=d_i+\ell_i \text{\quad and \quad} \ell_{i+1}:=4d_i\ell_i + 2\ell_i(\ell_i+1).\] 
Let $T$ be any infinite tree rooted at a vertex $v_{0,1}$, with $\ell_i$ vertices $v_{i,1},\dots,v_{i,\ell_i}$ in $T$ at distance $i$ from $v_{0,1}$, where each vertex $v_{i,j}$ has $4d_i+4j$ children in $T$. 
This is well-defined since 
$$\ell_{i+1}= 4d_i\ell_i + 2\ell_i(\ell_i+1) = \sum_{j=1}^{\ell_i}(4d_i+4j).$$
Let $H$ be the plane graph obtained from $T$ by adding, for each $i\geq 1$, a cycle $C_i$ on the vertices in $T$ at distance $i$ from the root. Let $\FF:=\FF_H$. 

Note that the root vertex $v_{0,1}$ has degree $\ell_1=8$ in $T$ and in $H$. For each non-root vertex $v_{i,j}$, 
$$15\leq 4d_i+7 \leq \deg_H(v_{i,j})=4d_i+4j+3 \leq 4d_i+4\ell_i+3.$$
Since $4d_{i+1}+7 = 4d_i+4\ell_i+3+4$, 
for any distinct vertices $v,w\in V(H)$ we have $|\deg_H(v)-\deg_H(w)|\geq 4$. 
By construction, every face of $H$ has size 3 or 4, and each vertex is incident to at most three faces of size 4. 
Thus, for any $H' \in\FF$, every vertex $v$ satisfies 
$\deg_H(v)\leq \deg_{H'}(v)\leq \deg_H(v)+3$.
Hence, if $H_1,H_2\in \FF$, then 
$\deg_{H_1}(v)\neq \deg_{H_2}(w)$ for any distinct vertices $v,w\in V(H)$. Thus, if $H_1$ and $H_2$ are isomorphic, then the isomorphism is the identity, implying $H_1=H_2$. 
That is, no two distinct plane triangulations in $\FF$ are isomorphic. On the other hand, there are infinitely many faces of $H$ with size 4. So there are uncountably many plane triangulations in $\FF$ (since at each face of size 4 there is a choice of two edges to add). Finally, every plane triangulation $H'$ in $\FF$ is 4-connected (since every triangle is a face) and has minimum degree 8, implying $H'$ is $f$-isoperimetric with $f(k):=4k$ by \cref{IsoperimetricMinDegree}.

We now show that $U$ contains a subdivision of $K_t$ for every $t \in \NN$.  Choose $\lambda \in \NN$ such that $\ell_{\lambda-1} \ge t$ and $4d_{\lambda-1} \ge t$,
 where $(\ell_i)_{i=0}^\infty$ and $(d_i)_{i=0}^\infty$  are the sequences used to define $H$. 
Let $Y$ be the set of vertices of $H$ contained in the disk bounded by the cycle $C_\lambda$ of $H$.  
Since $U$ contains all planar graphs, for each $F \in \FF$, there is an injective homomorphism $\phi_F: V(F) \to V(U)$. 
Since $\FF$ is uncountable, $Y$ is finite, and $V(U)$ is countable, there exists an uncountable subset $\FF_0$ of $\FF$ such that $\phi_F(y)=\phi_{F'}(y)$ for all $F,F' \in \FF_0$ and all $y \in Y$.  Let $\GG$ be the collection of subgraphs $G$ of $U$ such that there exists an $F \in \FF_0$ such that $\phi_F^{-1}$ is an isomorphism from $G$ to $F$.  

Choose $F \in \FF_0$ and let $X:=\phi_F(Y)$ (note that $X$ is independent of the choice of $F$).  Let $L'$ be a $\GG$-limit in $U$, with respect to an enumeration $x_1, x_2, \dots$ of $V(U)$, where $X=\{x_1,\dots,x_{|X|}\}$.  Let $L$ be the component of $L'$ which contains $X$. By \cref{LimitLemma}, $L$ is a $4$-connected $f$-isoperimetric planar triangulation, and there is an infinite set $\JJ$ of pairwise disjoint jumps of $L$ in $U$. Moreover, by the choice of the enumeration, $L[X]$ is isomorphic to a graph obtained from $H[Y]$ by adding an edge across each face of length $4$ of $H[Y]$. 

 Let $s \in \NN$ be the minimum number of steps required to transform the cyclic sequence 
$$(1,2), \dots ,(1,t), (2,1), (2,3), \dots, (2,t), (3,1), \dots, (3,t), \dots, (t,1), \dots, (t,t-1)$$ into
$$(1,2), (2,1), (1,3), (3,1), \dots, (1,t), (t,1), (2,3), (3,2), \dots, (t-1, t), (t,t-1)$$
by swapping adjacent elements.  For each $i \in [s]$, let $(a_i,b_i)$ and $ (c_i,d_i)$ be the adjacent elements swapped at step $i$.

 % Let $O_1$ and $O_2$ be cyclic orderings of $[k]$.  We say that $O_2$ is a \defn{swap} of $O_1$ if $O_2$ can be obtained from $O_1$ by swapping adjacent elements of $O_1$.  For each $k \in \NN$, define $s(k)$ to be the minimum integer such that $O_1$ can be transformed to $O_2$ with at most $s(k)$ swaps for all cyclic permutations $O_1$ and $O_2$ of $[k]$.

% Let $I:=\{(i,j): i,j \in [t], i \neq j\}$. Let $s \in \NN$ be the minimum number of steps required to transform the cyclic sequence 
% $$(1,2), \dots (1,t), (2,1), (2,3), \dots, (2,t), (3,1), \dots, (3,t), \dots, (t,1), \dots, (t,t-1)$$ into
% $$(1,2), (2,1), (1,3), (3,1), \dots, (1,t), (t,1), (2,3), (3,2), \dots, (t-1, t), (t,t-1)$$
% by swapping adjacent elements.  For each $i \in [s]$, let $(a_i,b_i), (c_i,d_i) \in I$ be the adjacent elements swapped at step $i$.   

Let $C^0$ and $C^{-1}$ be the cycles in $L$ corresponding to the cycles $C_{\lambda}$ and $C_{\lambda-1}$ in $H$, respectively. Note that $C^0$ is a tight cycle of $L$.  Since $\ell_{\lambda-1} \ge t$ and $4d_{\lambda-1} \ge t$, there exist distinct vertices $v_1, \dots, v_t \in V(C^{-1})$ such that each $v_i$ has at least $t-1$ neighbours in $V(C^0)$.  
Let $V^0:=\{v_{i,j}^0 : i,j \in [t], i \neq j\} \subseteq V(C^0)$ be such that for all distinct $i,j \in [t]$, $v_i$ is adjacent to $v_{i,j}^0$ in $L$, and 
the cyclic order of $V^0$ on $C^0$ is $$v_{1,2}^0, \dots ,v_{1,t}^0, v_{2,1}^0, v_{2,3}^0, \dots, v_{2,t}^0, v_{3,1}^0, \dots, v_{3,t}^0, \dots, v_{t,1}^0, \dots, v_{t,t-1}^0.$$

For each $k \in [s]$, let $L^k$ be the subgraph of $L$ between $C^{k-1}$ and $C^k$. Note that $|V(C^0)| =\ell_\lambda \geq 2t(t-1)+4$.  Therefore, by~\Cref{claim:switchpaths}, for each $k\in [s]$, there exist $V^k:=\{v_{i,j}^k : i,j \in [t], i \neq j\} \subseteq V(C^k)$, a jump $J^k$ of $L$, and a collection $\mathcal{P}^k$ of $t(t-1)$ $C^{k-1}$-$C^{k}$ paths in $L^k \cup J^k$ such that 
\begin{itemize}
    \item both ends of $J^k$ are contained in $L^k$,
    \item for all $v_{i,j}^{k-1} \in V^{k-1} \setminus \{v_{a_i,b_i}^{k-1}, v_{c_i,d_i}^{k-1}\}$, there exists a path in $\mathcal{P}^k$ with ends $v_{i,j}^{k-1}$ and $v_{i,j}^{k}$,  
    \item there is a path in $\mathcal{P}^k$ with ends $v_{a_i,b_i}^{k-1}$ and $v_{c_i,d_i}^{k}$, and
    \item there is a path in $\mathcal{P}^k$ with ends $v_{c_i,d_i}^{k-1}$ and $v_{a_i,b_i}^{k}$.
\end{itemize}

Letting $A:=L[V^0 \cup \{v_i : i \in [t]\}]$, we have that $A \cup \bigcup_{i \in [s]} \mathcal{P}^i \cup C^s$ contains a subdivision of $K_t$ with branch vertices $v_1, \dots, v_t$.  
\end{proof}

\note{Material about $k$-apex graphs was removed here, in response to referee comment.}

%%%%%%%%%%%%%%%%%%%%%%%%%%%%%%%%%%%%
\subsection{Non-Existence of Universal Graphs}
\label{ExcludingSubdivision}

This section gives a general lemma that can be used to show that various graph classes have no universal element. For a graph $G$, let \defn{$G^{+v}$} be the graph obtained from $G$ by adding a new vertex $v$ adjacent to every vertex in $G$. 

\note{We previously proved a lemma here about excluded subgraphs. But it did not imply our claim about chordal graphs. We now prove a slightly different lemma about hereditary classes that works for all our examples.}

\begin{lem}
\label{NonUniversal}
Let $\FF$ be a hereditary class of graphs such that $K_{\aleph_0}\not\in\FF$, and $G^{+v}\in\FF$ for every $G\in\FF$. Then there is no universal graph for $\FF$. 
\end{lem}

\begin{proof}
Suppose for the sake of contradiction that $U$ is a universal graph for $\FF$. 
So $U\in\FF$ and $U^{+v}\in\FF$ by assumption. Since $U$ is universal for $\FF$, there is a subgraph $U_1$ of $U$ isomorphic to $U^{+v}$. Let $v_1$ be the vertex in $U_1$ corresponding to $v$. Then $U_1-v_1$ contains a subgraph
isomorphic to $U$. Since $U$ is universal for $\FF$, 
$U_1-v_1$ contains a subgraph $U_2$ isomorphic to $U^{+v}$. 
Let $v_2$ be the vertex in $U_2$ corresponding to $v$. 
Repeat this argument to obtain a sequence $v_1,v_2, \ldots$ of vertices in $U$ inducing $K_{\aleph_0}$. 
Since $\FF$ is hereditary and $U\in\FF$, we have
$K_{\aleph_0}\in\FF$, which is the desired contradiction.
\end{proof}

\cref{NonUniversal} implies the following classes have no universal element:
\begin{itemize}
\item Chordal graphs containing no $K_{\aleph_0}$.
\item Graphs with no $K_{\aleph_0}$ minor (this was proved in \cite{DHV85}).
\item Graphs containing no $K_{\aleph_0}$ subdivision (this was left open in \cite{DHV85}).
\item Graphs in which each subgraph has a vertex of finite degree.
\item Graphs that do not contain infinitely many pairwise disjoint one-way infinite paths.
\item For any fixed $d\in\NN$, graphs that do not contain $d$ pairwise disjoint one-way infinite paths.
\item Graphs containing no tree where all the vertices, except possibly one, have degree $d$.
\item Graphs containing no subdivision of the $d$-regular tree:

\begin{proof}
Say $G$ is such a graph, but $G^{+v}$ contains a subdivision $S$ of the $d$-regular tree $T$. So $v\in V(S)$ and $S-v$ has a subgraph  $T_0$ that is a subdivision of $T$, except for one vertex $u$ of degree $d-1$. Let $x,y$ be distinct neighbours of $u$ in $T_0$. Taking the $xy$-path in $S$ through $u$ we obtain a subdivision of $T$ in $T-v$, which is in $G$. This contradiction shows that $G^{+v}$ is in our class.
\end{proof}
\end{itemize}

%%%%%%%%%%%%%%%%%%%%%%%%%%%
\section{Universality for Trees and Treewidth}
\label{TreesTreewidth}

This section proves universality results for trees and graphs of given treewidth. While such results are well known, our construction is novel and leads to strengthenings in terms of orientation- and labelling-preserving isomorphisms, which are a key tool used in our constructions for graphs defined by a tree-decomposition (\cref{DefinedByTreeDecomposition}), for $K_t$-minor-free graphs (\cref{ExcludedMinorNoSubdiv}), and for locally finite graphs (\cref{LocallyFiniteGraphs}).

\subsection{Universal Trees}
\label{UniversalTrees}

The following universal graph for (countable) trees is folklore. Let $\T$ be the graph with 
\begin{align*}
V(\T) &:=\{ (x_1,\dots,x_n) : n\in\NN_0, x_1,\dots,x_n \in\NN \} \text{ and}\\
E(\T) &:= \{ (x_1,\dots,x_n)(x_1,\dots,x_n,x_{n+1}) : n\in\NN_0, x_1,\dots,x_n,x_{n+1}\in\NN \}.
\end{align*}

\begin{thm}
\label{UniversalTree}
$\T$ is universal for the class of trees.
\end{thm}

\begin{proof}
Consider $\T$ to be rooted at vertex $()$. 
From each vertex $(x_1,\dots,x_n)$ in $\T$ there is a unique path 
$(x_1,\dots,x_n),(x_1,\dots,x_{n-1}),(x_1,\dots,x_{n-2}),\dots,(x_1),()$ to the root. 
Thus $\T$ is a tree. 
Since $V(\T)$ is a countable union of countable sets, $\T$ is countable. 
We now show that every countable tree $X$ is isomorphic to a subtree of $\T$.
Root $X$ at an arbitrary vertex $r$. Let $\ell:V(X)\to\NN$ be an injective
function, which exists since $X$ is countable. For each non-root vertex $w$ of
$X$, if $(r,w_1,w_2,\dots,w_n)$ is the path from $r$ to $w=w_n$ in $X$, then  let
$\phi(w):=(\ell(w_1),\ell(w_2),\dots,\ell(w_n))$. Finally, let $\phi(r):=()$.
Then $\phi$ is an isomorphism from $X$ to a subtree of $\TT$. \end{proof}

A \defn{labelling} of a graph $G$ is a function $f:V(G)\cup E(G)\to\NN$.
We now prove a strengthening of \cref{UniversalTree}, which will be used in \cref{DefinedByTreeDecomposition}. 

\begin{lem}
\label{UniversalTreePreserving}
There is a labelling of $\TT$ such that for every 1-orientation of $\TT$, for every rooted tree $T$, and for every labelling of $T$ there is a label-preserving orientation-preserving isomorphism from $T$ to a subtree of $\TT$. 
\end{lem}

\begin{proof}
Since every vertex of $\TT$ has infinite degree, there is a labelling of $\TT$ such that for every vertex $v$ of $\TT$ and for all $\alpha,\beta\in\NN$ there are infinitely many neighbours $w$ of $v$ for which $w$ is labelled $\alpha$ and $vw$ is labelled $\beta$. Fix any 1-orientation of $\TT$. Now, for every vertex $v$ of $\TT$ and for all $\alpha,\beta\in\NN$ there are infinitely many children $w$ of $v$ for which $w$ is labelled $\alpha$ and $vw$ is labelled $\beta$. 

Let $T$ be a tree rooted at $r$, and let $f$ be any labelling of $T$. We now construct a label-preserving orientation-preserving isomorphism $\phi$ from $T$ to a subtree of $\TT$. Let $\phi(r)$ be any vertex in $\TT$ labelled $f(r)$. Consider the vertices of $T$ in order of non-decreasing distance from $r$. Let $v$ be a vertex of $T$ such that $\phi(v)$ is already defined, but $\phi(w)$ is undefined for every child $w$ of $v$. Let $W$ be the set of children of $v$ in $T$. For all $\alpha,\beta\in\NN$ there are infinitely many children $z$ of $\phi(v)$ for which $z$ is labelled $\alpha$ and $\phi(v)z$ is labelled $\beta$. So $\phi$ can be extended to $W$ so that for each $w\in W$, we have $f(w)$ equals the label assigned to $\phi(w)$ in $\TT$, and $f(vw)$ equals the label assigned to $\phi(vw)$ in $\TT$. Hence $\phi$ is a label-preserving orientation-preserving isomorphism from $T$ to a subtree of $\TT$.
\end{proof}

As an aside, we show that \cref{UniversalTreePreserving} is not true for (unrooted) 1-oriented trees $T$. Suppose that there is a labelling and 1-orientation of the universal tree $\TT$ such that for every labelled and 1-oriented tree $T$, there is a label-preserving orientation-preserving isomorphism from $T$ to a subtree of $\TT$. In particular, for every labelled backward 1-way infinite path $P$, there is a label-preserving orientation-preserving isomorphism from $P$ to $\TT$. There are uncountably many labelled backward 1-way infinite paths (since at each vertex there are at least two choices of label). However, for each vertex $v$ of $\TT$ there is exactly one backward 1-way infinite path starting at $v$. Hence some labelling of the backward 1-way infinite path does not appear in $\TT$. 

To complete this subsection, we define a universal tree of given maximum degree. Let $\TT^{(d)}$ be the graph with 
\begin{align*}
V(\TT^{(d)}) & := \{ (x_1,\dots,x_n): n\in\NN_0,\, x_1,\dots,x_n\in[d], \,
\forall i\in[n-1] \, x_i\neq x_{i+1} \}\\
E(\TT^{(d)}) & := \{ (x_1,\dots,x_n)(x_1,\dots,x_n,x_{n+1}):\\
& \qquad n\in\NN_0,\, x_1,\dots,x_n,x_{n+1}\in[d], \, \forall i\in[n]\, x_i\neq x_{i+1}\}.
\end{align*}
A proof analogous to that of \cref{UniversalTree} shows that $\TT^{(d)}$ is a tree, and by definition, $\TT^{(d)}$ is $d$-regular. Indeed, up to isomorphism, $\TT^{(d)}$ is the unique $d$-regular tree, sometimes called the Bethe lattice or infinite Cayley tree \citep{Bethe35,Ostilli12}. This tree is used in \cref{SimpleTreewidth} as the basis for the definition of a universal graph of given simple treewidth.

%%%%%%%%%%%%%%%%%%%%%%%%%%%%%%%%%%%%%%%%%%%%%%%%%%%%%%%%%%%%%%%%
\subsection{Universality for Treewidth}
\label{Treewidth}

Chordal graphs and the theory of simplicial decompositions~\citep{Thom83,Diestel90,Halin84,Halin82,Halin78} can be used to define universal treewidth-$k$ graphs. We take a somewhat different approach that gives a new and explicit definition of a universal treewidth-$k$ graph $\TT_k$. This approach allows us to derive further properties of $\TT_k$, regarding label- and orientation-preserving isomorphisms (\cref{TreewidthUniversalPreservingInduced,TreewidthUniversalPreserving}), that are essential for results in \cref{ExcludedMinorNoSubdiv}. 

The following definitions lead to a new characterisation of graphs with given treewidth. Let $c$ be a colouring of a 1-oriented tree $T$. Let \defn{$\GGG{T}{c}$} be the graph with vertex-set $V(\GGG{T}{c})=V(T)$, where $vw\in E(\GGG{T}{c})$ if and only if $v$ is an ancestor of $w$ in $T$ and $v$ is the only vertex on the $vw$-path in $T$ coloured $c(v)$. 
Where it is clear from the context, we implicitly consider the edges $vw$ of $\GGG{T}{c}$ to be oriented from the ancestor $v$ to the descendent $w$. 

\note{Previously \cref{TcSimpOriented,FindSpanningTree} were combined.}

\begin{lem}
\label{TcSimpOriented}
For every 1-oriented tree $T$ and $(k+1)$-colouring $c$ of $T$, the graph $\GGG{T}{c}$ is chordal containing no $K_{k+2}$ subgraph, and is simplicially $k$-oriented. 
\end{lem}

\begin{proof}
By construction, each vertex $w$ has in-degree at most $k$ in $\GGG{T}{c}$, and the in-neighbourhood of $w$ in $\GGG{T}{c}$ lies on a single directed path in $T$ that ends at $w$. Consider two directed edges $uw$ and $vw$ in $\GGG{T}{c}$. So both $u$ and $v$ are ancestors of $w$ in $T$, and $c(u)\neq c(v)$. Without loss of generality, $u$ is an ancestor of $v$. Since $uw$ is an edge, $u$ is the only vertex on the $uw$-path in $T$ coloured $c(u)$. In particular, $u$ is the only vertex on the $uv$-path in $T$ coloured $c(u)$. So $uv$ is an edge of $\GGG{T}{c}$. Hence the in-neighbourhood of $w$ in $\GGG{T}{c}$ is a clique, and $\GGG{T}{c}$ is simplicially $k$-oriented. By \cref{ChordalCharacterisation}, $\GGG{T}{c}$ is chordal. Each vertex in $\GGG{T}{c}$ has in-degree at most $k$. Thus $\GGG{T}{c}$ contains no $K_{k+2}$. 
\end{proof}

The following converse to \cref{TcSimpOriented} is the key place where the well-founded property is used.

\note{Added the well-founded assumption to \cref{FindSpanningTree} and revised the proof.} 

\begin{lem}
\label{FindSpanningTree}
For every chordal graph $G$ containing no $K_{k+2}$ subgraph, and for every well-founded $k$-orientation of $G$, there is a 1-oriented rooted spanning tree $T$ of $G$, such that $G\subseteq \GGG{T}{c}$ for any $(k+1)$-colouring $c$ of $T$. 
\end{lem}

\begin{proof}
Let $G_1, G_2, \dots$ be the connected components of $G$.  By~\cref{WellFoundedRooted}, for each $i \in \mathbb{N}$, the well-founded $k$-orientation of $G_i$ has a root $r_i$.  For each $i \geq 2$ add an arc $(u,r_i)$, where $u$ is an arbitrary vertex of $V(G_{i-1})$. Note that this new graph is chordal, has no $K_{k+2}$ subgraph, and the newly added arcs extend the well-founded $k$-orientation of $G$.  Thus, we may assume that $G$ is connected. 

Let $Z$ be the set of oriented edges $vw$ of $G$ for which there is no vertex $x$ such that $vxw$ is an oriented path in $G$. Let $T$ be the spanning subgraph of $G$ with edge-set $Z$ (ignoring the edge orientation). We claim that $T$ is a 1-oriented rooted spanning tree $T$ of $G$.

Suppose that some vertex $v$ has in-degree at least 2 in $T$. Let $xv$ and $yv$ be two incoming edges incident to $v$ in $T$. Since the in-neighbourhood of $v$ is a clique in $G$, without loss of generality, $xy$ is an oriented edge in $G$. Thus $xyv$ is an oriented path in $G$, implying that $xv$ is not in $Z$. This contradiction shows that $T$ has in-degree at most 1. Since an acyclically oriented  cycle contains a vertex with in-degree 2, $T$ contains no cycle. 

For each vertex $w$ of $G$, let $G_w$ be the subgraph of $G$ induced by those vertices $v$ such that there is a directed $vw$-path in $G$. Since $G$ has no backward oriented 1-way infinite path, and since every vertex has indegree at most $k$, $G_w$ is finite by \cref{Konig}.

Now consider any directed edge $vw$ in $G$. Let $P$ be the longest $vw$-path in $G_w$, which is well-defined since $G_w$ is finite and $vw$ itself is such a path. If some directed edge $rs$ is in $P$, but not in $Z$, then there exists a vertex $x$ such that $rxs$ is a directed path in $G$. By acyclicity, $x\not\in V(P)$. Since $s$ is in $G_w$, so is $x$. Let $P'$ be the path obtained from $P$ by replacing the edge $rs$ by the path $rxs$. So $P'$ is a $vw$-path in $G_w$ that is longer than $P$, which contradicts the choice of $P$. Hence every edge of $P$ is in $Z$. In short, for every directed edge $vw$ in $G$, there is a $vw$-path in $Z$. 

We now show that $T$ is connected. Consider distinct vertices $a$ and $b$ in $G$. Since $G$ is connected, there is an $ab$-path $P$ in $G$, ignoring edge orientations. For each directed edge $vw\in E(P)$, there is a $vw$-path in $Z$. Hence $Z$ is connected (ignoring edge orientations), and $T$ is a spanning tree. By~\cref{WellFoundedRooted}, $G$ has a unique root $r$, so we may root $T$ at $r$.  

Suppose for the sake of contradiction that for some oriented edge $vw$ of $G$, there is no directed $vw$-path in $T$. Let $vw$ be such an edge that minimises the distance between $v$ and $w$ in $T$ (ignoring edge orientations). Let $P$ be the $vw$-path in $T$ (ignoring orientations). Let $x$ be the least common ancestor of $v$ and $w$ in $T$ (which exists since $T$ is 1-oriented). So $P$ is oriented from $x$ to $v$ and from $x$ to $w$ (and $x\neq v$ and $x\neq w$). Let $y$ be the neighbour of $w$ on the $vw$-path in $T$. Since the in-neighbourhood of $w$ is a clique, $vy$ or $yv$ is an edge of $G$, which contradicts the choice of $vw$. Thus for every oriented edge $vw$ of $G$, there is a directed $vw$-path in $T$. 

Consider any oriented path $v_1,v_2,\dots,v_p$ in $G$, where $v_1v_p$ is an edge of $G$. Since the in-neighbourhood of $v_p$ is a clique, $v_1v_{p-1}$ is an edge of $G$. Since the in-neighbourhood of $v_{p-1}$ is a clique, $v_1v_{p-2}$ is an edge of $G$. By induction, $v_1v_i$ is an edge of $G$, for each $i\in[2,p]$. 

Let $c$ be any $(k+1)$-colouring of $G$, which exists since chordal graphs are perfect and by the de~Bruijn--Erd\H{o}s Theorem~\citep{dBE51}. Consider an oriented edge $vw$ of $G$. As shown above there is a directed path $P$ from $v$ to $w$ in $T$, and $v$ is adjacent to every vertex in $P$. Thus $c(v)\neq c(x)$ for every vertex $x$ in $V(P) \setminus \{v\}$. By construction, $vw\in E(\GGG{T}{c})$. Hence $G\subseteq \GGG{T}{c}$.
\end{proof}

\note{We added the following paragraph.}

Note that \cref{FindSpanningTree} is false without the well-founded assumption, even in the $k=2$ case. As an example, let $G$ be obtained from a 2-way infinite path $P=(\dots,v_{-2},v_{-1},v_0,v_1,v_2,\dots)$ by adding one vertex $x$ adjacent to every vertex in $P$. So $G$ is chordal with no $K_4$ subgraph. Orient each edge $v_iv_{i+1}$ and $xv_i$. This orientation is rooted at $x$ but not well-founded. Note that $G$ has a unique 3-colouring (up to colour permutations), where $c(x)=2$ and $c(v_i)=i\bmod{2}$. Suppose there is a 1-oriented spanning tree $T$ of $G$ with $G\subseteq \GGG{T}{c}$. At least one edge $xv_i$ is in $T$. Thus $v_{i-1}v_i \not\in E(T)$, as otherwise $v_i$ would have indegree 2 in $T$. Since $v_{i-1}v_i$ is an edge of $\GGG{T}{c}$ but not of $T$, there is a directed path from $v_{i-1}$ to $v_i$ not using the edge $v_{i-1}v_i$, but in $G$ there is no such path. Hence there is no 1-oriented spanning tree $T$ of $G$ with $G\subseteq \GGG{T}{c}$.

\note{We slightly strengthened (b) to have a given root vertex (which we need later).}

\begin{lem}
\label{TreewidthCharacterisation}
The following are equivalent for a graph $G$ and $k\in\NN$:
\begin{enumerate}[(a)]
\item $G$ has treewidth at most $k$, 
\item for every vertex $v$ of $G$ there is a tree $T$ rooted at $v$ and there is a $(k+1)$-colouring $c$ of $T$ such that $V(T)=V(G)$ and $G\subseteq \GGG{T}{c}$, 
\item $G$ is a spanning subgraph of a chordal graph with no $K_{k+2}$ subgraph.
\end{enumerate}
\end{lem}

\begin{proof}
(a) $\Longrightarrow$ (b): 
We are given a graph $G$ with treewidth at most $k$, and a vertex $v\in V(G)$. By \cref{StandardTreewidth} there is a tree-decomposition $(B_x:x\in V(T))$ of $G$ such that:
\begin{itemize}
\item $V(T)=V(G)$, and $T$ is rooted at $v$ with $B_v=\{v\}$, 
\item $w\in B_w$ for each vertex $w\in V(T)=V(G)$,
\item  for every edge $vw$ of $G$, $v$ is an ancestor of $w$ or $w$ is an ancestor of $v$ in $T$, 
\item a tree-decomposition $(B_x:x\in V(T))$ of $G$ with width at most $k$, and 
\item a $(k+1)$-colouring $c$ of $G$ such that any two vertices in a common bag $B_x$ are assigned distinct colours. 
\end{itemize}
We now show that $G\subseteq \GGG{T}{c}$. By construction, $V(G) = V(\GGG{T}{c})$. Consider an edge $vw\in E(G)$. Then $c(v)\neq c(w)$. Without loss of generality, $w$ is a descendant of $v$. Consider some vertex $x$ on the $vw$-path in $T$. Since $v\in B_v\cap B_w$, it follows that $v$ is in $B_x$, implying $c(v)\neq c(x)$. Hence $vw\in E(\GGG{T}{c})$ by the definition of $E(\GGG{T}{c})$. Therefore $G\subseteq \GGG{T}{c}$. 

(b) $\Longrightarrow$ (c): Assume $T$ is a rooted tree and $c$ is a $(k+1)$-colouring of $T$ such that $V(T)=V(G)$ and $G\subseteq \GGG{T}{c}$. By \cref{TcSimpOriented}, $\GGG{T}{c}$ is a chordal graph with no $K_{k+2}$ subgraph containing $G$ as a spanning subgraph, as desired. (Note that this step does not require $T$ to be rooted, it only requires that $T$ is 1-oriented.) 

(c) $\Longrightarrow$ (a): Assume $G$ is a spanning subgraph of a chordal graph $G'$ with no $K_{k+2}$ subgraph. By \cref{ChordalCharacterisation}, $G'$ has a tree-decomposition in which each bag is a clique, and therefore of size at most $k+1$. The same tree-decomposition is a tree-decomposition of $G$. So $\tw(G)\leq k$.
\end{proof}

We now define a universal treewidth-$k$ graph. Let $\T$ be the universal tree defined in \cref{UniversalTree} rooted at vertex $()$. For $k\in\NN$, let $\TT_k := \GGG{\T}{c}$, where $c$ is a colouring of $\T$ with colour-set $[0,k]$, where every vertex coloured $i\in[0,k]$ has infinitely many children of each colour in $[0,k]\setminus\{i\}$. Note that the orientation of $\TT_k$ is well-founded. 

\begin{thm}
\label{TreewidthUniversal}
$\TT_k$ is universal for the class of graphs with treewidth at most $k$.
\end{thm}

\begin{proof}
$\GGG{\T}{c}$ has treewidth at most $k$ by \cref{TreewidthCharacterisation} and since $c$ uses $k+1$ colours. Conversely, let $G$ be a graph with treewidth at most $k$. By \cref{TreewidthCharacterisation}, there is a tree $T$ rooted at some vertex $r$, and there is a $(k+1)$-colouring $\alpha:V(T)\to[0,k]$ such that $V(T)=V(G)$ and $G\subseteq \GGG{T}{\alpha}$. By \cref{UniversalTreePreserving} there is a colour-preserving orientation-preserving isomorphism from $T$ (coloured by $\alpha$) to a subtree of $\TT$ (coloured by $c$). Thus $\GGG{T}{\alpha}$ is contained in $\GGG{\T}{c}$, and $G$ is contained in $\GGG{\T}{c}$.
\end{proof}

We now prove two strengthenings of \cref{TreewidthUniversal} that are used in \cref{ExcludedMinorNoSubdiv}. The first is an analogue of \cref{UniversalTreePreserving}. 

\note{The well-founded assumption has been introduced to the next two lemmas, and the first has been strengthened to give an induced subgraph, which we need later.}

\begin{lem}
\label{TreewidthUniversalPreservingInduced}
For each $k\in\NN$, there is a labelling and a well-founded $k$-orientation of $\TT_k$ such that for every rooted tree $T$  and every $(k+1)$-colouring $c$ of $T$, if $G=\GGG{T}{c}$ (which is implicitly $k$-oriented), then for every labelling of $G$, there is a label-preserving orientation-preserving isomorphism from $G$ to an induced subgraph of $\TT_k$. 
\end{lem}

\begin{proof}
Recall that $\TT$ is the universal tree rooted at vertex $()$, and that $\TT_k$ is the universal treewidth-$k$ graph with $V(\TT_k)=V(\TT)$. We may orient the edges of $E(\TT_k)$ such that for each oriented edge $vw$, $v$ is an ancestor of $w$ in $\TT$, and each vertex of $\TT_k$ has in-degree at most $k$. By definition, $\TT_k$ has a colouring with colours $[k+1]$ such that each vertex $v$ of $\TT$ has infinitely many children of each colour (distinct from the colour assigned to $v$). Label each vertex of $\TT_k$ such that for each vertex $v$ of $\TT$, for each colour $\alpha$ distinct from the colour assigned to $v$, and for each $\beta\in \NN$, there are infinitely many children of $v$ coloured $\alpha$ and labelled $\beta$. This is possible since $v$ has infinitely many children coloured $\alpha$. Consider each vertex $w$ of $\TT_k$. Let $v_1w,\dots,v_pw$ be the edges of $\TT_k$ incoming to $w$. So $p\leq k$. Let $W$ be the siblings of $w$ in $\TT$ assigned the same colour and the same label as $w$. So $W$ is infinite. Since every vertex $x\in W$ has the same colour as $w$ and the same parent as $w$ in $\TT$, by the construction of $\TT_k$, we have that $v_1x,\dots,v_px$ are the edges of $\TT_k$ incoming to $x$. Since $\NN^p$ is countable, there is a labelling of the edges $v_ix$ (where $i\in[p]$ and $x\in W$), such that for each $(\gamma_1,\dots,\gamma_p)\in \NN^p$ there are infinitely many vertices $x\in W$, such that the edge $v_ix$ of $\TT_k$ is labelled $\gamma_i$ for each $i\in[p]$. Call this property $(\star)$. 

Now consider the given tree $T$ rooted at $r$. Let $c$ be the given $(k+1)$-colouring of $T$. We may assume that the colour-set is $[k+1]$. The orientation of $T$ determines a $k$-orientation of $G:=\GGG{T}{c}$. Let $f$ be a given labelling of $G$. Our goal is to show that there is a label-preserving orientation-preserving colour-preserving isomorphism $\phi$ from $T$ to a subtree $\phi(T)$ of $\TT$, such that $\phi$ is also a label-preserving orientation-preserving colour-preserving isomorphism from $G$ to the induced subgraph $\TT_k[\phi(T)]$. We define $\phi(v)$ in order of non-decreasing $\dist_T(r,v)$. Let $\phi(r)$ be any vertex of $\TT_k$ labelled $f(r)$. Let $v$ be a vertex of $T$ for which $\phi(v)$ is already defined, but $\phi$ is defined for no child of $v$. For $\alpha\in[0,k]$ and $\beta\in\NN$, let $W_{\alpha,\beta}$ be the set of children $w$ of $v$ in $T$, such that $c(w)=\alpha$ and $f(w)=\beta$. Consider such a $w\in W_{\alpha,\beta}$. Let $v_1w,\dots,v_pw$ be the edges of $G$ incoming to $w$. So $p\leq k$. By the choice of $v$, we have that $\phi(v_1),\dots,\phi(v_p)$ are already defined. Since every vertex $x\in W_{\alpha,\beta}$ has the same colour as $w$ and the same parent as $w$ in $T$, by the construction of $G$, we have that $v_1x,\dots,v_px$ are the edges of $G$ incoming to $x$. For $(\gamma_1,\dots,\gamma_p)\in \NN^p$, let $W_{\alpha,\beta,\gamma_1,\dots,\gamma_p}$ be the set of vertices $x\in W_{\alpha,\beta}$ such that $f(v_ix)=\gamma_i$ for each $i\in[p]$. By property $(\star)$, there are infinitely many children $z$ of $\phi(v)$ in $\TT$, such that $z$ is coloured $\alpha$ and labelled $\beta$, and $\phi(v_i)z$ is labelled $\gamma_i$ for each $i\in[p]$. Injectively map $W_{\alpha,\beta,\gamma_1,\dots,\gamma_p}$ to these vertices $z$ under $\phi$. Repeating this step, $\phi$ is a label-preserving orientation-preserving colour-preserving isomorphism from $T$ to a subtree of $\TT$, such that $\phi$ is also a label-preserving orientation-preserving colour-preserving isomorphism from $G$ to the induced subgraph $\TT_k[\phi(T)]$.

It remains to show that $\phi(G)$ is an induced subgraph of $\TT_k$. Consider $\phi(v)\phi(w)\in E(\TT_k)$ where $v,w\in V(G)$. Without loss of generality, $\phi(v)$ is an ancestor of $\phi(w)$ in $\TT$, and $\phi(v)$ is the only vertex on the $\phi(v)\phi(w)$-path in $\TT$ coloured $c(v)$ (since $\phi$ is colour-preserving). The $vw$-path in $T$ is mapped by $\phi$ to the $\phi(v)\phi(w)$-path in $\TT$. So $v$ is the only vertex on the $vw$-path in $T$ coloured $c(v)$ (since $\phi$ is colour-preserving). Thus $vw\in E(G)$. Hence $\phi(G)$ is an induced subgraph of $\TT_k$. 
\end{proof}

\cref{FindSpanningTree,TreewidthUniversalPreservingInduced} immediately imply:

\begin{lem}
\label{TreewidthUniversalPreserving}
For each $k\in\NN$, there is a labelling and a well-founded $k$-orientation of $\TT_k$ such that for every chordal graph $G$ with no $K_{k+2}$ subgraph, for every labelling of $G$ and for every well-founded $k$-orientation of $G$, there is a label-preserving orientation-preserving isomorphism from $G$ to a subgraph of $\TT_k$. 
\end{lem}

%%%%%%%%%%%%%%%%%%%%
\subsection{Graphs Defined by a Tree-Decomposition}
\label{DefinedByTreeDecomposition}

Recall that $\DD(U)$ is the class of graphs that have a tree-decomposition over the graph $U$. We now construct a universal graph for $\DD(U)$.  This result is used to construct a universal graph for graphs of given simple treewidth (\cref{SimpleTreewidth}), as well as to construct a graph that contains every $X$-minor-free graph (\cref{ExcludedMinors}). This is the first place where we see the utility of our results about orientation and label-preserving isomorphisms developed above. 

\begin{lem}
\label{TreeDecompUniversal}
For every graph $U$ containing no $K_{\aleph_0}$ subgraph, there is a universal graph $\widehat{U}$ for $\DD(U)$.
\end{lem}

\begin{proof}
Let $\CC$ be the set of all pairs $((v_1,\dots,v_k),(w_1,\dots,w_k))$, where $\{v_1,\dots,v_k\}$  and $\{w_1,\dots,w_k\}$ are cliques in $U$ for any $k\in\NN_0$. Since $U$ is countable with no $K_{\aleph_0}$ subgraph, $\CC$ is countable. Enumerate $\CC=\{C_1,C_2,\dots\}$.  Let $\TT$ be the universal tree 1-oriented and labelled, such that every vertex is labelled 1, and for every vertex $v$ of $\TT$ and every $i\in\NN$ there are infinitely many children $w$ of $v$ with $vw$ labelled $i$. Let $\widehat{U}$ be the graph obtained from $\TT$ as follows. First, for each vertex $v$ of $\TT$, introduce a copy $U_v$ of $U$ in $\widehat{U}$, where $U_v$ and $U_w$ are disjoint for all distinct $v,w\in V(\TT)$. Second, for each edge $vw\in E(\TT)$, if $vw$ is labelled $i\in\NN$ and $C_i=((v_1,\dots,v_k),(w_1,\dots,w_k))$, then for each $j\in[k]$, identify vertex $v_j$ in $U_v$ with vertex $w_j$ in $U_w$. This defines $\widehat{U}$. 

For each vertex $v\in V(\TT)$, let $\widehat{U}_v$ be the subgraph of $\widehat{U}$ corresponding to the copy $U_v$ (after the above vertex identifications). Then $(V(\widehat{U}_v):v\in V(\TT))$ is a tree-decomposition of $\widehat{U}$, where each torso is isomorphic to $U$. Thus $\widehat{U}$ is in $\DD(U)$. 

We now show that $\widehat{U}$ contains every graph $G$ in $\DD(U)$. So $G$ has a tree-decomposition $(B_x:x\in V(T))$, such that for each node $x\in V(T)$, there is an isomorphism $\phi_x$ from the torso of $x$ to a subgraph of $U$. Root $T$ at an arbitrary node $r$. Label each vertex of $T$ by 1. Label each oriented edge $xy$ of $T$ as follows. Say $B_x\cap B_y=\{z_1,\dots,z_k\}$, which is a clique in the torsos of $x$ and $y$. For $j\in[k]$, let $v_j$ and $w_j$ be the vertices of $U$ such that $\phi_x(z_j)=v_j$ and $\phi_y(z_j)=w_j$. So $\{v_1,\dots,v_k\}$ and $\{w_1,\dots,w_k\}$ are cliques of $U$.  Label $xy$ by $i\in\NN$ so that $C_i= ((v_1,\dots,v_k),(w_1,\dots,w_k))$. By \cref{UniversalTreePreserving}, there is a label-preserving orientation-preserving isomorphism $\phi$ from $T$ to $\TT$. 
By construction, $\bigcup_{v\in V(T)}\widehat{U}_{\phi(v)}$ contains a subgraph isomorphic to $G$.
\end{proof}

\cref{TreeDecompUniversal} can be extended as follows to allow for tree-decompositions with adhesion $k$. The proof is identical to that of \cref{TreeDecompUniversal}, except that in the definition of $\CC$ we only consider cliques with size at most $k$.

\begin{lem}
\label{TreeDecompUniversalAdhesion}
For every graph $U$ containing no $K_{\aleph_0}$ subgraph and for every $k\in\NN$, there is a universal graph $\widehat{U}^k$ for the class $\DD_k(U)$.
\end{lem}

%%%%%%%%%%%%%%%%%%%%%%%%%%%
\subsection{Treewidth Extendability}
\label{TreewidthExtendability}

Recall that a graph class $\Gamma$ is \defn{extendable} if the following property holds for every graph $G$: if every finite subgraph of $G$ is in $\Gamma$, then $G$ is in $\Gamma$. (Recall that graphs are assumed to be countable.)

\citet{Thomas88} proved the following result. 

\begin{lem}[{\protect Countable Tree-Width Compactness Theorem~\citep{Thomas88}}]
\label{TreewidthExtendable}
Treewidth is extendable. 
\end{lem}

See \citep{KT91,Thomassen89} for simpler proofs of \cref{TreewidthExtendable}. The proof of \cref{TreewidthExtendable} due to \citet{Thomassen89} generalises as follows. Define a \defn{forbidden pattern} to be a sequence $\FF=(\FF_k:k\in\NN)$ where $\FF_k$ is any class of finite graphs, such that for every $k\in\NN$, every graph in $\FF_{k+1}$ has a subgraph in $\FF_k$. Define the \hdefn{$\FF$}{width} of a graph $G$ to be the minimum $k\in\NN$ such that $G$ is a subgraph of a chordal graph containing no subgraph in $\FF_k$. If there is no such $k$, then $G$ has \defn{infinite} $\FF$-width. Observe that if $\FF_k=\{K_{k+2}\}$ then $\FF$-width equals treewidth by \cref{TreewidthCharacterisation}. Thus 
\cref{TreewidthExtendable} is implied by the  following generalisation.

\begin{lem}
\label{FwidthExtendable}
$\FF$-width is extendable for every forbidden pattern $\FF$.
\end{lem}

\begin{proof}
Let $G$ be an infinite graph such that every finite subgraph of $G$ has $\FF$-width at most $k$. Our goal is to show that $G$ has $\FF$-width at most $k$. Let $\GG$ be the class of graphs $G'$ such that $V(G')=V(G)$, $E(G)\subseteq E(G')$ and every finite subgraph of $G'$ has $\FF$-width at most $k$. Thus $G\in\GG$. Consider the partial order defined on $\GG$ by inclusion (of the corresponding edge-sets). 

Let $\XX$ be a chain in $\GG$. Let $G'':= \bigcup_{G'\in\XX} G'$. We claim that $G''\in\GG$. Certainly, $V(G'')=V(G)$ and $E(G)\subseteq E(G'')$. If $G''$ contains a finite subgraph $H$ with $\FF$-width greater than $k$, then since $\XX$ is totally ordered by inclusion and $H$ is finite, some graph $G'$ in $\XX$ would contain $H$, which contradicts the definition of $\GG$. Hence $G''$ is an upper bound on $\XX$ in $\GG$. 

By Zorn's Lemma (\cref{Zorn}), $\GG$ contains a maximal element $G_0$. We claim that $G_0$ is chordal. Suppose on the contrary that $G_0$ has a chordless cycle $C$ of length at least 4. For each pair of non-adjacent vertices $x,y\in V(C)$, let $G_{xy}$ be the graph obtained from $G_0$ by adding the edge $xy$. Since $G_0$ is maximal with respect to inclusion, $G_{xy}\not\in\GG$. Since $V(G_{xy})=V(G)$ and $E(G)\subseteq E(G_{xy})$, it must be that $G_{xy}$ contains a finite subgraph $H_{xy}$ with $\FF$-width greater than $k$. Let $H$ be the union of $C$ and all 
graphs $H_{xy}-xy$ taken over all pairs of non-adjacent vertices $x,y\in V(C)$. Since $|C|$ is
finite and each $H_{xy}$ is finite, $H$ is finite. By construction, $H$ is a
subgraph of $G_0$, which is in $\GG$. Thus $H$ has $\FF$-width at most $k$. By
the definition of $\FF$-width, $H$ is a subgraph of a chordal graph $H'$
containing no element of $\FF_k$. Since $H'$ is chordal and 
$C$ is a cycle of $H'$, we have that $H'$ contains some chord,
say $xy$, of $C$. Now $H_{xy}\subseteq H+xy\subseteq H'$. 
Since $H'$ has $\FF$-width at most $k$, so does $H_{xy}$, which is a contradiction. 
Hence $G_0$ is chordal. 

Since $G_0\in\GG$, every finite subgraph of $G_0$ has $\FF$-width at most $k$. In particular, no element of $\FF_k$ is a subgraph of $G_0$ (since elements of $\FF_k$ have $\FF$-width greater than $k$). Hence $G_0$ has $\FF$-width at most $k$. Therefore $G$ has $\FF$-width at most $k$ since $G\subseteq G_0$.
\end{proof}

See \cref{SimpleTreewidth} for another application of \cref{FwidthExtendable}. 

The following result generalises \cref{TreewidthExtendable}, and is a direct corollary of a result of \citet[Theorem~3.9]{KT90}.
 
\begin{lem}[\citep{KT90}] 
\label{TreeDecompExtendable}
For every hereditary and extendable class of graphs  $\Gamma$, the class $\DD(\Gamma)$ is extendable. 
\end{lem}

For a graph $G$ and $c\in\NN_0$, let $\tw_c(G)$ be the minimum integer $k$ such that $\tw(G-S)\leq k$ for some $S\subseteq V(G)$ with $|S|\leq c$. Of course, $\tw(G)=\tw_0(G)$ and $\tw(G)\leq \tw_c(G)+c$. We need the following generalisation of \cref{TreewidthExtendable}.

\begin{lem}
\label{TreewidthApexExtendable}
$\tw_c$ is extendable for every $c\in\NN_0$.
\end{lem}

\cref{TreewidthApexExtendable} follows from the next lemma by induction on $c$ (with \cref{TreewidthExtendable} in the base case). 

\begin{lem}
\label{MonoExtend}
Let $\Gamma$ be a monotone and extendable graph class. Let $\Gamma^+$ be the class of graphs $G$ such that $G\in\Gamma$ or $G-v\in\Gamma$ for some vertex $v$ of $G$. Then $\Gamma^+$ is monotone and extendable. 
\end{lem}

\note{Revised \cref{MonoExtend} according to referee's comment.}

\begin{proof}
Since $\Gamma$ is monotone, so too is $\Gamma^+$. It remains to show that $\Gamma^+$ is extendable. 	Let $G$ be a graph such that every finite subgraph of $G$ is in $\Gamma^+$. Our goal is to show that $G\in\Gamma^+$. Let $v_1,v_2,\dots$ be an arbitrary ordering of $V(G)$. For $i\in\NN$, initialise $$X_i:= \{ (i,j): j\in[i], G[\{v_1,\dots,v_{j-1},v_{j+1},\dots,v_i\}]  \in \Gamma \}.$$	Add $(i,0)$ to $X_i$ if $G[\{v_1,\dots,v_{i}\}]  \in \Gamma$. Let $X$ be the graph with vertex-set $V(X):= \bigcup_{i\in\NN} X_i$ and all edges of the form $(i,j)(i+1,j)$, $(i-1,0)(i,i)$, or $(i,0)(i+1,0)$. We now show that for $i\geq 2$, each vertex $(i,j)\in X_i$ has a neighbour in $X_{i-1}$. If $j=0$ then $(i-1,0)$ is a neighbour of $(i,j)$ in $X_{i-1}$. 	Now suppose that $j\in[i-1]$. Then $G[\{v_1,\dots,v_{j-1},v_{j+1},\dots,v_i\}]  \in \Gamma$. Since $\Gamma$ is monotone, $G[\{v_1,\dots,v_{j-1},v_{j+1},\dots,v_{i-1}\}] \in \Gamma$, implying that $(i-1,j)$ is a neighbour of $(i,j)$ in $X_{i-1}$. 	Finally, if $j=i$, then $(i-1,0)$ is a neighbour of $(i,j)$ in $X_{i-1}$. 	So every vertex in $X_i$ has a neighbour in $X_{i-1}$. By \cref{Konig}, $X$ contains an infinite path  $(1,x_1),(2,x_2),\dots$ where $(i,x_i)\in X_i$. First consider the case that $x_i=0$ for every $i\in\NN$. Then every finite subgraph of $G$ is in $\Gamma$. Since $\Gamma$ is extendable, $G\in\Gamma\subseteq \Gamma^+$, as desired. Now suppose that $x_i\neq 0$ for some $i\in\NN$. Thus, for some $i\in \NN$, we have  $x_1=\dots =x_{i-1}=0$ and $i=x_i=x_{i+1}=x_{i+2}=\cdots$. For every finite subgraph $G'$ of $G-v_i$, there is an integer $n$, such that $G'$ is a subgraph of $G[\{v_1,\dots,v_{i-1},v_{i+1},\dots,v_n\}]$, which is in $\Gamma$. Since $\Gamma$ is 	monotone, $G'\in\Gamma$. Since $\Gamma$ is extendable,  $G-v_i\in\Gamma$. Hence $G\in\Gamma^+$ and $\Gamma^+$ is extendable. 
\end{proof}

%%%%%%%%%%%%%%%%%%%%
\subsection{Simple Treewidth}
\label{SimpleTreewidth}

A tree-decomposition $(B_x:x\in V(T))$ of a graph $G$ is \hdefn{$k$}{simple}, for some $k\in\NN$,  if it has  width  at most $k$, and for every set $S$ of $k$ vertices in $G$, we have $|\{x\in V(T): S\subseteq B_x\}|\leq 2$. The \defn{simple treewidth} of a graph $G$, denoted by \defn{$\stw(G)$}, is the minimum $k\in\NN$ such that $G$ has a $k$-simple tree-decomposition. Simple treewidth appears in several places in the literature under various guises \citep{KU12,KV12,MJP06,Wulf16}. The following facts are well known: A connected finite graph has simple treewidth 1 if and only if it is a path. A finite graph has simple treewidth at most 2 if and only if it is outerplanar. A finite graph has simple treewidth at most 3 if and only if it has treewidth 3 and is planar~\citep{KV12}. The edge-maximal finite  graphs with simple treewidth 3 are ubiquitous objects, called  \defn{planar 3-trees} in structural graph theory and graph drawing~\citep{AP-SJADM96,KV12}, called \defn{stacked polytopes} in polytope theory~\citep{Chen16}, and called \defn{Apollonian networks} in enumerative and random graph theory~\citep{FT14}. It is also known and easily proved that $\tw(G) \leq \stw(G)\leq \tw(G)+1$ for every finite graph $G$ (see \citep{KU12,Wulf16}). A similar proof shows this result for infinite graphs.

Simple treewidth can be characterised in terms of chordal supergraphs. The following lemma for finite graphs is due to \citet{Wulf16}. Let $W_k$ be the graph with vertex-set $\{u,v,w,x_1,\dots,x_k\}$, where each of $\{u,x_1,\dots,x_k\}$, $\{v,x_1,\dots,x_k\}$ and $\{w,x_1,\dots,x_k\}$ are cliques, and $\{u,v,w\}$ is an independent set.

\begin{lem}
\label{SimpleTreewidthChordal}
A graph $G$ has simple treewidth at most $k\in\NN$ if and only if $G$ is a subgraph of a chordal graph containing no $K_{k+2}$ or $W_k$ subgraph.
\end{lem}

\begin{proof}
By \cref{TreeDecompositionClique}, in every tree-decomposition of $W_k$ with width $k$, the cliques $\{u,x_1,\dots,x_k\}$, $\{v,x_1,\dots,x_k\}$ and $\{w,x_1,\dots,x_k\}$ are in distinct bags. Therefore, $\{x_1, \dots, x_k\}$ is contained in at least three bags, and such a tree-decomposition is not $k$-simple. Thus $\stw(W_k)>k$.

Now consider a graph $G$ with  simple treewidth at most $k$. Starting from a $k$-simple tree-decomposition of $G$, let $G'$ be the graph obtained from $G$ by adding an edge between any two non-adjacent vertices in a common bag. We obtain a $k$-simple tree-decomposition of $G'$, implying $\stw(G')\leq k$. By \cref{ChordalCharacterisation}, $G'$ is chordal, and $W_k\not\subseteq G'$ since $\stw(W_k)>k$. By \cref{TreeDecompositionClique}, $K_{k+2}\not\subseteq G'$. 

We now prove that if $G$ is a subgraph of a chordal graph with no $K_{k+2}$ or
$W_k$ subgraph, then $\stw(G)\leq k$. It suffices to prove the result when $G$
is chordal with no $K_{k+2}$ or $W_k$ subgraph. Let $\{X_i:i\in I\}$ be the
maximal cliques of $G$. Let $T$ be the graph with vertex-set $I$, where $ij\in
E(T)$ if and only if $X_i\cap X_j\neq\emptyset$ and $i\neq j$. Since $G$ is
chordal, by \cref{ChordalCharacterisation}, $G$ has a tree-decomposition $(B_a:a\in V(T))$ in which each bag $B_a$ is a clique. We may assume that $B_a\setminus B_b\neq\emptyset$ for all distinct $a,b\in V(T)$. 
Since $K_{k+2}\not\subseteq G$, the width is at most $k$. 
If some set $X$ of $k$ vertices appears in distinct bags $B_a,B_b,B_c$, then $B_a\cup B_b\cup B_c$
induce $W_k$. Thus $(B_a:a\in V(T))$ is a $k$-simple tree-decomposition of $G$, and $\stw(G)\leq k$. 
\end{proof}

By \cref{SimpleTreewidthChordal}, if $\FF_k=\{K_{k+2},W_k\}$ then $\FF$-width equals simple treewidth, and \cref{FwidthExtendable} implies:

\begin{lem}
\label{SimpleTreewidthExtendable}
Simple treewidth is extendable. 
\end{lem}

For $k\in\NN$, define the directed graph $\RR_k$ as follows. 
Recall that $\TT^{(k+1)}$ is the universal $(k+1)$-regular tree defined in \cref{UniversalTrees}. 
Fix an orientation of $\TT^{(k+1)}$ in which each vertex has in-degree 1 and out-degree $k$, which exists by \cref{TreeOrientation}. 

We now construct a particular colouring $c:V(\TT^{(k+1)})\to[0,k]$. Let $Q=(q_i:i\in\ZZ)$ be a 2-way infinite path in $\TT^{(k+1)}$ with every edge in $Q$ oriented from $q_i$ to $q_{i+1}$. We call $Q$ the \defn{central path} of $\TT^{(k+1)}$. Let $c(q_i):= i\bmod{(k+1)}$ for each $i\in\ZZ$. We say $Q$ is coloured `modulo $k+1$'. We now extend this colouring of $Q$ to all of $\TT^{(k+1)}$, such that: 
\begin{enumerate}[(i)] 
\item adjacent vertices are assigned distinct colours, 
\item sibling vertices (that is, with a common parent) are assigned distinct colours.
\end{enumerate}

To do so, consider the vertices $v\in V(\TT^{(k+1)})\setminus V(Q)$ in non-decreasing order of their distance from $Q$. 
Let $u$ be the parent of $v$ in $\TT^{(k+1)}$. 
So $u$ is already coloured. 
Let $w_1,\dots,w_p$ be the siblings of $v$ that are already coloured. 
No other neighbour of $v$ is already coloured. 
Choose $c(v) \in [0,k]\setminus\{c(u),c(w_1),\dots,c(w_p)\}$, which is non-empty since $p\leq k-1$ (since $u$ has out-degree $k$ in $\TT^{(k+1)}$). 
Colour all of $\TT^{(k+1)}$ by this procedure. 
Properties (i) and (ii) are immediate. 
This completes the construction of $c$.

Now define $\RR_k:=\GGG{\TT^{(k+1)}}{c}$. For example, \cref{UniversalOuterplanar,UniversalPlanar} illustrate $\RR_2$ and $\RR_3$. 

\begin{figure}[!b]
\centering
\includegraphics[width=\textwidth]{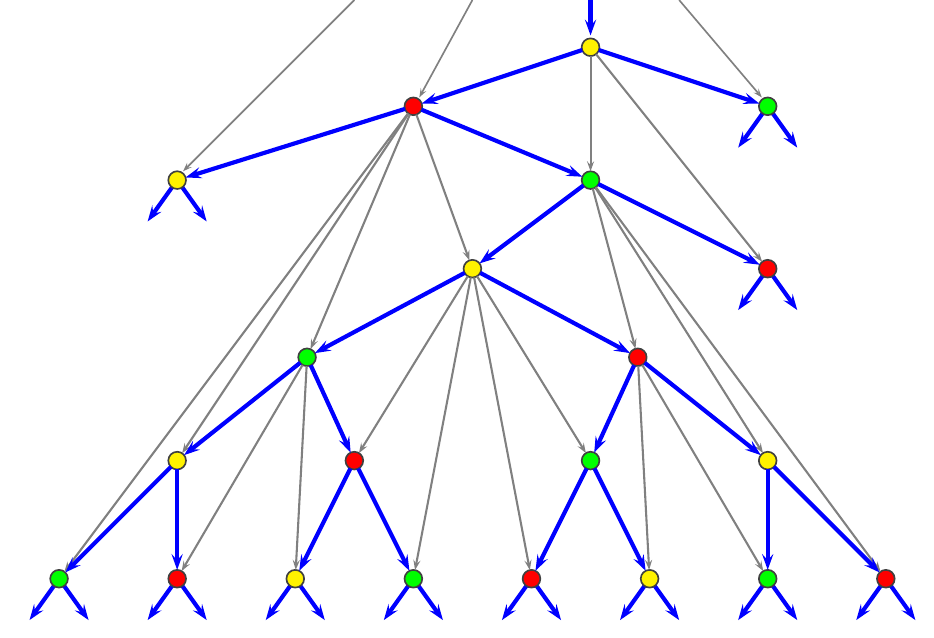}
\caption{Illustration of the outerplanar graph $\RR_2$ with the 1-oriented 3-regular universal tree highlighted.}
\label{UniversalOuterplanar}
\end{figure}

\begin{figure}[!t]
    \centering
    \includegraphics[width=\textwidth]{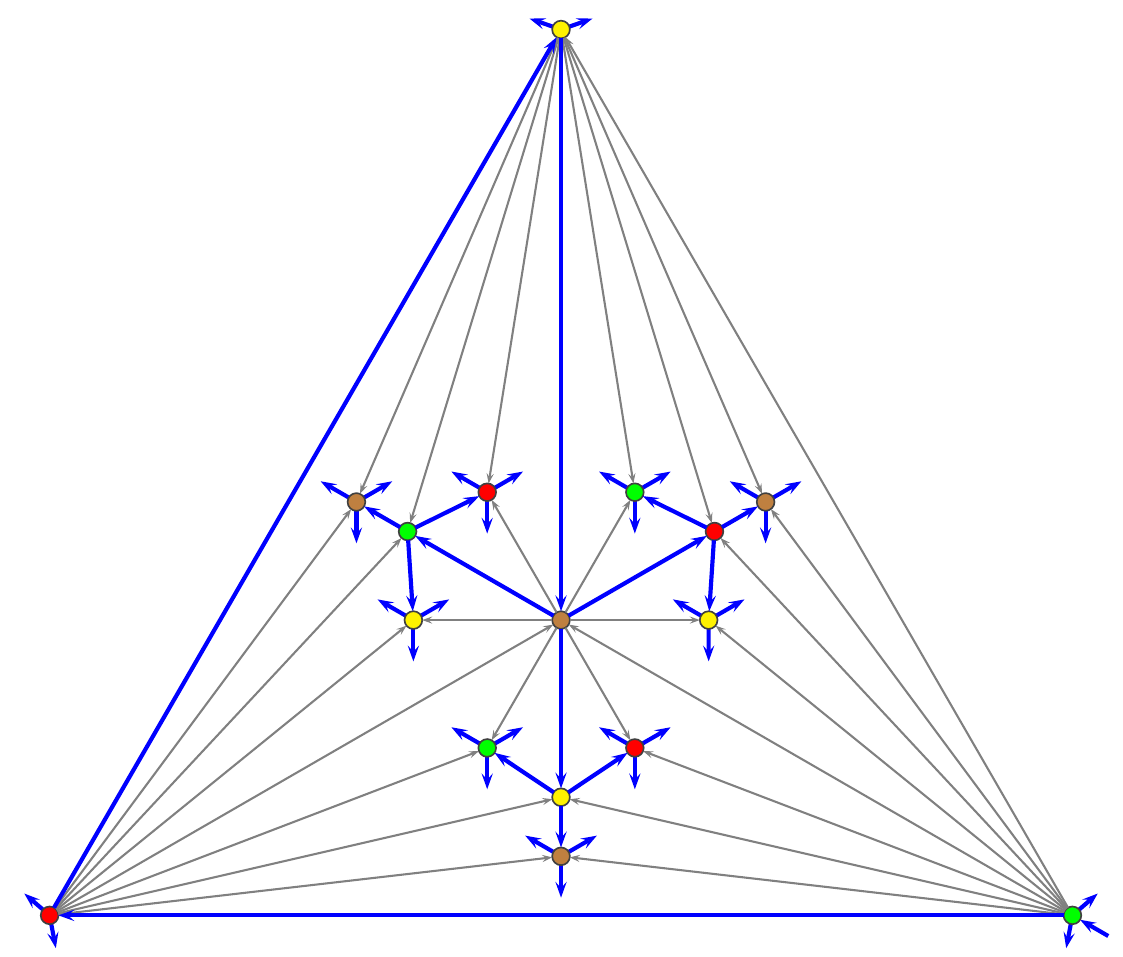}
    \caption{Illustration of the planar graph $\RR_3$ with the 1-oriented 4-regular universal tree highlighted. }
    \label{UniversalPlanar}
\end{figure}

\begin{lem}
\label{SimpleTreewidthRk}
$\RR_k$ has simple treewidth $k$.
\end{lem}

\begin{proof}
Let $\TT^{(k+1)}$ be the oriented tree, let $Q$ be the 2-way infinite path in $\TT^{(k+1)}$,  and let $c$ be the $(k+1)$-colouring used in the definition of $\RR_k$. Note that $c$ is a $(k+1)$-colouring of $\RR_k$. The orientation of $\TT^{(k+1)}$ determines an acyclic orientation of $\RR_k$. For each vertex $x$ of $\TT^{(k+1)}$, let $P_x$ be the path induced by $x$ and its ancestors in $\TT^{(k+1)}$. Thus, $P_x$ and $P_x\cap Q$ are backward oriented 1-way infinite paths. Since $Q$ is coloured modulo $k+1$, all $k+1$ colours appear on $P_x \cap Q$. Thus $x$ has indegree $k$ in $\RR_k$. Let $B_x$ be the closed in-neighbourhood of $x$. So $|B_x|=k+1$.

We now show that $(B_x:x\in V(\TT^{(k+1)}))$ is a $k$-simple tree-decomposition of $\RR_k$. By construction, each edge $vw$ of $\RR_k$ has its ends in $B_w$. In fact, $B_w$ is a clique of size $k+1$ in $\RR_k$, and $c(x) \neq c(y)$ for all distinct $x,y \in B_w$. 
Consider the set of bags that contain a given vertex $u$.  Say $u$ is in the bag $B_w$ for some $w\neq u$. Then $uw\in E(\RR_k)$ and $w$ is a descendant of $u$. Let $v$ be any vertex on the $uw$-path in $\TT^{(k+1)}$. Since $uw\in E(\RR_k)$, we have  $uv\in E(\RR_k)$ and $u$ is also in $B_v$. Hence the sets of bags that contain $u$ form a connected subtree of $\TT^{(k+1)}$. Thus $(B_u:u\in V(\TT^{(k+1)}))$ is a tree-decomposition of $\RR_k$ of width at most $k$. 

We now show that this tree-decomposition is $k$-simple. Since every bag is a clique, it suffices to show that every clique $S$ of $k$ vertices is contained in at most two bags. Let $C$ be the set of colours assigned to vertices in $S$. Since $S$ is a clique, $|C|=k$. Let $\alpha$ be the element of $[0,k]\setminus C$. Let $u$ be the vertex in $S$ that is a sink in the acyclic orientation of $\RR_k[S]$. By construction, $S\subseteq B_u$. Consider a vertex $v\neq u$ such that $S\subseteq B_v$. For each vertex $w\in S\setminus\{u\}$, we have $u\not\in B_w$. Thus $v\not\in S$. Since $B_v$ is a clique, $v$ is coloured $\alpha$. Since $S\subseteq B_v$, $v$ is a descendant of $u$ in $\TT^{(k+1)}$. Suppose that $v$ is not a child of $u$ in $T$. Let $w$ be any internal vertex on the $uv$-path in $T$. Then the colour of $w$ is assigned to some vertex in $S$, which implies that $S\not\subseteq B_v$. Thus $v$ is a child of $u$. By property (ii) there is only one child of $u$ coloured $\alpha$. So $S$ is in at most two bags.

Therefore $(B_u:u\in V(\TT^{(k+1)}))$ is a $k$-simple tree-decomposition of $\RR_k$. Since $\RR_k$ contains $K_{k+1}$, it has simple treewidth $k$. \end{proof}

There are infinite graphs with simple treewidth $k$ that are not subgraphs of $\RR_k$. The first example is the disjoint union of countably many infinite 2-way paths, which has simple treewidth 1, but is not a subgraph of $\RR_1$. Examples for all $k$ are easily constructed. Nevertheless, 
\cref{SimpleTreewidthTreeDecomposition} below characterises graphs with simple treewidth $k$. The proof depends on the following `normalisation' lemma (which is trivial in the finite case). 
In a tree-decomposition $(B_x:x\in V(T))$ of a graph $G$, an  edge $xy\in E(T)$ is \defn{superfluous} if $B_x\subseteq B_y$ or $B_y\subseteq B_x$. 
%Let $T'$ be the tree obtained from $T$ by contracting $xy$ into a new node $z$. Let $B'_z:=B_y$, and let $B'_a:=B_a$ for each $a\in V(T'-z)$. Then $\DD'=(B'_a:a\in V(T'))$ is a tree-decomposition of $G$, said to be obtained from $\DD$ by \defn{bag-contraction}. Note that if $\DD$ has width $k$, then $\DD'$ has width $k$; and if $\DD$ is $k$-simple, then $\DD'$ is $k$-simple. 

\note{We introduced \cref{Superfluous} to avoid the `repeat-this-step-forever' argument in \cref{SimpleTreewidthTreeDecomposition}.}

\begin{lem}
\label{Superfluous}
For any tree-decomposition $\DD$ of a graph $G$ with width $k$, there is a tree-decomposition $\DD'$ of $G$ with width $k$ and with no superfluous edges. Moreover, if $\DD$ is $k$-simple, then $\DD'$ is $k$-simple. 
\end{lem}
\begin{proof}
Say $\DD=(B_x:x\in V(T))$.  Define a partition $\Pi$ of $V(T)$ to be \defn{valid} if for each $X \in \Pi$, $X$ induces a connected subtree of $T$ and there exists $x \in X$ such that $B_x=\bigcup_{y \in X} B_y$.    Let $(\mathcal{V}, \preceq)$ be the poset of valid partitions of $V(T)$, where $\Pi_1 \preceq \Pi_2$ if $\Pi_1$ is a refinement of $\Pi_2$. Note that $\mathcal{V}$ is non-empty since the singletons are a valid partition of $V(T)$. Let $\mathcal{C}$ be a chain in $\mathcal{V}$.  Define a partition $\Pi_{\mathcal{C}}$ where $x,y \in V(T)$ are in the same part if and only they are in the same part of some $\Pi \in \mathcal{C}$.  Note that $\Pi_{\mathcal{C}}$ is a valid partition and hence is an upperbound for $\mathcal{C}$.  By Zorn's lemma (\cref{Zorn}), $(\mathcal{V}, \preceq)$ has a maximal element $\mathcal{M}$.  Let $\DD'$ be the tree-decomposition obtained from $\DD$ by contracting each $P \in \mathcal{M}$ to a single vertex $x_P$ and defining $B_{x_P}:=\bigcup_{y \in P} B_y$. 

By the definition of `valid', for each $P \in \mathcal{M}$ there exists $p \in P$ such that $B_p=B_{x_P}$. Thus, every bag of $\DD'$ is a bag of $\DD$, and every largest bag of $\DD$ is a bag of $\DD'$, so the width of $\DD'$ equals $k$. Towards a contradiction, suppose $x_Px_Q$ is an edge of $\DD$ such that $B_{x_P} \subseteq B_{x_Q}$.  By the definition of `valid', there exists $q \in Q$ such that $B_q=B_{x_Q}$.  Thus, the partition of $V(T)$ obtained from $\mathcal{M}$ by merging $P$ and $Q$ contradicts the maximality of $\mathcal{M}$.  Finally, observe that the contraction operation used above maintains $k$-simplicity. 
\end{proof}

\begin{lem}
\label{SimpleTreewidthTreeDecomposition} 
For $k\in\NN$, a graph has simple treewidth at most $k$ if and only if it has a tree-decomposition over $\RR_k$ with adhesion at most $k-1$.
\end{lem}

\begin{proof}
Let $G$ be a graph with simple treewidth $k$. Let $(B_x:x\in V(T))$ be a $k$-simple tree-decomposition of $G$. We may assume that $B_x$ is a clique  
in $G$ for each node $x\in V(T)$. By \cref{Superfluous}, we may assume that no edge of $T$ is superfluous. 
Say an edge $xy\in E(T)$ is \defn{thick} if $|B_x\cap B_y|=k$, which implies $|B_x|=|B_y|=k+1$ since no edge is superfluous. For each vertex $x\in V(T)$ there are at most $k+1$ possible values for $B_x\cap B_y$ where $xy$ is a thick edge incident to $x$. So if $x$ is incident with $k+2$ thick edges, then $|B_x\cap B_y|=|B_x\cap B_z|$ for distinct $y,z\in N_T(x)$, implying three bags contain $B_x\cap B_y$, which contradicts the $k$-simplicity of the tree-decomposition. Hence each vertex $x\in V(T)$ is incident with at most $k+1$ thick edges. Let $F$ be the forest with $V(F):=V(T)$ consisting only of the thick edges. Hence $F$ has maximum degree at most $k+1$. 

\note{Minor revisions here, addressing referee comments.} 
Let $F_1,F_2,\dots$ be the connected components of $F$. Let $G_i := G [ \bigcup\{ B_x : x\in V(F_i)\}]$ for $i\in\NN$. 
For each vertex $v\in V(G_i)$, initialise $F_{i,v}$ to be the subtree $F_i[ \{ x\in V(F_i): v \in B_x \} ]$. Let $c:V(G)\to[0,k]$ be a $(k+1)$-colouring of $V(G)$ such that vertices in a common bag are assigned distinct colours (which exists since $\chi(G)\leq\tw(G)+1$ using the de~Bruijn--Erd\H{o}s~Theorem~\citep{dBE51}). 

We now show that $G_i$ is contained in $\RR_k$. To do so, we first manipulate $F_i$ so that $V(F_i)=V(G_i)$ and the colouring $c$ of $G_i$ applies to $F_i$. We then show that $G_i \subseteq \GGG{F_i}{c}$ and that $\GGG{F_i}{c}$ is contained in $\RR_k$, implying that $G_i$ is contained in $\RR_k$. We distinguish two cases. 

Case 1. $F_i$ has a node $s$ of degree at most $k$: 
\note{Added more detail here.} 
Orient every edge of $F_i$ away from $s$. Now $F_i$ is 1-oriented, and every node of $F_i$ has out-degree at most $k$. Say $B_s=\{v_1,\dots,v_p\}$ where $p\in[k]$. Add new nodes $s_1,\dots,s_{p-1}$ to $F_i$, where  $(s_1,\dots,s_{p-1},s)$ is a directed path, and let $B_{s_i}:=\{v_1,\dots,v_i\}$ for each $i\in[p-1]$. Observe that $B_x$ is still a clique for each node $x\in V(F)$, and vertices in a common bag are still assigned distinct colours under $c$. Consider $F_i$ to be rooted at $s_1$. Now every node of $F_i$ has out-degree at most $k$ and in-degree 1, except for $s_1$ which has in-degree 0. Note that (a) $|B_{s_1}|=1$, and (b) for every arc $xy$ of $F_i$, we have that $|B_y\setminus B_x|=1$ (since $xy$ is thick, or by construction for arcs $s_i s_{i+1}$). In case (a) rename $s_1$ by the element of $B_{s_1}$. 
In case (b), rename $y$ by the element of $B_y\setminus B_x$. 
For each vertex $v$ of $G_i$ there is a node $y$ with parent $x$ in $F_i$ such that $B_y\setminus B_x=\{v\}$. Thus, after this renaming, $V(F_i)=V(G_i)$ 
Consider the colouring $c$ of $G_i$ to also be a colouring of $F_i$. 
We now show that $G_i \subseteq \GGG{F_i}{c}$ (using a similar argument to the proof of \cref{TreewidthCharacterisation}). By construction, $V(G_i) = V(\GGG{F_i}{c})$. Consider an edge $vw\in E(G_i)$. Then $c(v)\neq c(w)$. Since $v$ and $w$ appear in a common bag, without loss of generality, $w$ is a descendant of $v$ in $F_i$. Consider some vertex $x$ on the $vw$-path in $F_i$. Since $v\in B_v\cap B_w$, it follows that $v$ is in $B_x$, implying $c(v)\neq c(x)$. Hence $vw\in E(\GGG{F_i}{c})$ by the definition of $E(\GGG{F_i}{c})$. Therefore $G_i\subseteq \GGG{F_i}{c}$. 

We now show that $G_i$ is contained in $\RR_k$. Since $G_i\subseteq \GGG{F_i}{c}$, it  suffices to show that there is a colour-preserving orientation-preserving isomorphism $\phi$ from $F_i$ to $\TT^{(k+1)}$, where the colouring of $V(F_i)$ is inherited from the colouring of $G_i$ (since $V(F_i)=V(G_i))$ and the colouring of $V(\TT^{(k+1)})$ is given in the definition of $\RR_k$ above. Map the root $s$ of $F_i$ to $\phi(s)$, where $\phi(s)$ is any vertex of $\TT^{(k+1)}$ with the same colour as $s$.  Now inductively assume that $\phi(t)$ has already been defined for all vertices $t$ of $F_i$ at distance at most $d$ from $s$.  Let $u$ be a vertex of $F_i$ at distance $d+1$ from $s$ and $w$ be the parent of $u$ in $F_i$.  By the definition of $\TT^{(k+1)}$, $\phi(w)$ has a child $u'$ of the same colour as $u$.  Thus, we may define $\phi(u)$ to be $u'$.

Case 2. $F_i$ is $(k+1)$-regular: 
\note{Added more detail here. } 
Assign each edge $xy$ of $F_i$ the unique colour not present on the vertices of $B_x\cap B_y$, which is well-defined since there are exactly $k$ vertices in $B_x\cap B_y$ and they are assigned distinct colours. Since the tree-decomposition is $k$-simple, $F_i$ is now properly $(k+1)$-edge coloured, and every colour is present at each node of $F_i$. Thus there exists an infinite 2-way path 
$P=(v_j:j\in\ZZ)$ in $F_i$, where each edge $v_jv_{j+1}$ in $P$ is coloured $j\bmod{(k+1)}$. Orient each edge $v_jv_{j+1}$ of $P$ from $v_j$ to $v_{j+1}$ and orient each edge of $F_i-E(P)$ away from $P$. By construction, each vertex of $F_i$ has in-degree 1 and out-degree $k$. 
For each vertex $v$ of $G_i$ coloured $\alpha\in[0,k]$, by construction, each edge $xy$ of $F_{i,v}$ is not coloured $\alpha$. Since $F_{i,v}$ intersects $P$ in a subpath, $F_{i,v}$ intersects at most $k$ edges in $P$. In particular, some edge of $P$ is not in $F_{i,v}$. 
If $F_{i,v}$ does not intersect $P$, then the vertex of $F_{i,v}$ closest to $P$ has in-degree 0 in $F_{i,v}$. 
If $F_{i,v}$ does intersect $P$, then the vertex $v_a$ in $P\cap F_{i,v}$ with $a$ minimum has in-degree 0 in $F_{i,v}$ since $v_{a-1}v_a$ is the incoming edge at $v_a$ and $v_{a-1}$ is not in $F_{i,v}$
We have shown that $F_{i,v}$ has a vertex $x_v$ of in-degree 0 for each vertex $v\in V(G_i)$. 
Suppose that $x_v=x_w$ for distinct vertices $v,w\in V(G_i)$. Let $x:=x_v$. 
Let $yx$ be the incoming arc at $x$. 
Since $x$ has in-degree 0 in both $F_{i,v}$ and $F_{i,w}$, we have $v,w\in B_x \setminus B_y$, which contradicts the fact that $xy$ is thick. 
Hence $x_v\neq x_w$ for distinct vertices $v,w\in V(G_i)$. 
Each node $x \in V(F_i)$ equals $x_v$ for some $v\in V(G_i)$. Rename each $x_v$ by $v$. Now $V(F_i)=V(G_i)$. 
Consider the colouring $c$ of $G_i$ to also be a colouring of $F_i$. 
We now apply the argument in the proof of \cref{TreewidthCharacterisation} to conclude that $G_i \subseteq \GGG{F_i}{c}$. By construction, $V(G_i) = V(\GGG{F_i}{c})$. Consider an edge $vw\in E(G_i)$. Then $c(v)\neq c(w)$. Since $v$ and $w$ appear in a common bag, without loss of generality, $w$ is a descendant of $v$. Consider some vertex $x$ on the $vw$-path in $F_i$. Since $v\in B_v\cap B_w$, it follows that $v$ is in $B_x$, implying $c(v)\neq c(x)$. Hence $vw\in E(\GGG{F_i}{c})$ by the definition of $E(\GGG{T}{c})$. Therefore $G_i\subseteq \GGG{F_i}{c}$. 

We now show that $\GGG{F_i}{c}$ is contained in $\RR_k$. It suffices to show that there is a colour-preserving orientation-preserving isomorphism $\phi$ from $F_i$ to $\TT^{(k+1)}$. Here `colour-preserving' is with respect to the vertex-colouring of $F_i$ (not the edge-colouring) and with respect to the vertex-colouring of $\TT^{(k+1)}$ given in the definition of $\RR_k$ above. Up to isomorphism, $\TT^{(k+1)}$ is the unique oriented tree with in-degree 1 and out-degree $k$ at every node. So $F_i$ is isomorphic to $\TT^{(k+1)}$. In the colouring of $\TT^{(k+1)}$ given in the definition of $\RR_k$, the central path of $\TT^{(k+1)}$ is coloured modulo $k+1$. The path $P$ in $F_i$ is also coloured modulo $k+1$. Thus there is colour-preserving orientation-preserving isomorphism from $P$ to the central path of $T^{(k+1)}$, which determines a colour-preserving orientation-preserving isomorphism from $F_i$ to $\TT^{(k+1)}$ (by  considering the vertices of $F_i$ in order of their distance from $P$). Hence $\GGG{F_i}{c}$ is contained in $\RR_k$. Since $G_i\subseteq \GGG{F_i}{c}$, $G_i$ is contained in $\RR_k$. (In fact, in this case, $G_i$, $\GGG{F_i}{c}$ and $\RR_k$ are pairwise isomorphic.)\ 

Let $T'$ be obtained from $T$ by contracting each subtree $F_i$ into a vertex $x_i$. Then $(V(G_i): x_i \in V(T'))$ is a tree-decomposition of $G$. Since each bag $B_x$ is a clique in $G$, the torso of each node $x_i$ is $G_i$ itself. Hence  $(V(G_i): x_i \in V(T'))$ is a tree-decomposition of $G$ over $\RR_k$. 

We now prove the converse. Let $(B_x:x\in V(T))$ be a tree-decomposition of a graph $G$ over $\RR_k$ with adhesion $k-1$. Let $G_x$ be the torso of $x\in V(T)$. So $G_x$ is contained in $\RR_k$. By \cref{SimpleTreewidthRk}, $G_x$ has a $k$-simple tree-decomposition $(B^x_y: y \in V(T^x))$. Initialise $F$ to be the disjoint union of $\{T^x:x\in V(T)\}$. Associate with each node $y$ of $F$ the corresponding bag $B_y:= B^x_y$.  
Then for each edge $x_1x_2\in V(T)$, let $B^{x_1}_{y_1}$ and $B^{x_2}_{y_2}$ 
be bags respectively in the tree-decomposition of $G_{x_1}$ and $G_{x_2}$ with $B_{x_1}\cap B_{x_2} \subseteq B^{x_1}_{y_1} \cap B^{x_2}_{y_2}$, which exist since $B_{x_1}\cap B_{x_2}$ is a clique in $G_{x_1}$ and in $G_{x_2}$ (by the definition of torso). Add an edge in $F$ between $y_1$ and $y_2$. We obtain a tree-decomposition $(B_y:y\in V(F))$ of $G$. Since the tree-decomposition of each $G_x$ is $k$-simple, and $|B_{x_1}\cap B_{x_2}|\leq k-1$ for each edge $x_1x_2$ of $T$, the tree-decomposition  $(B_y:y\in V(F))$ of $G$ is also $k$-simple. Hence $G$ has simple treewidth at most $k$. 
\end{proof}

Let $\SS_k:= \widehat{\RR_k}^{k-1}$ defined in \cref{TreeDecompUniversalAdhesion}. Then  \cref{SimpleTreewidthTreeDecomposition,TreeDecompUniversalAdhesion} imply:

\begin{thm}
\label{SimpleTreewidthUniversal}
For each $k\in\NN$, the graph $\SS_k$ is universal for the class of graphs with simple treewidth at most $k$. 
\end{thm}

We now focus on graphs with simple treewidth 2 or 3. 

\begin{thm}
$\SS_2$ is universal for the class of outerplanar graphs.
\end{thm}

\begin{proof}
If a graph $G$ has a tree-decomposition of adhesion 1 in which every torso is outerplanar, then $G$ is also outerplanar. By definition, $\SS_2$ has a tree-decomposition of adhesion 1 in which every torso is isomorphic to $\RR_2$ (which is outerplanar). Thus $\SS_2$ is outerplanar. 

We now prove the converse. That is, $\SS_2$ contains every outerplanar graph. By \cref{SimpleTreewidthUniversal}, it suffices to show that every outerplanar graph has simple treewidth 2. Since simple treewidth is extendable (\cref{SimpleTreewidthExtendable}), it suffices to show that every finite outerplanar graph has simple treewidth 2, which is a folklore and easily proved result \citep{MJP06,Wulf16,KU12}.
\end{proof}

\begin{thm}
$\SS_3$ is universal for the class of planar graphs with treewidth 3.
\end{thm}

\begin{proof}
If a graph $G$ has a tree-decomposition of adhesion 2 in which every torso is planar, then $G$ is also planar. If a graph $G$ has a tree-decomposition of adhesion 2 in which every torso has treewidth at most 3, then $G$ also has treewidth at most 3. Now $\SS_3$ has a tree-decomposition of adhesion 2 in which every torso is isomorphic to $\RR_3$ (which is planar with treewidth 3). Thus $\SS_3$ is planar with treewidth 3. 

We now prove the converse. That is, $\SS_3$ contains every planar graph with treewidth 3. By \cref{SimpleTreewidthUniversal}, it suffices to show that every planar graph with treewidth 3 has simple treewidth 3. Since simple treewidth is extendable (\cref{SimpleTreewidthExtendable}), it suffices to show that every finite planar graph with treewidth 3 has simple treewidth 3. This was proved by \citet{KU12} (also see \citep{Wulf16}) using the result of \citet{KV12}, who showed that for every finite planar graph $G$ of treewidth 3 there is a 3-tree (an edge-maximal graph of treewidth 3) that is planar and contains $G$ as a spanning subgraph. 
\end{proof}

%%%%%%%%%%%%%%%%%%%%%%%%%%%%%%%%
\section{Treewidth--Path Product Structure}
\label{TreewidthPathStructure}

The primary result of this section describes a graph, the strong product of a bounded treewidth graph and a path, that contains every planar graph, and satisfies several of the key properties mentioned in \cref{Intro}. The result is extended in various ways for graphs embeddable on any fixed surface, for any proper minor-closed class, and for several non-minor-closed graph classes of interest.

%%%%%%%%%%%%%%%%%%%%%%%%%%%%%%%%%%%%%%%%%
\subsection{Layered Partitions}

A \defn{layering} of a graph $G$ is an ordered partition $(V_0,V_1,\dots)$ of $V(G)$ such that for every edge $vw\in E(G)$, if $v\in V_i$ and $w\in V_j$, then $|i-j| \leq 1$. Note that a layering is equivalent to a partition whose quotient is a path. \citet{DJMMUW20} defined the \defn{layered width} of  a partition $\PART$ of a graph $G$  to be the minimum $\ell\in\NN$ such that for some layering $(V_0,V_1,\dots)$ of $G$, each part in $\PART$ has at most $\ell$ vertices in each layer $V_i$. \citet{DJMMUW20} proved the following lemma in the finite case. The straightforward proof also works for infinite graphs.

\begin{lem}[{\protect\citep[Observation~35]{DJMMUW20}}]
\label{MakeProductGeneral} 
If a graph $G$ has a partition $\PART$ with layered width $\ell$, then $G$ is contained in $ (G/ \PART) \boxtimes \PP \boxtimes K_\ell$. Conversely, for every graph $H$ and  subgraph $G$ of $H \boxtimes \PP \boxtimes K_{\ell}$, there is a partition $\PART$ of $G$ with layered width at most $\ell$ such that $G/\PART$ is contained in $H$. 
\end{lem}

For $a,c\in\NN_0$ and $k,\ell\in\NN$, let  $\Gamma_{k,c,\ell,a}$ be the class of graphs isomorphic to subgraphs of $( (\TT_k + K_c) \boxtimes \PP \boxtimes K_{\ell}) + K_a$. \cref{TreewidthUniversal,MakeProductGeneral} imply:

\begin{lem}
\label{ProductSubgraphLayeredWidth}
A graph $G$ is in $\Gamma_{k,c,\ell,a}$ if and only if there is a set $A\subseteq V(G)$ of size at most $a$, and a partition $\PART$ of $G-A$ with layered width $\ell$ such that $\tw_c( (G-A)/\PART) \leq k$. 
\end{lem}

The following notation will be helpful to prove that $\Gamma_{k,c,\ell,a}$ is extendable. If $\PART_1$ and $\PART_2$ are partitions of a set $X$, then \DefNoIndex{$\PART_1\subseteq \PART_2$} if for all $A_1\in\PART_1$ and $A_2\in\PART_2$, either $A_1\cap A_2=\emptyset$ or $A_1\subseteq A_2$. If $\PART_1,\PART_2,\dots$ are partitions of a set $X$ and $\PART_1\subseteq \PART_2 \subseteq\dots$, then $\bigcup_{n\in\NN} \PART_n$ is the partition of $X$, where $A$ is in $\bigcup_{n\in\NN} \PART_n$ if $A\in\PART_n$ for some $n\in\NN$, and for all $n\in\NN$ no strict superset of  $A$ is in $\PART_n$. 

\begin{lem}
\label{FiniteToInfiniteGamma}
$\Gamma_{k,c,\ell,a}$ is extendable
\end{lem}

\begin{proof} 
Let $G$ be a graph such that every finite subgraph of $G$ is in $\Gamma_{k,c,\ell,a}$. Let $(v_1,v_2,\dots)$ be a vertex-ordering of $G$. Let $G_n := G[ \{v_1,\dots,v_n\}]$ for $n\in\NN$. Let $X_n$ be the set of all triples $(A_n,\LL_n,\PART_n)$ such that $A_n$ is a set of at most $a$ vertices in $G_n$, $\LL_n$ is a layering of $G_n-A_n$, and $\PART_n$ is a partition of $G_n-A_n$ such that $|L\cap P|\leq \ell$ for all $L\in\LL_n$ and $P\in \PART_n$, and $\tw_c((G_n-A_n)/\PART_n)\leq k$. Since $G_n$ is in $\Gamma_{k,c,\ell,a}$, \cref{ProductSubgraphLayeredWidth} implies that $X_n\neq\emptyset$. And $X_n$ is finite since $G_n$ is finite. Let $Q$ be the graph with vertex-set $V(Q):= \bigcup_{n\in\NN} X_n$, where each $(A_n,\LL_n,\PART_n)\in X_n$ is adjacent to $(A_n\setminus \{v_n\},\LL_n-v_n,\PART_n-v_n)$, which is in $X_{n-1}$. By \cref{Konig}, there is a path $(A_1,\LL_1,\PART_1),(A_2,\LL_2,\PART_2),\dots$ in $Q$ with $(A_n,\LL_n,\PART_n)\in X_n$ for all $n\in\NN$.  Then  $A_1\subseteq A_2\subseteq\dots$  and $\LL_1\subseteq \LL_2\subseteq\dots$  and $\PART_1\subseteq \PART_2\subseteq\dots$. Let $A:=\cup_n A_n$ and $\LL:=\cup_n \LL_n$ and $\PART:= \bigcup_n \PART_n$. Then $A$ is a set of at most $a$ vertices in $G$, $\LL$ is a layering of $G-A$, and $\PART$ is a partition of $G-A$ such that $|L\cap P|\leq \ell$ for each $L\in\LL$ and $P\in\PART$. For every finite subgraph $G'$ of $(G-A)/\PART$, for some $n\in\NN$, $G'$ is a subgraph of $(G_n-A_n)/\PART_n$, implying that $\tw_c(G') \leq \tw_c((G_n-A_n)/\PART_n) \leq k$. By \cref{TreewidthApexExtendable}, $\tw_c((G-A)/\PART)\leq k$. Hence $G\in \Gamma_{k,c,\ell,a}$ by \cref{ProductSubgraphLayeredWidth}. Therefore $\Gamma_{k,c,\ell,a}$ is extendable.  
\end{proof}

\cref{TreewidthUniversal,FiniteToInfiniteGamma} imply:

\begin{lem}
\label{FiniteToInfiniteProduct}
Let $G$ be a graph such that every finite subgraph of $G$ is contained in $( (H+K_c) \boxtimes P \boxtimes K_\ell)+K_a$ for some graph $H$ with treewidth at most $k$ and for some path $P$. Then $( (\TT_k + K_c) \boxtimes \PP \boxtimes K_\ell)+K_a$ contains $G$. 
\end{lem}

An analogous proof using \cref{SimpleTreewidthUniversal} instead of 
\cref{TreewidthUniversal} and \cref{SimpleTreewidthExtendable} instead of \cref{TreewidthExtendable} gives:

\begin{lem}
\label{FiniteToInfiniteSimple}
Let $G$ be a graph such that every finite subgraph of $G$ is contained in $( (H+K_c) \boxtimes P \boxtimes K_\ell)+K_a$ for some graph $H$ with simple treewidth at most $k$ and for some path $P$. Then $( (\SS_k+K_c) \boxtimes \PP \boxtimes K_\ell)+K_a$ contains $G$. 
\end{lem}

%%%%%%%%%%%%%%%%%%%%%%%%%%%%
\subsection{Planar Graphs}
\label{PlanarGraphs}

The starting point for the proof of \cref{InfinitePlanarStructure} is the result of \citet{PS21}, who showed that every planar graph has a partition into geodesic paths whose contraction gives a graph with treewidth at most 8. This result was refined by \citet{DJMMUW20} as follows. 

\begin{thm}[{\protect\citep[Theorem~36(a)]{DJMMUW20}}]
\label{FinitePlanarStructure}
Every finite planar graph $G$ is contained in $H\boxtimes P$, for some planar graph $H$ with  treewidth at most 8 and for some path $P$.
\end{thm}

\cref{FinitePlanarStructure} has been used to solve several open problems regarding queue layouts~\citep{DJMMUW20}, non-repetitive colourings~\citep{DEJWW20}, centred colourings~\citep{DFMS21}, clustered colourings~\citep{DEMWW22}, adjacency labellings (equivalently, strongly universal graphs)~\citep{BGP22,DEGJMM21,EJM23,DGJMMW}, twin-width~\citep{BKW26,JP22}, and vertex rankings~\citep{BDJM}.

\citet{UWY22} modified the proof of \cref{FinitePlanarStructure} to establish the following.

\begin{thm}[{\protect\citep[Theorem~2]{UWY22}}]
\label{FinitePlanarStructure6}
Every finite planar graph $G$ is contained in $H\boxtimes P$, for some planar graph $H$ with simple treewidth at most 6 and for some path $P$.
\end{thm}

\cref{FiniteToInfiniteProduct,FinitePlanarStructure6} imply the following theorem introduced in \cref{Intro}. 

\begin{thm}
\label{InfinitePlanarStructure6}
$\SS_6 \boxtimes \PP$ contains every planar graph.
\end{thm}

\citet{DJMMUW20} also proved the following variation on \cref{FinitePlanarStructure}. 

\begin{thm}[{\protect\citep[Theorem~36(b)]{DJMMUW20}}]
\label{SimpleFinitePlanarStructure}
Every finite planar graph $G$ is contained in $H\boxtimes P \boxtimes K_3$ for some planar graph $H$ with $\tw(H)\leq 3$ and for some path $P$. 
\end{thm}

In \cref{SimpleFinitePlanarStructure}, since $H$ is planar and $\tw(H)\leq 3$, by the above-mentioned result of \citet{KV12}, we have $\stw(H)\leq 3$. This can also be seen directly from the proof of \cref{SimpleFinitePlanarStructure} and is implicitly mentioned in \citep{DJMMUW20}. Then \cref{FiniteToInfiniteProduct,SimpleFinitePlanarStructure} imply:

\begin{thm}
\label{SimpleInfinitePlanarStructure}
$\SS_3\boxtimes \PP \boxtimes K_3$ contains every planar graph.
\end{thm}

We now show that the graphs described in \cref{InfinitePlanarStructure6,SimpleInfinitePlanarStructure} satisfy properties \cref{KeyPropertyBoundedMinDegree}--\cref{KeyPropertyBoundedColouringNumbers} from \cref{Intro}. It follows from results of \citet{HW25} that every finite subgraph of $\SS_6\boxtimes \PP$ has minimum degree at most $19$, and  every finite subgraph of $\SS_3\boxtimes \PP \boxtimes K_3$ has minimum degree at most $32$. Note that $\chi( \SS_6\boxtimes \PP) \leq 14$ and $\chi(\SS_3\boxtimes \PP \boxtimes K_3)\leq 24$ since $\chi(A\boxtimes B) \leq \chi(A)\,\chi(B)$.  \cref{ProductTreewidth} implies that every $n$-vertex subgraph of  $\SS_6 \boxtimes \PP$ or $\SS_3\boxtimes \PP \boxtimes K_3$ has treewidth $O(\sqrt{n})$ and has a balanced separation of order $O(\sqrt{n})$.  \cref{GenColourProduct} implies that  $\SS_6 \boxtimes \PP$ and $\SS_3\boxtimes \PP \boxtimes K_3$ (since $\tw(\SS_3\boxtimes K_3)\leq 11$) have linear  colouring numbers. Together, this says that $\SS_6 \boxtimes \PP$ and $\SS_3\boxtimes \PP \boxtimes K_3$ satisfy properties \cref{KeyPropertyBoundedMinDegree}--\cref{KeyPropertyBoundedColouringNumbers} from \cref{Intro}.

%%%%%%%%%%%%%%%%%%%%%%%%%%%%%%%%
\subsection{Graphs on Surfaces}
\label{Genus}

\citet{DJMMUW20} proved generalisations of \cref{FinitePlanarStructure,SimpleFinitePlanarStructure} for finite graphs of bounded Euler genus. 
Building on this work, \citet{UWY22} and \citet{DHHW22} proved the following result.

% %\begin{thm}[\citep{DJMMUW20}]
% \label{FiniteSurfaceProduct}
% Every finite graph $G$ of Euler genus $g$ is isomorphic to a subgraph of:
% \begin{enumerate}[label=(\alph*)]
% \item $H \boxtimes P \boxtimes K_{\max\{2g,1\}}$ for some apex graph $H$ of treewidth 9 and some path $P$,
% \item $H \boxtimes P \boxtimes K_{\max\{2g,3\}}$ for some apex graph $H$ of treewidth 4 and some path $P$,
% \item $(K_{2g} + H ) \boxtimes P$ for some planar graph $H$ of treewidth 8 and some path $P$.
% \end{enumerate}
% \end{thm}

\begin{thm}[{\protect\citep[Theorem~5]{UWY22},~\citep[Theorem~3]{DHHW22}}]
\label{FiniteSurfaceProduct}
Every finite graph $G$ of Euler genus $g$ is contained in:
\begin{enumerate}[label=(\alph*)]
\item $H \boxtimes P \boxtimes K_{\max\{2g,3\}}$ for some planar graph $H$ of simple treewidth 3 and some path $P$,
\item $(H+K_{2g}) \boxtimes P$ for some planar graph $H$ of simple treewidth 6 and some path $P$.
\end{enumerate}
\end{thm}

% Since $\stw(G)\leq \stw(G-v)+1$ for every graph $G$ and vertex $v$ of $G$, applying \cref{FinitePlanarStructure6} instead of \cref{FinitePlanarStructure}, the treewidth bounds in \cref{FiniteSurfaceProduct} become:

% \begin{thm}
% \label{SimpleFiniteSurfaceProduct}
% Every finite graph $G$ of Euler genus $g$ is isomorphic to a subgraph of: 
% \begin{enumerate}[label=(\alph*)]
% \item $ H \boxtimes P \boxtimes K_{\max\{2g,1\}}$ for some apex graph $H$ of simple treewidth 7 and some path $P$, 
% \item $H \boxtimes P \boxtimes K_{\max\{2g,3\}}$ for some apex graph $H$ of simple treewidth 4 and some path $P$,
% \item $(H + K_{2g} ) \boxtimes P$ for some planar graph $H$ of simple treewidth 6 and some path $P$.
% \end{enumerate}
% \end{thm}

\cref{FiniteToInfiniteProduct,FiniteSurfaceProduct} imply:

\begin{thm}
\label{InfiniteSurfaceProduct}
For every $g\in\NN_0$, both
$\SS_3 \boxtimes \PP \boxtimes K_{\max\{2g,3\}}$ and 
$(\SS_6 +K_{2g} )  \boxtimes \PP$ 
contain every graph of Euler genus $g$.
\end{thm}

A graph $X$ is \defn{apex} if $X-v$ is planar for some vertex $v$ (or $V(X)=\emptyset$). \citet{DJMMUW20} proved the following product structure theorem for finite apex-minor-free graphs. 

\begin{thm}[{\protect\citep[Corollary~40]{DJMMUW20}}] 
\label{ApexMinorFree}
For every finite apex graph $X$, there exists $c\in\NN$ such that every finite  $X$-minor-free graph $G$ is contained in $H\boxtimes P$ for some graph $H$ with $\tw(H)\leq c$ and for some path $P$.
\end{thm}

\cref{FiniteToInfiniteProduct,ApexMinorFree} imply:

\begin{thm} 
For every finite apex graph $X$ there exists $c\in \NN$ such that $\TT_c \boxtimes \PP$ contains  every  $X$-minor-free graph $G$.
\end{thm}

\citet{DEMWW22} showed an analogous result for  bounded degree graphs excluding an arbitrary fixed minor.

\begin{thm}[{\protect\citep[Theorem~24]{DEMWW22}}]
\label{FiniteMinorFreeDegreeStructure}
For every finite graph $X$ there exists $c\in\NN$ such that for all $\Delta\in\NN$ every finite $X$-minor-free graph with maximum degree $\Delta$ is contained in $H\boxtimes P$ for some graph $H$ with $\tw(H) \leq c\Delta$ and for some path $P$. 
\end{thm}

\cref{FiniteToInfiniteProduct,FiniteMinorFreeDegreeStructure} imply:

\begin{thm}
\label{MinorFreeDegreeStructure}
For every graph $X$ there exists $c\in\NN$ such that for all $\Delta\in\NN$, the graph $\TT_{c\Delta} \boxtimes \PP$ contains every  $X$-minor-free graph with maximum degree at most $\Delta$.
\end{thm}

\subsection{Beyond Minor-Closed Classes}
\label{NonMinorClosed}

Recent research studies graph product structure theorems for various non-minor-closed graph classes~\citep{DMW23,HW24,DHSW24}. Here we extend their results for infinite graphs. First consider graphs that can be drawn on a fixed surface with a bounded number of crossings per edge. A graph is \hdefn{$(g,k)$}{planar} if it has a drawing in a surface of Euler genus at most $g$ such that each edge is involved in at most $k$ crossings. Even in the simplest case, there are $(0,1)$-planar graphs that contain arbitrarily large complete graph minors \citep{DEW17}. 

\begin{thm}[{\protect\citep[Theorem~11]{DMW23}}] 
\label{gkPlanarStructure1}
Every finite $(g,k)$-planar graph is contained in $H\boxtimes P$, for some graph $H$ of treewidth $O(gk^6)$ and for some path $P$. 
\end{thm}

\begin{thm}[{\protect\citep[Corollary~12]{DHSW24}}]  
\label{gkPlanarStructure2}
There is a function $f$ such that every $(g,k)$-planar graph $G$ is contained in
$H \boxtimes P \boxtimes K_{f(g,k)}$ for some graph $H$ with $\tw( H ) \leq 963\, 922\, 179$.
\end{thm}

\cref{FiniteToInfiniteProduct,gkPlanarStructure1,gkPlanarStructure2} imply:

\begin{thm}
\label{gkPlanarInfinite1}
For all $g,k\in\NN$ there exists an integer $c\in O(gk^6)$ such that $\TT_c\boxtimes \PP$ contains every $(g,k)$-planar graph.
\end{thm}

\begin{thm}
\label{gkPlanarInfinite2}
There is a function $f$ such that $\TT_{963\, 922\, 179}\boxtimes \PP \boxtimes K_{f(g,k)}$ contains every $(g,k)$-planar graph.
\end{thm}

\note{Note that \cref{gkPlanarStructure2,gkPlanarInfinite2} have been added to the paper. }

Map graphs, which are defined as follows, provide another example of a non-minor-closed class that has a product structure theorem. Start with a graph $G_0$ embedded in a surface of Euler genus $g$, with each face labelled a `nation' or a `lake', where each vertex of $G_0$ is incident with at most $d$ nations. Let $G$ be the graph whose vertices are the nations of $G_0$, where two vertices are adjacent in $G$ if the corresponding faces in $G_0$ share a vertex. Then $G$ is called a \hdefn{$(g,d)$}{map graph}. A $(0,d)$-map graph is called a (plane) \hdefn{$d$}{map graph}; see \citep{FLS-SODA12,CGP02} for example. The $(g,3)$-map graphs are precisely the graphs of Euler genus at most $g$; see \citep{DEW17}. So $(g,d)$-map graphs generalise graphs embedded in a surface, and we now assume that $d\geq 4$ for the remainder of this section. 

\begin{thm}[{\protect\citep[Theorem~18]{DMW23}}]
\label{MapPartition}
Every finite $(g,d)$-map graph is  isomorphic to  a subgraph of:
\begin{itemize}
\item  $H\boxtimes P\boxtimes K_{O(gd^2)}$, where $H$ is a graph with treewidth at most 14 and $P$ is a path,
\item $H\boxtimes P$, where $H$ is a graph with treewidth $O(gd^2)$ and $P$ is a path.
\end{itemize}
\end{thm}

\cref{FiniteToInfiniteProduct,MapPartition} imply:

\begin{thm}
For all $g,d\in\NN$ there exists $c\in O(gd^2)$ such that $\TT_c\boxtimes \PP$ contains every  $(g,d)$-map graph.
\end{thm}

A \defn{string graph} is the intersection graph of a set of curves in the plane with no three curves meeting at a single point; see  \cite{PachToth-DCG02,FP10,FP14} for example. For $\delta\in\NN$, if each curve is in at most $\delta$ intersections with other curves, then the corresponding string graph is called a \hdefn{$\delta$}{string graph}. A \hdefn{$(g,\delta)$}{string} graph is defined analogously for curves on a surface of Euler genus at most $g$.  

\begin{thm}[{\protect\citep[Theorem~14]{DMW23}}] 
\label{StringPartition}
Every finite  $(g,\delta)$-string graph is contained in $H\boxtimes P$, for some graph $H$ of treewidth $O(g\delta^7)$ and some path $P$.
\end{thm}

\cref{FiniteToInfiniteProduct,StringPartition} imply:

\begin{thm}
For all $g,\delta\in\NN$ there exists $c\in O(g\delta^7)$ such that $\TT_c \boxtimes \PP$ contains every  $(g,\delta)$-string graph.
\end{thm}

Finally, we mention the following result, which is an immediate corollary of \cref{FiniteMinorFreeDegreeStructure} and Theorem~16 by \citet{DMW23}. The \DefnIndex{$k$-th power}{power} of a graph $G$ is the graph $G^k$ with $V(G^k):=V(G)$, where $vw\in E(G^k)$ whenever $\dist_G(v,w)\in[k]$. 

\begin{thm}
\label{kPower}
For every finite graph $X$ there exists $c\in\NN$ such that for all $k,\Delta\in\NN$ and for every finite $X$-minor-free graph $G$ with maximum degree $\Delta$, the $k$-th power $G^k$ is contained in $H\boxtimes P$, for some graph $H$ of treewidth $k^4(c\Delta)^{k}$ and some path $P$.
\end{thm}

\cref{FiniteToInfiniteProduct,kPower} imply:

\begin{thm}
For every finite graph $X$ there exists $c\in\NN$ such that for all $k,\Delta\in\NN$, for some integer $t\leq k^4(c\Delta)^{k}$, the graph $\TT_t\boxtimes \PP$ contains the $k$-th power $G^k$ of every  $X$-minor-free graph $G$ with maximum degree $\Delta$.
\end{thm}

%%%%%%%%%%%%%%%%
\subsection{Excluding an Arbitrary Minor}
\label{ExcludedMinors}

\citet{DJMMUW20} proved the following product structure theorem for finite graphs excluding a fixed minor. 

\begin{thm}[{\protect\citep[Theorem~42]{DJMMUW20}}] 
\label{MinorProduct}
For every finite graph $X$ there exist $k,a\in\NN$ such that every finite $X$-minor-free graph $G$ can be obtained by clique-sums of graphs $G_1,\dots,G_n$ such that for  $i\in[n]$, for some graph $H_i$ with treewidth at most $k$ and some path $P_i$, $G_i$ is contained in $(H_i  \boxtimes P_i ) + K_a$. 
\end{thm}

\begin{thm}
\label{MinorUniversal}
For every finite graph $X$ there is a constant $c$ and there is a graph $U_X$ that contains every $X$-minor-free graph, such that:
\begin{itemize}
    \item every $n$-vertex subgraph $G$ of $U_X$ has treewidth at most $c \sqrt{n}$ and has a balanced separation of order $c\sqrt{n}$, 
    \item $U_X$ has linear colouring numbers, and
    \item $U_X$ has linear expansion. \note{This point has been added.}
\end{itemize}
\end{thm}

\begin{proof}
\cref{FiniteToInfiniteGamma,TreeDecompExtendable} implies that 
$\DD(\Gamma_{k,c,\ell,a})$ is extendable for all $k,\ell\in\NN$ and $c,a\in\NN_0$. 
\cref{MinorProduct} says that every finite $X$-minor-free graph is in 
$\DD( \Gamma_{k,0,1,a} )$ for some $k,a$ depending only on $X$. 
Thus every $X$-minor-free graph is in $\DD( \Gamma_{k,0,1,a} )$. 
Let $U_X :=\reallywidehat{( \TT_k \boxtimes \PP ) + K_a}$. By  \cref{TreeDecompUniversal}, $U_X$ contains every  $X$-minor-free graph. \cref{ProductTreewidth,TreewidthOverTreeDecomposition} imply that every $n$-vertex subgraph of $U_X$ has treewidth at most $c\sqrt{n}$ and has a balanced separation of order $c\sqrt{n}$. \cref{GenColourProductTreeDecomp,GenColourTreeDecomp} imply that $U_X$ has linear  colouring numbers. \cref{SergeyCorollary} implies $U_X$ has linear expansion.
\end{proof}

We finish this section by mentioning one more property of the graph $U_X$ in \cref{MinorUniversal}. \citet{DDOSRSV04} proved that for every finite graph $X$ and integer $k\geq 2$, there is an integer $c$ such that every finite $X$-minor-free graph (a) has a vertex $k$-colouring such that the subgraph induced by any $k-1$ colour classes has treewidth at most $c$, and (b) has an edge $k$-colouring such that the subgraph induced by any $k-1$ colour classes has treewidth at most $c$.  \citet{DDOSRSV04} called these \defn{low treewidth colourings}. \citet{DJMMUW20} showed that this result is a corollary of \cref{MinorProduct}. Their proof also implies that $U_X$ has low treewidth colourings. It is interesting that not only does every $X$-minor-free graph have low treewidth colourings, but there is a single graph that contains all $X$-minor-free graphs and itself admits low treewidth colourings. 

%%%%%%%%%%%%%%%%%%%%%%%%%%%%%%%%%%%%%%%%%%%%%%%%%%%%%
\section{Avoiding an Infinite Complete Graph Subdivision}
\label{NoInfiniteSubdivision}

This section focuses on key property \cref{KeyPropertyNoInfiniteSubdivision}, which says that a graph that contains every planar graph should contain no $K_{\aleph_0}$ subdivision. We start the section by giving several characterisations of graphs that contain no $K_{\aleph_0}$ subdivision (\cref{NoInfiniteSubdivisionCharacterisation}). We then construct a graph that contains every planar graph, but contains no $K_{\aleph_0}$ subdivision, amongst other key properties. This shows that \cref{InfiniteCliqueMinor} is best possible in the sense that it cannot be strengthened to conclude that every graph that contains every planar graph must contain a subdivision of $K_{\aleph_0}$. Our construction is, in fact, stronger in two respects: it works for $K_t$-minor-free graphs, and actually gives induced subgraphs. Using the same method, we construct a strongly universal treewidth-$k$ graph, and a graph that contains every locally finite graph as an induced subgraph, but contains no $K_{\aleph_0}$ subdivision (\cref{LocallyFiniteGraphs}). We conclude the section by showing that every graph that contains all planar graphs has a subgraph of infinite edge-connectivity that contains all planar graphs (\cref{InfiniteEdgeConnectivity}). 

%%%%%%%%%%%%%%%%%%%%%%%%%%%%%%%%%%%%%%%%%%%%%%%
\subsection{Characterisations}
\label{NoInfiniteSubdivisionCharacterisation}

\note{In response to the referee's comment, we have revised the following discussion and theorem to make the structure clearer. In particular, we have added property \cref{WIT} (and the definition of `treewidth $<\aleph_0$') so that the statement of the result of \citet{RST-TAMS92} is clear. We have also removed two conditions from \cref{InfiniteCompleteSubdivision} since we never used them, and they were not needed in the circular sequence of implications.}

The following theorem presents several characterisations of graphs containing no subdivision of $K_{\aleph_0}$. 
\citet{RST-TAMS92} showed that parts \cref{NoSub} and \cref{WIT} are equivalent. Later, \citet{Diestel94} gave a simple proof showing that \cref{NoSub} and \cref{WIT} are equivalent.  Diestel's proof uses normal spanning trees; see \citep{Jung69,BD94,Pitz20} for more on this theme.
Parts \cref{NoSubChordal}--\cref{NoSubModel} are included in \cref{InfiniteCompleteSubdivision} because of their relationship to other parts of this paper. In particular, condition~\cref{NoSubChordalPart} is used in \cref{ExcludedMinorNoSubdiv} below. Condition~\cref{NoSubModel} is interesting since it is an infinite analogue of $r$-shallow minors; see \cref{Expansion}.

\begin{thm}
\label{InfiniteCompleteSubdivision}
The following are equivalent for any countable graph $G$:
\begin{enumerate}[(a)]
\item\label{NoSub} $G$ contains no subdivision of $K_{\aleph_0}$; 
\item\label{WIT} $G$ has treewidth $<\aleph_0$; 
\item\label{NoSubChordal} $G$ is a spanning subgraph of a chordal graph that contains no $K_{\aleph_0}$ subgraph;
\item\label{NoSubChordalPart} $G$ is a spanning subgraph of a graph $G'$ that has a finite chordal partition $\PART$ such that $G'/\PART$ contains no $K_{\aleph_0}$ subgraph;
\item\label{NoSubModel} $G$ contains no model of $K_{\aleph_0}$ with branch sets of finite radius.
\end{enumerate}
\end{thm}

\begin{proof}
\cref{NoSub} $\Longrightarrow$ \cref{WIT}:
As explained above,~\citet{RST-TAMS92} actually proved that \cref{NoSub} and \cref{WIT} are equivalent.
%\end{proof}

\smallskip

%\begin{proof} \cref{WIT} $\Longrightarrow$ \cref{NoSubSimpDec}: Let $(B_x\subseteq V(G):x\in V(T))$ be a tree-decomposition of $T$ certifying that $G$ has finite or semi-infinite treewidth.  Let $H$ be the minimal supergraph of $G$ such that each $B_x$ is a clique. Since  $G$ has finite or semi-infinite treewidth, $H$ contains no $K_{\aleph_0}$. By~\cref{ChordalCharacterisation}, $H$ is a chordal graph that contains no $K_{\aleph_0}$ and has a simplicial decomposition in which each simplicial summand is a finite complete graph. \end{proof}

%\begin{proof}
\cref{WIT} $\Longrightarrow$ \cref{NoSubChordal}:
Let $(B_x\subseteq V(G):x\in V(T))$ be a tree-decomposition of $T$ with finite bags certifying that $G$ has treewidth $<\aleph_0$.  Let $G'$ be the minimal supergraph of $G$ such that each $B_x$ is a clique. Clearly, $G$ is a spanning subgraph of $G'$. Since  $G$ has treewidth $<\aleph_0$, $G'$ contains no $K_{\aleph_0}$. By~\cref{ChordalCharacterisation}, $G'$ is chordal.
%\end{proof}

\smallskip

%It is immediate that \cref{NoSubSimpDec} $\Longrightarrow$ \cref{NoSubChordal}.

%\begin{proof}
\cref{NoSubChordal} $\Longrightarrow$ \cref{NoSubChordalPart}:
Suppose that $G$ is a spanning subgraph of a chordal graph $G'$ containing no $K_{\aleph_0}$ subgraph. Let $\PART$ be the (trivial) chordal partition $\{\{v\}:v\in V(G')\}$ of $G'$. So $G'/\PART\cong G'$, and thus $G'/\PART$ contains no $K_{\aleph_0}$ subgraph.
%\end{proof}

\smallskip

%\begin{proof}
\cref{NoSubChordalPart} $\Longrightarrow$ \cref{NoSub}: Suppose that $\PART$ is a finite chordal partition of $G'$ such that $G'/\PART$ contains no $K_{\aleph_0}$ subgraph. Suppose that $V$ is an infinite set of vertices in $G'$ such that no  pair of vertices in $V$ are separated by a finite vertex-cut. For each $v\in V$, let $A_v$ be the part of $\PART$ containing $v$. Since each part of $\PART$ is finite, we may assume (by taking a subset) that $A_v\neq A_w$ for all distinct $v,w\in V$. Suppose that $A_vA_w\notin E(G/\PART)$ for distinct $v,w\in V$. Let $S$ be a minimal subset of $V(G/\PART)$ separating $A_v$ and $A_w$. Since $G'/\PART$ is chordal, $S$ is a clique in $G'/\PART$, which is finite by assumption. Hence $\bigcup_{X\in S} X$ is a finite subset of $V(G')$ separating $v$ and $w$, which is a contradiction. Thus $A_vA_w\in E(G'/\PART)$ for all distinct $v,w\in V$. Therefore $\{A_v:v\in V\}$ induces $K_{\aleph_0}$ in $G'/\PART$. This contradiction shows that there is no infinite set $V$ of vertices in $G'$ such that no pair of vertices in $V$ is separated by a finite cut. In particular, $G'$ and hence also $G$ contains no $K_{\aleph_0}$ subdivision. 
%\end{proof}

\smallskip

%\begin{proof}
\cref{NoSub} $\Longrightarrow$ \cref{NoSubModel}: 
Suppose that $G$ contains a model $(X_i)_{i\in\NN}$ of $K_{\aleph_0}$ such that $X_i$ has finite radius for each $i\in\NN$. We may assume that for all distinct $i,j\in\NN$, there is precisely one edge joining $X_i$ and $X_j$. Let that edge be $v_{i,j}v_{j,i}$, where $v_{i,j}\in X_i$ and $v_{j,i}\in X_j$. We may also assume that $V(G)=\bigcup_{i\in\NN}V(X_i)$, and that each $X_i$ is minimal with the above properties. So each $X_i$ is a tree with finite radius, and every leaf of $X_i$ is adjacent to some vertex in $X_j$ for some $j\neq i$. Say $X_i$ is \defn{clean} if
%\begin{enumerate}[(i)] 
%\item for each $j\in\NN\setminus\{i\}$, there is a vertex $v_{i,j}$ in $X_i$ adjacent to some vertex in $X_j$, 
%\item 
there is a vertex $r_i$ of infinite degree in $X_i$, such that for all distinct $j,k\in\NN$ with $j,k>i$, the $r_iv_{i,j}$-path and the $r_iv_{i,k}$-path in $X_i$ only intersect at $r_i$. 
%\end{enumerate}
Suppose that $X_1,\dots,X_{\ell-1}$ are clean for some $\ell\in\NN$. Since $(X_i)_{i\in\NN}$ is a model of a complete graph, for each $j\in\NN$ with $j<\ell$, there is a vertex $v_{\ell,j}\in X_\ell$ adjacent to $v_{j,\ell}\in X_j$. Since $X_\ell$ has finite radius, some vertex $r_\ell\in V(X_\ell)$ has infinite degree in $G$. By the minimality of $X_\ell$, there is an infinite set $I\subseteq\{\ell+1,\ell+2,\dots\}$, such that for any distinct $j,k\in I$, the $r_\ell v_{\ell,j}$-path and the $r_\ell v_{\ell,k}$-path in $X_\ell$ only intersect at $r_\ell$. Replace $X_\ell$ by the minimal subtree of $X_\ell$ containing
the union of the $r_\ell v_{\ell,j}$-paths in $X_\ell$, taken over all $j\in[\ell-1]\cup I$. Replace $(X_i)_{i\in\NN}$ by 
$(X_i)_{i\in [\ell] \cup I}$. In this sequence, $X_1,\dots,X_{\ell}$ are clean. 

Repeat this step to obtain an infinite $K_{\aleph_0}$-model $(X_i)_{i\in\NN}$ in which every $X_i$ is clean. 
%We may assume that $v_{i,j}v_{j,i}$ is an edge of $G$ for all distinct $i,j\in\NN$. 
We now construct a subdivision of $K_{\aleph_0}$ in $G$ rooted at $(r_{i^2})_{i\in\NN}$. 
For $i,j\in\NN$ with $i<j$, let $P_{i,j}$ be the path
from $r_{i^2}$ to $v_{i^2,j^2+i}$ in $X_{i^2}$, 
followed by the edge $v_{i^2,j^2+i}v_{j^2+i,i^2}$,
followed by the path from $v_{j^2+i,i^2}$ to $v_{j^2+i,j^2}$ in $X_{j^2+i}$, 
followed by the edge $v_{j^2+i,j^2}v_{j^2,j^2+i}$,
followed by the path from $v_{j^2,j^2+i}$ to $r_{j^2}$ in $X_{j^2}$. 
Note that the $P_{i,j}$ are internally disjoint
(since for any $i<j$ and $\ell<k$ with $j<k$, we have $j^2+i<k^2+\ell$). 
So $G$ contains a $K_{\aleph_0}$ subdivision. 
%\end{proof}

\smallskip

%\begin{proof}
\cref{NoSubModel} $\Longrightarrow$ \cref{NoSub}: If $G$ contains a subdivision of $K_{\aleph_0}$, then it is easy to see that $G$ contains a model of $K_{\aleph_0}$ with branch sets of finite radius. 
\end{proof}

%%%%%%%%%%%%%%%%%%%%%%%%%%%%%%%%%%%%%%%%%%%%%%%%%%%%%%x
\subsection{\texorpdfstring{$K_t$-Minor-Free Graphs}{Kt-Minor-Free Graphs}}
\label{ExcludedMinorNoSubdiv}

The main result of this section, \cref{KtMinorFreeNoSubdiv}, says that for every integer $t\geq 3$, there is a graph $U$ that contains every $K_t$-minor-free graph as an induced subgraph, and $U$ satisfies key properties \cref{KeyPropertyBoundedMinDegree}, \cref{KeyPropertyBoundedColouringNumbers}
and \cref{KeyPropertyNoInfiniteSubdivision}. The main tool is based on chordal partitions, which have been used by several authors in the study of $K_t$-minor-free graphs~\cite{RS98,vdHW18,Andreae86,HOQRS17,SSW19}. Chordal partitions are useful since if $\PART$ is a connected chordal partition of a $K_t$-minor-free graph $G$, then the quotient $G/\PART$ is a minor of $G$, implying $G/\PART$ contains no $K_t$ subgraph and $\tw(G/\mathcal{P})\leq t-2$ by \cref{TreewidthCharacterisation}. Thus, using chordal partitions, the structure of $K_t$-minor-free graphs can be described in terms of much simpler graphs of treewidth at most $t-2$. We also use the machinery developed to construct a strongly universal treewidth-$k$ graph~(\cref{StronglyUniversalTreewidth}) and a graph that contains every locally finite graph as an induced subgraph, but contains no $K_{\aleph_0}$ subdivision~(\cref{StronglyUniversalLocallyFinite}).

\note{The referee made an important observation about the possibility of incompatible isomorphisms. To fix this we introduced ordered graphs and divided graphs.} 

All these results rely on the following definitions. 
An \defn{ordered set} consists of a countable set $S$ equipped with a linear ordering of $S$, denoted $S=(v_1,v_2,\dots)$. (Here a `linear ordering' has order-type that of the set $\NN$.)\ An \defn{ordered graph} $G$ consists of a graph $G$ and a linear ordering of $V(G)$. Subgraphs of an ordered graph $G$ are ordered according to the given ordering of $V(G)$. 
If a graph $G_1$ is ordered $(v_1,v_2,\dots)$
and a graph $G_2$ is ordered $(w_1,w_2,\dots)$, 
then write \defn{$G_1\cong G_2$} if for all integers $j>i\geq 1$, 
$$v_iv_j\in E(G_1) \text{ if and only if } w_iw_j\in E(G_2).$$
A \defn{divided graph} is a triple $(G,A,B)$ such that:
\begin{itemize}
    \item $G$ is a graph, 
    \item $V(G)=A\cup B$ and $A\cap B=\emptyset$,
    \item $A$ is an ordered set, and
    \item $B$ is an ordered set.
\end{itemize}
So $G[A]$ and $G[B]$ are ordered graphs. 
Suppose that $(G_1,A_1,B_1)$ and $(G_2,A_2,B_2)$ are divided graphs, where $A_i=(v_{i,1},v_{i,2},\dots)$ and
$B_i=(w_{i,1},w_{i,2},\dots)$ for each $i\in\{1,2\}$, 
and $|A_1|=|A_2|$ and $|B_1|=|B_2|$. 
Then we write 
\defn{$(G_1,A_1,B_1)\cong (G_2,A_2,B_2)$} if 
for all integers $j>i\geq 1$, 
\begin{align*}
& v_{1,i}v_{1,j}\in E(G_1) \text{ if and only if } v_{2,i}v_{2,j}\in E(G_2) \text{ and}\\
& w_{1,i}w_{1,j}\in E(G_1) \text{ if and only if } w_{2,i}w_{2,j}\in E(G_2),
\end{align*}
and for all integers $i,j\geq 1$, 
$$v_{1,i}w_{1,j}\in E(G_1) \text{ if and only if } v_{2,i}w_{2,j}\in E(G_2).$$
This implies that $G_1[A_1]\cong G_2[A_2]$ and
$G_1[B_1]\cong G_2[B_2]$.

Let $\XX:=\{(m,n):m\in\NN_0,n\in\NN,m<n\}$. 
Let $\HH$ be a set $\{H_{m,a}:m\in\NN_0,a\in\NN\}$ of 
%pairwise disjoint 
finite ordered graphs. 
An \hdefn{$\HH$}{system} is a set 
$$\JJ=\{ (J_{m,a,n,b,\ell},A_{m,a},B_{n,b}) :(m,n)\in\XX,\,a,b,\ell\in\NN\}$$ 
such that for all $(m,n)\in \XX$ and $a,b,\ell\in\NN$: 
\begin{itemize}
\item $(J_{m,a,n,b,\ell},A_{m,a},B_{n,b})$ is a divided graph;
\item  $J_{m,a,n,b,\ell}[A_{m,a}]\cong H_{m,a}$ and $J_{m,a,n,b,\ell}[B_{n,b}]\cong H_{n,b}$; 
\item if $\ell=1$ then there are no edges between $A_{m,a}$ and $B_{n,b}$ in $J_{m,a,n,b,\ell}$.
 \end{itemize}

For any $\HH$-system $\JJ$ and $k\in\NN$, let \defn{$\JJ^{(k)}$} be the set of all graphs $G$ with the following properties:
\begin{itemize}
\item there is a partition $\PART$ of $V(G)$, where each $A\in \PART$ is an ordered set (so $G[A]$ is an ordered graph);
\item there is a  tree $T$ with $V(T)=\PART$ rooted at some part $R\in\PART$ (and thus $T$ is oriented away from $R$); 
\item there is a $(k+1)$-colouring $c$ of $T$ such that $G/\PART$ is a spanning subgraph of $Q:=\GGG{T}{c}$ (which implies $\tw(G/\PART)\leq \tw(Q) \leq k$ by \cref{TreewidthCharacterisation}); and  
\item there is a labelling of $Q$, such that:
\begin{itemize}
\item for each part $A\in\PART$, 
if $\dist_T(R,A)=m$ and $A$ is labelled $a$, then $G[A]\cong H_{m,a}$; and
\item for every oriented edge $AB\in E(Q)$, 
\begin{itemize}
\item if $\dist_T(R,A)=m$ and $A$ is labelled $a\in\NN$ and $\dist_T(R,B)=n$ and $B$ is labelled $b\in\NN$ and $AB$ is labelled $\ell\in\NN$, then 
$(G[A\cup B],A,B)\cong (J_{m,a,n,b,\ell},A_{m,a},B_{n,b})$; and
%(which implies that $G[A]\cong H_{m,a}$ and $G[B]\cong H_{n,b}$); and
\item if $AB\in E(Q)\setminus E(G/\PART)$ then $\ell=1$. 
\end{itemize}
\end{itemize}
\end{itemize}

\begin{lem}
\label{JJJ}
For every set $\HH=\{H_{m,a}:m\in\NN_0 ,a\in\NN\}$ where each $H_{m,a}$ is a finite ordered graph, for every $\HH$-system $\JJ=\{ (J_{m,a,n,b,\ell},A_{m,a},B_{n,b}) :(m,n)\in\XX,\,a,b,\ell\in\NN\}$, and for every $k\in\NN$, 
\begin{itemize}
\item no graph in $\JJ^{(k)}$ contains a subdivision of $K_{\aleph_0}$;
\item $\JJ^{(k)}$ has a strongly universal graph $U_{\JJ,k}$;
\item if, for some $t\in\NN$, every graph in $\HH$ is $t$-colourable, then every graph in $\JJ^{(k)}$ is $t(k+1)$-colourable; and
\item if, for some $d,r\in \NN$, every graph in $\HH$ is $d$-degenerate, and for all $(m,n)\in\XX$ and $a,b,\ell\in\NN$, in the graph $J_{m,a,n,b,\ell}$, each vertex in $B_{n,b}$ has at most $r$ neighbours in $A_{m,a}$, then every finite subgraph of any graph in $\JJ^{(k)}$ is $(rk+d)$-degenerate.
\end{itemize}
\end{lem}

\begin{proof}
Let $G\in\JJ^{(k)}$. Let $\PART$ be the partition of $G$ and let $Q$ be the supergraph of $G/\PART$ witnessing that $G\in\JJ^{(k)}$. Let $G'$ be the graph obtained by adding edges to $G$ so that $A\cup B$ is a clique in $G'$ for each edge $AB\in E(Q)$. (This makes sense since $A,B\subseteq V(G)$.)\ Now, $\PART$ is a connected partition of $G'$ with quotient $Q$, so $\PART$ is a chordal partition. Since every graph in $\HH$ is finite, $\PART$ is a finite chordal partition. By \cref{InfiniteCompleteSubdivision}\cref{NoSubChordalPart}, $G'$ and thus $G$ contains no $K_{\aleph_0}$ subdivision. This proves the first claim. 

Let $\TT$ be the universal tree rooted at an arbitrary vertex $r$ (see \cref{UniversalTree}). Let $c$ be a $(k+1)$-colouring of $\TT$, such that each vertex $v$ of $\TT$ has infinitely many children of each colour distinct from the colour assigned to $v$. So $\TT_k=\GGG{\TT}{c}$ is the universal treewidth-$k$ graph (\cref{TreewidthUniversal}). 

%By \cref{TreewidthUniversalPreserving}, there is a rooted $k$-orientation and labelling of $\TT_k$, such that for every chordal graph $Z$ containing no $K_{k+2}$ subgraph, for every  well-founded $k$-orientation of $Z$, and for every labelling of $Z$, there is an orientation-preserving label-preserving isomorphism from $Z$ to a subgraph of $\TT_k$. 

The graph $U=U_{\JJ,k}$ is obtained from $\TT_k$ as follows. 
For each vertex $v$ of $\TT$ at distance $m$ from $r$ in $\TT$ and labelled $a\in\NN$, replace $v$ by a copy of the ordered graph $H_{m,a}$ with vertex-set $U_v\subseteq V(U)$, where $U_v\cap U_w=\emptyset$ for all distinct $v,w\in V(\TT_k)$. Then, for each oriented edge $vw$ of $\TT_k$ labelled $\ell\in \NN$, if $(m,n):=(\dist_{\TT}(r,v),\dist_{\TT}(r,w))\in\XX$ and $v$ is labelled $a\in\NN$ and $w$ is labelled $b\in\NN$, then add edges between $U_v$ and $U_w$ so that 
\begin{equation}
\label{Iso}
(U[ U_v \cup U_w ], U_v, U_w) \cong (J_{m,a,n,b,\ell},A_{m,a},B_{n,b}).
\end{equation}
So $\{U_v:v\in V(\TT_k)\}$ is a partition of $U$, where $\TT_k$ takes the role of $Q$ in the above definition of $\JJ^{(k)}$. So $U\in \JJ^{(k)}$. 

We now show that every graph $G$ in $\JJ^{(k)}$ is isomorphic to an induced subgraph of $U$. Let $\PART$, $T$, $R$, $c$ and $Q=\GGG{T}{c}$ be as in the definition of $\JJ^{(k)}$. Thus, there is a rooted $k$-orientation and labelling of $Q$, such that for every oriented edge $AB\in E(Q)$, 
if $(m,n):=(\dist_T(R,A),\dist_T(R,B))\in\XX$ and $A$ is labelled $a\in\NN$ and $B$ is labelled $b\in\NN$ and $AB$ is labelled $\ell\in\NN$, then 
\begin{equation}
\label{IsoIso}
(G[A\cup B],A,B)\cong (J_{m,a,n,b,\ell},A_{m,a},B_{n,b}), 
\end{equation} 
and if $AB\in E(Q)\setminus E(G/\PART)$ then $\ell=1$. By \cref{TreewidthUniversalPreservingInduced}, there is an 
orientation-preserving label-preserving isomorphism $\phi$ from $Q$ to an induced subgraph of $\TT_k$. 

Thus, for every part $A\in\PART$ labelled $a\in\NN$, we have that $\phi(A)$ is a vertex in $\TT_k$ labelled $a$. Let 
$$U_G:=U\left[\bigcup_{A\in\PART} U_{\phi(A)}\right].$$ 
Since $\phi$ is orientation-preserving and label-preserving, for every oriented edge $AB\in E(Q)$ labelled $\ell\in\NN$, if $A$ is labelled $a\in\NN$ and $B$ is labelled $b\in\NN$, then $\phi(A)\phi(B)$ is an edge in $\TT_k$ labelled $\ell$, and  
\begin{equation}
\label{IsoIsoIso}
(G[A\cup B],A,B)\cong 
(J_{m,a,n,b,\ell},A_{m,a},B_{n,b}) \cong 
(U[ U_{\phi(A)} \cup U_{\phi(B)} ], U_{\phi(A)}, U_{\phi(B)} ). 
\end{equation}
Note that to conclude that $G[A\cup B]\cong J_{m,a,n,b,\ell}$ in the case that $AB\in E(Q)\setminus E(G/\PART)$, we use the fact that $AB$ is labelled $\ell=1$, implying that there is no edge between $A_{m,a}$ and $B_{n,b}$ in $J_{m,a,n,b,\ell}$ (by the definition of $\HH$-system).

Consider a vertex $v$ of $G$. Let $A$ be the part of $\PART$ containing $v$. \note{revisions have been made here} For each edge of $\TT_k$ incident to $A$, $v$ is mapped to the same vertex in $U_{\phi(A)}$ under the isomorphisms described in \cref{IsoIsoIso} (because the given ordering of $G[A]$ is preserved). 
Let $\psi(v)$ be this vertex. So $\psi(v)$ is in $U_G$. 

For every edge $vw\in E(G)$, if $v\in A\in\PART$ and $w\in B\in\PART$, then $A=B$ or $AB\in E(Q)$, implying $\psi(v)\psi(w)\in E(U_G)$ by the isomorphisms in \cref{IsoIsoIso}. So $G$ is isomorphic to a spanning subgraph of $U_G$. We claim that in fact $G$ is isomorphic to $U_G$. Consider an edge $xy$ of $U_G$. So $x\in U_{\phi(A)}$ and $y\in U_{\phi(B)}$ for some $A,B\in V(Q)$. 
Since $\phi(Q)$ is an induced subgraph of $\TT_k$, either $A=B$ or $AB\in E(Q)$. Thus $\psi^{-1}(x)\psi^{-1}(y)$ is an edge of $G$ by the isomorphisms in \cref{IsoIsoIso}. Hence $G$ is isomorphic to $U_G$, which proves the second claim. 

If every graph in $\HH$ is $t$-colourable, then $\chi(G/\PART)\leq\tw(G/\PART)+1\leq k+1$, and a product colouring gives a $t(k+1)$-colouring of $G$. This proves the third  claim. 

For the fourth claim, each graph in $\HH$ has an acyclic orientation with in-degree at most $d$ at each vertex. Say $G\in\JJ^{(k)}$. Let $\PART$ and $Q$ be as in the definition of $\JJ^{(k)}$. So $Q$ has a $k$-orientation. Each $A\in \PART$ has in-degree at most $k$ in $Q$, and the subgraph $G[A]$ (which is isomorphic to a graph in $\HH$) has an acyclic orientation with in-degree at most $d$ at each vertex. For each edge $AB\in E(Q)$, orient each $AB$-edge of $G$ from $A$ to $B$. By assumption, each vertex in $B$ has at most $r$ neighbours in $A$. In total, the in-degree of each vertex in $G$ is $kr+d$. Thus each finite subgraph of $G$ has minimum degree at most $kr+d$. 
\end{proof}

A very simple special case of the above construction gives a strongly universal graph for treewidth at most~$k$.

\begin{prop}
\label{StronglyUniversalTreewidth}
For each $k\in\NN$ there is a strongly universal treewidth-$k$ graph.
\end{prop}

\begin{proof}
Let $\HH:=\{H_{m,1}:m\in\NN_0\}$ where $H_{m,1}$ is an ordered $K_1$ for each $m\in\NN_0$.  
Let $J'$ be the graph with vertex-set $\{v,w\}$ and edge-set $\{vw\}$. 
Let $J''$ be the graph with vertex-set $\{v,w\}$ and edge-set $\emptyset$. For $(m,n)\in\XX$ and $a,b,\ell\in \NN$, let $A_{m,a}:=\{v\}$ and $B_{n,b}:=\{w\}$; 
let $J_{m,a,n,b,\ell}:=J'$ if $\ell$ is even, and let $J_{m,a,n,b,\ell}:=J''$ if $\ell$ is odd. 
So $\JJ=\{ (J_{m,a,n,b,\ell},A_{m,a},B_{n,b}) :(m,n)\in\XX,a,b,\ell\in\NN\}$ is an  $\HH$-system. For every graph $G$ in $\JJ^{(k)}$, if $\PART$ is the partition of $G$ in the definition of $\JJ^{(k)}$, then since each part is a singleton, $\tw(G)=\tw(G/\PART)\leq k$.

We now show that every graph $G$ of treewidth $k$ is in $\JJ^{(k)}$. 
By \cref{TreewidthCharacterisation}, $G$ is a spanning subgraph of a chordal graph $Q=\GGG{T}{c}$ containing no $K_{k+2}$ subgraph, for some spanning tree $T$ of $Q$ rooted at $R$, and for some $(k+1)$-colouring $c$ of $T$. Let $\PART$ be the partition of $G$ and of $Q$ with one part for each vertex. So $\PART$ is a chordal partition of $Q$. Label every vertex of $G/\PART$ by 1. The orientation of $T$ away from the root determines a well-founded $k$-orientation of $Q$. \note{minor revision here} 
Consider an oriented edge $AB$ of $Q$.
Let $m:=\dist_T(R,A)$ and $n:=\dist_T(R,B)$. 
If $AB\in E(G/\PART)$, then label $AB$ by 2, so 
$(G[A\cup B],A,B)\cong(J',\{v\},\{w\})=(J_{m,1,n,1,2},A_{m,1},B_{n,1})$. 
If $AB\not\in E(G/\PART)$, then label $AB$ by 1, so 
$(G[A\cup B],A,B)\cong(J'',\{v\},\{w\})=(J_{m,1,n,1,1},A_{m,1},B_{n,1})$.
Since $J'$ has an edge and $J''$ does not, $\PART$ is a partition of $G$, where $T$ and $Q$ satisfy the definition of $\JJ^{(k)}$. Thus $G$ is in $\JJ^{(k)}$.

Hence $\JJ^{(k)}$ is exactly the set of all graphs with treewidth at most $k$. By \cref{JJJ}, $\JJ^{(k)}$ has a strongly universal graph $U_{\JJ,k}$, which therefore is a strongly universal treewidth-$k$ graph. 
\end{proof}

The following definition and lemmas are used in our proof for $K_t$-minor-free graphs below. If $S$ is a set of $k\geq 2$ vertices in a connected graph $G$ and $r\in S$, then a set $X\subseteq V(G)$ is an \hdefn{$S$}{connector} (\DefNoIndex{rooted} at $r$) if there are geodesic paths $P_1,\dots,P_{k-1}$ in $G$ starting at $r$, such that $S\subseteq X=\bigcup_{i=1}^{k-1} V(P_i)$, and $G[X]$ is a connected subgraph with bandwidth at most $2k-3$, and every vertex in $G-X$ has at most $3k-3$ neighbours in $X$. An $S$-connector is called an \hdefn{$h$}{connector} for any integer $h\geq |S|$. The next lemma is implicit in \citep{vdHW18}; we include the proof for completeness. 

\note{In the above definition, we previously had ``bandwidth at most $k-1$, and every vertex in $G-X$ has at most $2k-2$ neighbours in $X$''. As pointed out by the referee, this was wrong. The definition has been fixed, and a detailed proof of 
\cref{ConnectingSubgraph2} has been added.} 

\begin{lem}
\label{ConnectingSubgraph2}
For any finite set $S$ of $k\geq 2$ vertices in a connected graph $G$, and for any vertex $r\in S$, there is an $S$-connector in $G$ rooted at $r$. 
\end{lem}

\begin{proof}
    Let $S=\{r,s_1,\dots,s_{k-1}\}$.
    For $i\in[k-1]$, let $P_i$ be a geodesic path from $r$ to $s_i$ and  
    let $X:=\bigcup_{i=1}^{k-1} V(P_i)$. For each integer $d\geq 0$, each $P_i$ has at most one vertex at distance $d$ from $r$. 
    So $X$ has at most $k-1$ vertices at distance $d$ from $r$. 
   Let $(v_1,v_2,\dots)$ be an ordering of  $X$ by distance from $r$, breaking ties arbitrarily. 
    For each edge $v_iv_j$ of $G[X]$, if $d:=\dist_G(r,v_i)$ and $d':=\dist_G(r,v_j)$, then $|d-d'|\leq 1$. 
    Thus $|i-j|\leq 2k-3$. 
    Hence $G[X]$ has bandwidth at most $2k-3$. 
    For $i\in[k-1]$, since $P_i$ is a geodesic, each vertex $w$ in $G-X$ has at most three neighbours in $P_i$.
    Thus $w$ has at most $3k-3$ neighbours in $X$. 
    Hence $X$ is an $S$-connector. 
\end{proof}

We now define a set $\HH$ of finite ordered graphs and an $\HH$-system $\JJ$ to be used in \cref{MinorJJJ} below. This system is designed so that \cref{PartitionGenColNum} can be used to show that graphs in $\JJ^{(t-2)}$ have bounded generalised colouring numbers. 

Fix $t\in\NN$ with $t\geq 3$. For $m\in\NN_0$, let $\HH_m$ be the class of all connected finite graphs $H$ with $V(H)\subseteq \NN_0^{m+1} \times [t-2]$, where each $v\in V(H)$ is denoted by $v=(v_0,v_1,\dots,v_m,v_\star) \in \NN_0^{m+1}\times[t-2]$, such that:
\begin{enumerate}[(a)]
\item for all $v,w\in V(H)$ if $v_m=w_m$ and $v_\star=w_\star$, then $v=w$; 
\item there is exactly one vertex $r\in V(H)$ with $r_m=0$; 
\item for every edge $vw\in E(H)$ and for every $i\in[0,m]$, we have $|v_i-w_i|\leq 1$;
\item for each $i\in[t-2]$, the set $\{r\}\cup \{v\in V(H): v_\star = i \}$ induces a path in $H$.
\end{enumerate}
\note{In the above definition, we previously had property (e) saying $H$ has bounded bandwidth, but this is implied by the above conditions. So we removed it.}
The case with $m=0$ is illustrated in \cref{HnExample}.

Consider each graph $H\in \HH_m$ to be ordered, first by the $v_m$ coordinate and then by the $v_\star$ coordinate.  By (a) there are at most $t-2$ vertices for each given $v_m$ coordinate. By (c) with $i=m$, under this ordering, $H$ has bandwidth at most $2t-5$.  

\begin{figure}[!ht]
\centering
\includegraphics{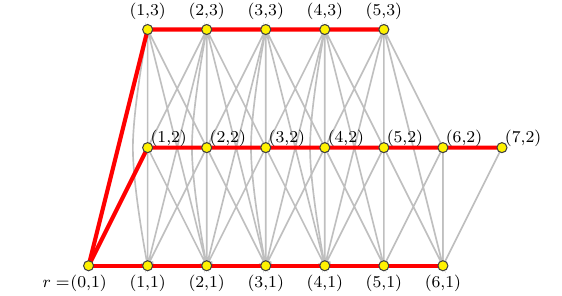}
\caption{A graph in $\HH_0$ with $t=5$. The red (thick) edges must be present, the grey (thin) edges are possibly present.}
\label{HnExample}
\end{figure}

\note{We previously claimed without proof that the path induced by 
$\{r\}\cup \{v\in V(H): v_\star = i \}$
is a geodesic in $H$. The referee said ``This is false. See example in referee report.'' We don't understand this figure in the referee report. We have written a detailed proof of the claim.}

\begin{lem}
For every graph $H\in\HH_m$    and $i\in[t-2]$ 
the path induced by 
$\{r\}\cup \{v\in V(H): v_\star = i \}$
is a geodesic in $H$ with end $r$.
\end{lem}

\begin{proof}
Let $P:=\{r\}\cup \{v\in V(H): v_\star = i \}$. By assumption, $H[P]$ is a path. If $|P|=1$ then $H[P]$ is a 1-vertex geodesic (with vertex-set $\{r\}$), and we are done. Now assume that $|P|\geq 2$. Let $v$ be a neighbour of $r$ in $H[P]$. By (b), $r_m=0$. By (c), $|v_m-r_m|=|v_m|\leq 1$. By (b), $v_m\neq 0$. Thus $v_m=1$. Since $v_\star=i$, by (a), there is only one such vertex $v$. Hence $r$ is an end of $H[P]$. 

Suppose that $H[P]=(v^0,v^1,v^2,\dots,v^n)$ where $r=v^0$. We now prove that $v^j_m=j$ by induction on $j\geq 0$. The cases where $j\in\{0,1\}$ are proved above. Now assume that $j\geq 2$ and that the claim holds up to $j-1$. Since $v^{j-1}v^j\in E(H)$, by (c), $|v^j_m-v^{j-1}_m|=|v^j_m-(j-1)|\leq 1$. By (a), $v^j_m\neq v^{j-1}_m=j-1$ and  
$v^j_m\neq v^{j-2}_m=j-2$. Thus $v^j_m=j$ as claimed. 

Consider any $v^0v^n$-path $Q$ in $H$. For any edge $ab\in E(Q)$, by (c), $|a_m- b_m|\leq 1$. Since $v^0_m=0$ and $v^n_m=n$, we have $|E(Q)|\geq n$. Thus, $H[P]$ is a geodesic, as claimed. 
\end{proof}

Since $\HH_m$ is a countable union of finite sets, $\HH_m$ is countable. Let $H_{m,1},H_{m,2},\dots$ be an arbitrary enumeration of $\HH_m$. 
Let $\HH:=\{H_{m,a}:m\in\NN_0,a\in\NN\}$.
For $(m,n)\in\XX$ and $a,b\in\NN$, let $A_{m,a}:=V(H_{m,a})$ and $B_{n,b}:=V(H_{n,b})$ (which are ordered sets), and 
let $\JJ_{m,a,n,b}$ be the set of all divided graphs $(J,A_{m,a},B_{n,b})$ such that:
\begin{itemize}
    \item for all $v\in A_{m,a}$ and $w\in B_{n,b}$, if $vw\in E(J)$ then 
    $|v_i-w_i|\leq 1$ for each $i\in[0,m]$; and
    \item every vertex in $B_{n,b}$ has at most $3t-6$ neighbours in $A_{m,a}$.
\end{itemize}
Note that $\JJ_{m,a,n,b}$ is countable. 
Let $(J_{m,a,n,b,\ell})_{\ell\in\NN}$ be an enumeration of $\JJ_{m,a,n,b}$, where $J_{m,a,n,b,1}$ is the divided graph $(J,A_{m,a},B_{n,b})\in \JJ_{m,a,n,b}$ with no edges between $A_{m,a}$ and $B_{n,b}$. Now $\JJ=\{(J_{m,a,n,b,\ell},A_{m,a},B_{n,b}):(m,n)\in\XX,\,a,b,\ell\in\NN\}$ is an $\HH$-system. For $(m,n)\in\XX$, let $\JJ_{m,n}:=\bigcup\{\JJ_{m,a,n,b}:a,b\in\NN\}$. 

The next lemma is inspired by an analogous result for finite graphs by \citet{vdHW18}. 

\begin{lem}
\label{MinorJJJ}
Every $K_t$-minor-free graph $G$ is in $\JJ^{(t-2)}$. 
\end{lem}

\begin{proof}
Since $\TT_k$ contains infinitely many disjoint copies of $\TT_k$ as an induced subgraph, it suffices to assume that $G$ is connected. We may also assume that $G$ is infinite. Let $u_1,u_2,\dots$ be an arbitrary enumeration of $V(G)$. Below we define $(\PART_i,T_i,S_i,c_i)_{i\in\NN}$, where for $i\in\NN$, 
\begin{enumerate}[(P1)] 
 \item $\PART_i$ is a connected chordal partition of $G$ (that is, $G/\PART_i$ is a chordal graph, and $G[A]$ is a connected ordered graph for each $A\in\PART_i$);
 \item $T_i$ is a spanning tree of $G/\PART_i$ and $c_i$ is a $(t-1)$-colouring of $T_i$ such that $G/\PART_i$ is a spanning subgraph of $\GGG{T_i}{c_i}$;
\item  $S_i$ is a subtree of $T_i$ with exactly $i$ nodes, and every node of $T_i-V(S_i)$ is a leaf of $T_i$;
\item $S_1\subseteq S_2\subseteq \dots\subseteq S_i $;
\item for $j\in[i]$ and for each part $A\in\PART_j$, we have $c_j(A)=c_i(A)$;
\item both $S_i$ and $T_i$ are rooted at part $R\in\PART_i$ where $R:=\{u_1\}$;
 \item for each part $B\in V(S_i)$, if $m:=\dist_{S_i}(R,B)$ and $R=X_0,X_1,\dots,X_m=B$ is the $RB$-path in $S_i$, and $B^+$ is the vertex-set of the component of $G-(X_0\cup\dots\cup X_{m-1})$ containing $B$, then 
\begin{itemize}
\item $B$ is a $(t-1)$-connector in $G[B^+]$ rooted at some vertex $r_B$; and 
\item $G[B]$ is isomorphic to some graph in $\HH_m$, where each vertex $v\in B$ is mapped to $(v_0,v_1,\dots,v_m,v_\star)$, where $v_j=\dist_{G[X^+_j]}(r_{X_j},v)$ for each $j\in[0,m]$; and 
\end{itemize}
\item for each directed edge $AB$ of $G/\PART_i$ where $A,B\in V(S_i)$, if $(m,n):=(\dist_{S_i}(R,A),\dist_{S_i}(R,B))\in\XX$, then the divided graph $(G[A\cup B], A,B)$ is isomorphic to some divided graph in $\JJ_{m,n}$,
\item for each directed edge $AB$ of $G/\PART_i$ where $A\in V(S_i)$ and $B\in V(T_i)\setminus V(S_i)$, each vertex in $B$ has at most $3t-6$ neighbours in $A$. 
\end{enumerate}

We first define $(\PART_1,T_1,S_1,c_1)$. Let $(B_j)_{j\in J}$ be the vertex-sets of the components of $G-u_1$. Let $\PART_1$ be the partition of $G$ with parts $R:=\{u_1\}$ and $(B_j)_{j\in J}$. So $G/\PART_1$ is a star with $|J|$ leaves. Let $T_1$ be this star. Colour $R$ by 1 and colour each node $B_j$ by 2. So $T_1$ is a spanning tree of $G/\PART_1$ and $c_1$ is a $(t-1)$-colouring of $T_1$ such that $G/\PART_1=T_1=\GGG{T_1}{c_1}$. Let $S_1$ be the tree with one vertex $R$. Consider $T_1$ and $S_1$ to be rooted at $R$. Then $(\PART_1,T_1,S_1,c_1)$ satisfies the above conditions. 

Note that (P1)--(P3) imply that the nodes of $S_i$ and $T_i$ are parts of $\PART_i$. So (P4) implies that for $j\in[i]$ each part of $\PART_j$ that is in $S_j$ is also a part of $\PART_i$ in $S_i$. Note that (P4) also implies that $B^+$ (defined in (P7)) does not depend on $i$.

Suppose that $(\PART_1,T_1,S_1,c_1),\dots,(\PART_i,T_i,S_i,c_i)$ are defined. Since $|V(S_i)|=i$ and each part of $\PART_i$ in $S_i$ is finite, some vertex of $G$ is in a part of $\PART_i$ that is a node of $T_i-V(S_i)$. Let $j$ be the minimum integer such that $u_j$ is in a part $A\in \PART_i$ that is a node of $T_i-V(S_i)$. So $A$ is a leaf of $T_i$ by (P3). Let $B_1,\dots,B_s$ be the neighbours of $A$ in $G/\PART_i$. So $\{A,B_1,\dots,B_s\}$ is a clique in $G/\PART_i$, implying  $s\in[0,t-2]$. 
Let $B_s$ be the parent of $A$ in $T_i$. So $B_1,\dots,B_s$ are ancestors of $A$ in $T_i$. Let $m:=\dist_{T_i}(R,B_s)$. Let $R=X_0,X_1,\dots,X_m=B_s$ be the $RB_s$-path in $T_i$. For $\ell\in[0,m]$, let $X^+_\ell$ be the vertex-set of the component of $G-(X_0\cup\dots\cup X_{\ell-1})$ containing $X_\ell$. By (P7), $X_\ell$ is a $(t-1)$-connector in $G[X^+_\ell]$ rooted at some vertex $r_{X_\ell}$. Note that $A\subseteq X^+_\ell$. 

For $p\in[s]$, let $a_p$ be any vertex in $A$ adjacent to some vertex in $B_p$, which exists since $AB_p\in E(G/\PART_i)$. Let $C$ be a $\{u_j,a_1,\dots,a_s\}$-connector in $G[A]$ rooted at $u_j$, which exists by \cref{ConnectingSubgraph2}. Let $P_1,\dots,P_s$ be the corresponding geodesic paths. So $C=\bigcup_{p=1}^s V(P_p)$, and $C$ is a $(t-1)$-connector in $G[A]$. 

Let $\PART_{i+1}$ be the partition of $G$ obtained from $\PART_i$ by replacing $A$ by $C$ and the vertex-sets of the components of $G[A]-C$. For $p\in[s]$, since $a_p\in C$, we have $B_pC$ is an edge of $G/\PART_{i+1}$. 
So $\{C,B_1,\dots,B_s\}$ is a clique in $G/\PART_{i+1}$. For each component $X$ of $G[A]-C$, the neighbourhood of $V(X)$ in $G/\PART_{i+1}$ is a subset of $\{C,B_1,\dots,B_s\}$. 
So $G/\PART_{i+1}$ is chordal. By construction, each part of $\PART_{i+1}$ is connected. Thus (P1) is satisfied for $i+1$. 

Let $T_{i+1}$ be the tree obtained from $T$ by renaming the node $A$ by $C$, and by adding, for each component $G[X]$ of $G[A]-C$, the node $X$ and the edge $XC$ to $T_{i+1}$. So $X$ is a leaf in $T_{i+1}$. Let $c_{i+1}$ be the colouring of $T_{i+1}$ obtained from $c_i$ as follows. 
Let $c_{i+1}(C):=c_i(A)$. 
So $C$ has the same in-neighbourhood in $\GGG{T_{i+1}}{c_{i+1}}$ as $A$ in $\GGG{T_i}{c_i}$.
For each component $G[X]$ of $G[A]-C$, the neighbourhood of $X$ in $G/\PART_{i+1}$ is a subset of the clique $\{C,B_1,\dots,B_s\}$ of size at most $t-2$ (since $G$ has no $K_t$ minor). 
Let $c_{i+1}(X)$ be a colour not used by its neighbours in $G/\PART_{i+1}$, which exists since there are $t-1$ colours. 
Since $B_pA$ is an edge of $G/\PART_i$ for each $p\in[s]$, 
if $B_pX$ is an edge of $G/\PART_{i+1}$, then $B_pX$ is an edge of $\GGG{T_{i+1}}{c_{i+1}}$. 
Thus $G/\PART_{i+1}$ is a spanning subgraph of $\GGG{T_{i+1}}{c_{i+1}}$, and (P2) holds for $i+1$. (P5) holds by the definition of $c_{i+1}$. 

Let $S_{i+1}$ be the tree obtained from $S_i$ by adding a new node $C$ adjacent to $B_s$. Thus (P3), (P4) and (P6) are satisfied for $i+1$. 

Now we show that (P7) holds for $i+1$. By construction, $C$ is a $(t-1)$-connector in $G[A]$ and $A=C^+$ (where $C^+$ is the vertex-set of the component of $G-(X_0\cup\dots\cup X_m)$ containing $C$). Thus $C$ is a $(t-1)$-connector in $G[C^+]$. For each $v\in C$, let $v_{m+1}:=\dist_{G[A]}(u_j,v)$, and for each $\ell\in[0,m]$, let $v_\ell:=\dist_{G[X^+_\ell]}(r_\ell,v)$. Thus, $|v_\ell-w_\ell|\leq 1$ for each edge $vw$ in $G[C]$ and for each $\ell\in[m+1]$. For each $v\in C$, let $v_\star$ be an integer $p\in[s]$ such that $v_\star$ is in $P_p$. Since $P_p$ is a geodesic, if $v_{m+1}=w_{m+1}$ and $v_\star=w_\star$, then $v=w$. Thus $\{v\in C: v_\star = p \}$ induces a path in $H$, namely $P_p$. This shows that $G[C]$ is isomorphic to a graph in $\HH_{m+1}$, and (P7) holds for $i+1$.

Next, we show that (P8) holds for $i+1$. Consider $p\in [s]$. Let $\ell:=\dist_{T_i}(R,B_p)$. Thus $G[B_p]$ is isomorphic to some graph in $\HH_\ell$. By (P9), each vertex in $C$ has at most $3t-6$ neighbours in $B_p$.  
For $v\in B_p$ and $w\in C$ and $j\in[\ell]$ 
we have $v_j=\dist_{G[B^+_p]}(r_{B_p},v)$ (by (P7)) and
$w_j=\dist_{G[B^+_p]}(r_{B_p},w)$ (by the definition of $w_j$); 
so if $vw\in E(G)$ then $|v_j-w_j|\leq 1$. Hence the divided graph $(G[B_p\cup C],B_p,C)$ is isomorphic to some divided graph in $\JJ_{\ell,m+1}$. This shows that (P8) holds for $i+1$. 

Finally, by \cref{ConnectingSubgraph2}, every vertex in $G[A]-C$ has at most $3t-6$ neighbours in $A$, implying (P9) is satisfied for $i+1$. 

This shows that $(\PART_1,T_1,S_1,c_1),\dots,(\PART_{i+1},T_{i+1},S_{i+1},c_{i+1})$ satisfy (P1)--(P9), which completes the definition of $(\PART_j,T_j,S_j,c_j)_{j\in\NN}$.

By (P1)--(P4), $S_1,S_2,\dots$ are trees, where each node of $S_i$ is a part of $\PART_i$, and $S_1\subseteq S_2\subseteq \dots$. For $i\in\NN$, let $\PART'_i$ be the subset of $\PART_i$ corresponding to parts in $S_i$. Thus a part of $\PART'_i$ is also a part of $\PART'_\ell$ for all $\ell\geq i$; that is, $\PART'_i\subseteq \PART'_\ell$. Let $\PART:= \bigcup_{i\in\NN}\PART'_i$. We claim that $\PART$ is a partition of $G$. By construction, each vertex of $G$ is in at most one part of $\PART$. Consider vertex $u_j$ of $G$. If $u_j$ is no part of $\PART'_j$, then each of the $j$ nodes of $S_j$ contains a vertex $u_\ell$ with $\ell<j$ (by the choice of connector roots), which is a contradiction. So $u_j$ is in some part of $\PART'_j$. Hence $\PART$ is a partition of $G$. Let $S:=\bigcup_{i\in\NN}S_i$. Then $S$ is a spanning tree of $G/\PART$ rooted at $R=\{u_1\}$. For each part $A\in\PART$, let $c(A):=c_i(A)$ for any $i\in\NN$ for which $A\in\PART'_i$. The choice of $i$ is irrelevant by (P5). Every edge of $G/\PART$ is an edge of $\GGG{S_i}{c_i}$ for some $i\in\NN$, so $G/\PART$ is a spanning subgraph of $\GGG{S}{c}$. For each part $A\in \PART$, if $m:=\dist_S(R,A)$,  then by (P7), $G[A]\cong H_{m,a}$ for some $a\in \NN$; label $A$ by $a$. 
For each oriented edge $AB\in E(G/\PART)$ with $(m,n):=(\dist_S(R,A),\dist_S(R,B))\in\XX$, by (P8) the divided graph $(G[A\cup B],A,B)$ is isomorphic to some divided graph $(J_{m,a,n,b,\ell},A_{m,a},B_{n,b}) \in\JJ_{m,n}$; label $AB$ by $\ell$. Therefore $G\in\JJ^{(t-2)}$ where $Q:=G/\PART$.
\end{proof}

\begin{thm}
\label{KtMinorFreeNoSubdiv}
For each integer $t\geq 3$, there is a graph $U$ with the following properties:
\begin{itemize}
\item $U$ contains every $K_t$-minor-free graph as an induced subgraph, 
\item $U$ contains no subdivision of $K_{\aleph_0}$,
\item $U$ is $(2t-4)(t-1)$-colourable, 
\item $U$ is $(t-1)(3t-7)$-degenerate, 
\item $\col_r(U) \leq (t-2)(t-1)(2r+1)$ for every $r\in\NN$.
\end{itemize}
\end{thm}

\begin{proof}
Let $\HH$ and $\JJ$ be as defined prior to \cref{MinorJJJ}. By \cref{JJJ}, there is a strongly-universal graph $U:=U_{\JJ,t-2}$ in $\JJ^{(t-2)}$ that contains no subdivision of $K_{\aleph_0}$. By \cref{MinorJJJ}, every $K_t$-minor-free graph $G$ is in $\JJ^{(t-2)}$. So $G$ is isomorphic to an induced subgraph of $U$. 

Since $\chi(H)\leq\bw(H)+1$ for every graph $H$, every graph in $\HH$ is $(2t-4)$-colourable. Thus the third part of \cref{JJJ} is applicable with $b=2t-4$, implying $U$ is $(2t-4)(t-1)$-colourable. By assumption, the fourth part of \cref{JJJ} is applicable with $d=2t-5$ (since degeneracy is at most the bandwidth) and $k=t-2$ and $r=3t-6$. Thus every finite subgraph of $U$ is $(rk+d)$-degenerate, where $rk+d=(3t-6)(t-2)+(2t-5)=(t-1)( 3t-7)$.

We now show that $\col_r(U) \leq (t-2)(t-1)(2r+1)$. \note{Minor revisions here, addressing a comment from the referee.} Let $\PART$ be the chordal partition of $U$ witnessing that $U\in \JJ^{(t-2)}$. By \cref{ChordalCharacterisation2}, there is an acyclic $(t-2)$-orientation of $U/\PART$. Let $H$ be a finite subgraph of $U$. Let $\YY$ be the set of parts in $\PART$ that intersect $V(H)$. \note{The referee crossed out ``that intersect $V(H)$'', but we think this phrase is necessary.} Let $H^+$ be the subgraph of $U$ induced by the union of all the parts in $\YY$ that intersect $V(H)$. Since $H$ is finite and each part of $\YY$ is finite, $H^+$ is finite. Note that $H^+/\YY$ is an induced subgraph of $U /\PART$. So $\YY$ is a chordal partition of $H^+$, and the $(t-2)$-orientation of $U/\PART$ defines a $(t-2)$-orientation of $H^+/\YY$. Since $\YY$ is finite, there is an ordering $Y_1,\dots,Y_n$ of the parts in $\YY$, such that $i<j$ for every oriented edge $Y_iY_j\in E(H^+/\YY)$. ($Y_1,\dots,Y_n$ is called a \defn{perfect elimination ordering} in the literature on chordal graphs.)

Consider $Y_i\in\YY$. There are at most $t-2$ neighbouring parts $Y_j\in N_{H^+/\YY}(Y_i)$ with $j<i$ (and they form a clique in $H^+/\YY$, although we will not need this property). Let $H^+_i$ be the connected component of $H^+-(Y_1\cup\dots\cup Y_{i-1})$ that contains $Y_i$. Observe that $H^+_i$ is precisely the subgraph of $H^+$ induced by the union of all parts $Y_j\in\YY$ such that $i=j$ or $Y_j$ is a descendent of $Y_i$ in $H^+/\YY$. By the definition of $\HH$, $V(Y_i)$ is the union of the vertex-sets of $t-2$ geodesic paths $P_1,\dots,P_{t-2}$ in $H^+[Y_i]$. We claim that $P_1,\dots,P_{t-2}$ are geodesics in $H^+_i$. Let $x$ and $y$ be vertices in some $P_j$. Let $P$ be any $xy$-path in $H^+_i$. Say $P$ has length $\ell$. By the definition of $\HH_n$ and $\JJ_{m,n}$, for each edge $vw$ in $P$, we have $|v_i-w_i|\leq 1$. Thus $|x_i-y_i|\leq\ell$, implying $\dist_{P_j}(x,y)\leq\ell$. 
Thus $P_1,\dots,P_{t-2}$ are geodesics in $H^+_i$. 
Hence \cref{PartitionGenColNum} is applicable to $H^+$ with $d=p=t-2$. Thus $\col_r(H) \leq \col_r(H^+) \leq (t-2)(t-1)(2r+1)$. 
By \cref{colExtendable}, $\col_r$ is extendable.
Therefore $\col_r(U) \leq (t-2)(t-1)(2r+1)$. 
\end{proof}

%%%%%%%%%%%%%%%%%%%%%%%%%%%%%%%%%%%%
\subsection{Locally Finite Graphs}
\label{LocallyFiniteGraphs}

We have shown above that there exists a graph that contains every planar graph as an induced subgraph, but contains no subdivision of $K_{\aleph_0}$. Indeed, the result holds for $K_t$-minor-free graphs. We now show the same conclusion holds for all locally finite graphs.

\begin{prop}
\label{StronglyUniversalLocallyFinite}
There is a graph that contains every locally finite graph as an induced subgraph, but contains no subdivision of $K_{\aleph_0}$. 
\end{prop}

\begin{proof}
\note{Minor revisions here in response to referee comments. Also revised according to the `divided graph' setup. }
Let $\FF$ be the set of all finite ordered graphs, with one representative from each isomorphism class. So $\FF$ is countable. For each $m\in\NN_0$, let $(H_{m,a})_{a\in\NN}$ be an enumeration of $\FF$. 
Let $\HH:=\{H_{m,a}:m\in\NN_0,a\in\NN\}$. 
For $(m,n)\in \XX$ and $a,b\in\NN$, let $\JJ_{m,a,n,b}$ be the set of all graphs $J$ that can be obtained from the disjoint union of $H_{m,a}$ and $H_{n,b}$ by adding edges between 
$H_{m,a}$ and $H_{n,b}$. 
So $\JJ_{m,a,n,b}$ is countable. 
Let $(J_{m,a,n,b,\ell})_{\ell\in\NN}$ be an enumeration of $\JJ_{m,a,n,b}$, where $J_{m,a,n,b,1}$ is the graph 
with no edges between $H_{m,a}$ and $H_{n,b}$. 
Consider the $\HH$-system, 
$$\JJ=\{ (J_{m,a,n,b,\ell},A_{m,a},B_{n,b}) :(m,n)\in\XX,\,a,b,\ell\in\NN\}.$$ 

By \cref{JJJ}, $\JJ^{(1)}$ has a strongly universal graph $U_{\JJ,1}$ that contains no subdivision of $K_{\aleph_0}$. We now prove that every locally finite graph $G$ is in $\JJ^{(1)}$. 

Let $G_1,G_2,\dots$ be the connected components of $G$. For $i\in\NN$, let $r_i$ be any vertex of $G_i$. For $j\in\NN_0$, let $V_{i,j}$ be the set of vertices in $G_i$ at distance $j$ from $r_i$. Note that $V_{i,0}=\{r_i\}$. Fix $i\in\NN$. We prove by induction on $j\in\NN_0$ that $V_{i,j}$ is finite. The base case holds since $|V_{i,0}|=1$. Assume that $V_{i,j}$ is finite. Let $\Delta_{i,j}$ be the maximum degree of vertices in $V_{i,j}$. Thus $|V_{i,j+1}|\leq \Delta_{i,j}\,|V_{i,j}|$, and $V_{i,j+1}$ is finite. 

So $\PART:=\{V_{i,j}:i\in\NN,j\in\NN_0\}$ is a partition of $G$. The quotient $G/\PART$ is the disjoint union of infinite 1-way paths (one for each component of $G$). Let $T$ be the graph obtained from $G/\PART$ by adding an edge between $V_{i,0}$ and $V_{i+1,0}$ for each $i\in\NN$. So $T$ is a tree, which has a 1-orientation obtained by directing the edges from $V_{i,j}$ to $V_{i,j+1}$, and from $V_{i,0}$ to $V_{i+1,0}$. Moreover, $R:=V_{1,0}$ is the root of $T$ (with in-degree 0). Note that $\dist_T(R,V_{i,j})=i+j-1$. Let $c$ be a proper 2-colouring of $T$. Let $Q:=\GGG{T}{c}$ which equals $T$.

%\tony{For the proof of Proposition 6.8, I noticed that the divided graph indices at the end of proof are abbreviated and technically do not match the notation used at the beginning of the proof. For example $(J_{a,b,\ell},H_a,H_b)$, $H_a$, $A_{a,b,1}$, $B_{a,b,1}$ versus the defined $J_{m,a,n,b,\ell}$, $H_{m,a}$, $A_{m,a}$, $B_{n,b}$. I think it is better if we use the full indices.  Also, the divided graph's components should be the ordered sets $A_{m,a}$ and $B_{n,b}$, not $H_a$ and $H_b$.  Can you double-check this?}

For $i\in\NN$ and $j\in\NN_0$, since $V_{i,j}$ is finite, $G[V_{i,j}]\cong H_{m,a}$ for some $a\in\NN$, where $m:=i+j-1=\dist_T(R,V_{i,j})$. Label $V_{i,j}$ by $a$. If $V_{i,j}$ is labelled $a$ and $V_{i,j+1}$ is labelled $b$, 
and $m=i+j-1$ and $n=i+(j+1)-1$, then there exists $c\in\NN$ such that $(G[V_{i,j}\cup V_{i,j+1}],V_{i,j},V_{i,j+1})\cong (J_{m,a,n,b,\ell},V(H_{m,a}),V(H_{n,b}))\in \JJ$; label the edge $V_{i,j}V_{i,j+1}$ of $Q$ by $\ell$. For $i\in\NN$, if $V_{i,0}$ is labelled $a$ and $V_{i+1,0}$ is labelled $b$, then label the edge $V_{i,0}V_{i+1,0}$ of $Q$ by 1. In this case, if  $m=i+0-1$ and $n=(i+1)+0-1$, then $(G[V_{i,0}\cup V_{i+1,0}],V_{i,0},V_{i+1,0})\cong(J_{m,a,n,b,1},V(H_{m,a}),V(H_{n,b}))\in \JJ$ since there are no edges between $V(H_{m,a})$ and $V(H_{n,b})$ in $J_{m,a,n,b,1}$.

%\david{old paragraph to be removed} For $i\in\NN$ and $j\in\NN_0$, since $V_{i,j}$ is finite, $G[V_{i,j}]\cong H_a$ for some $a\in\NN$. Label $V_{i,j}$ by $a$. If $V_{i,j}$ is labelled $a$ and $V_{i,j+1}$ is labelled $b$, then there exists $c\in\NN$ such that $(G[V_{i,j}\cup V_{i,j+1}],V_{i,j},V_{i,j+1})\cong(J_{a,b,\ell},H_a,H_b)\in \JJ$; label the edge $V_{i,j}V_{i,j+1}$ of $Q$ by $\ell$. For $i\in\NN$, if $V_{i,0}$ is labelled $a$ and $V_{i+1,0}$ is labelled $b$, then label the edge $V_{i,0}V_{i+1,0}$ of $Q$ by 1. \note{Fixed a typo here}  Thus $(G[V_{i,0}\cup V_{i+1,0}],V_{i,0},V_{i+1,0})\cong(J_{a,b,1},H_a,H_b)\in \JJ$ since there are no edges between $A_{a,b,1}$ and $B_{a,b,1}$ in $J_{a,b,1}$. 

Together, this shows that $G\in\JJ^{(1)}$, and therefore $G$ is isomorphic to an induced subgraph of $U_{\JJ,1}$. 
\end{proof}

%
%\david{copy of previous definition, so that we can verify $G\in\JJ^{(1)}$ (to be removed): 
%For any $\HH$-system $\JJ$ and $k\in\NN$, let \defn{$\JJ^{(k)}$} be the set of all graphs $G$ with the following properties:
%\begin{itemize}
%\item there is a partition $\PART$ of $V(G)$, where each $A\in \PART$ is an ordered set (so $G[A]$ is an ordered graph);
%\item there is a  tree $T$ with $V(T)=\PART$ rooted at some part $R\in\PART$ (and thus $T$ is oriented away from $R$); 
%\item there is a $(k+1)$-colouring $c$ of $T$ such that $G/\PART$ is a spanning subgraph of $Q:=\GGG{T}{c}$ (which implies $\tw(G/\PART)\leq \tw(Q) \leq k$ by \cref{TreewidthCharacterisation}); and  
%\item there is a labelling of $Q$, such that:
%\begin{itemize}
%\item for each part $A\in\PART$, 
%if $\dist_T(R,A)=m$ and $A$ is labelled $a$, then $G[A]\cong H_{m,a}$; and
%\item for every oriented edge $AB\in E(Q)$, 
%\begin{itemize}
%\item if $\dist_T(R,A)=m$ and $A$ is labelled $a\in\NN$ and $\dist_T(R,B)=n$ and $B$ is labelled $b\in\NN$ and $AB$ is labelled $\ell\in\NN$, then 
%$(G[A\cup B],A,B)\cong (J_{m,a,n,b,\ell},A_{m,a},B_{n,b})$; and
%%(which implies that $G[A]\cong H_{m,a}$ and $G[B]\cong H_{n,b}$); and
%\item if $AB\in E(Q)\setminus E(G/\PART)$ then $\ell=1$. 
%\end{itemize}
%\end{itemize}
%\end{itemize}}

%%%%%%%%%%%%%%%%%%%%%%%%%%%%%%%%%%%%%%%%%%%%%
\subsection{Forcing Infinite Edge-Connectivity}
\label{InfiniteEdgeConnectivity}

As shown above, there is a graph $U$ that contains all planar graphs, but does not contain a subdivision of an infinite clique.  In particular, $U$ does not contain a subgraph of infinite vertex-connectivity. We now show that the analogous statement for edge-connectivity is false.  That is, infinite edge-connectivity is unavoidable in any graph that contains all planar graphs, and the same property holds for other natural graph classes. 

\note{This section has been substantially revised.}

\begin{thm}
\label{Universal edge-connectivity}
If a graph $U$ contains all planar graphs, then $U$ contains a subgraph of infinite edge-connectivity that contains all planar graphs.
\end{thm}

\begin{thm}
\label{Universal edge-connectivity2}
If a graph $U$ contains all $K_4$-minor-free graphs, then $U$ contains a subgraph of infinite edge-connectivity that contains all $K_4$-minor-free graphs.
\end{thm}

\begin{thm}
\label{Universal edge-connectivity3}
If a graph $U$ contains all outerplanar graphs, then $U$ contains a subgraph of infinite edge-connectivity that contains all outerplanar graphs.
\end{thm}

The proofs of these results depend on the following definitions. Let $H_0$ be a graph and $B \subseteq E(H_0)$.  For each $xy \in B$, add a new vertex of degree 2 adjacent to $x$ and $y$. Call the resulting graph $H_1$. For each $xy \in E(H_1) \setminus E(H_0)$, add a new vertex of degree 2 adjacent to $x$ and $y$. Call the resulting graph $H_2$. Continue like this defining $H_3,H_4, \ldots$ and let the \defn{$B$-limit} of $H_0$ be $\bigcup_{i \in N_0} H_i$.  Let $\mathcal{G}$ be a class of graphs. We say that $\GG$ is \defn{robust} if  $\GG$ is closed under disjoint union and for every $G \in \mathcal{G}$ there exists $H \in \mathcal{G}$ and a spanning tree $T$ of $H$ such that $G \subseteq H$ and the $E(T)$-limit of $H$ is in $\GG$.    

Observe that the class $\GG$ of planar graphs is robust because every planar graph is a subgraph of a connected planar graph and for every planar graph $G$, the $E(G)$-limit of $G$ is also planar.  Similarly, the class of $K_4$-minor-free graphs is robust.  Finally, observe that the class of outerplanar graphs is robust because every outerplanar graph is a subgraph of a connected outerplanar graph, and if $G$ is a connected outerplane graph, where every vertex of $G$ is on the boundary of face $f$, and $B$ is the set of edges on the boundary of $f$, then $G$ has a spanning tree consisting only of edges in $B$, and the $B$-limit of $G$ is outerplanar. The following lemma applied to these three examples proves the above three results. 

\begin{lem} \label{extremefamily}
Let $\GG$ be a robust class of graphs. If a graph $U$ contains every graph in $\GG$, then $U$ contains a subgraph of infinite edge-connectivity that contains all graphs in $\GG$.
\end{lem}

\begin{proof}
For each $G \in \GG$, we will define a graph $I(G) \in \GG$ such that $G \subseteq I(G)$ and $I(G)$ has infinite edge-connectivity.  Since $\GG$ is robust, there exists $H_0 \in \GG$ and a spanning tree $T$ of $H_0$ such that $G \subseteq H_0$ and the $E(T)$-limit of $H_0$ is in $\GG$.  Let $I(G)$ be the $E(T)$-limit of $H_0$.  Since $\GG$ is robust, $I(G) \in \GG$.  Let $x,y \in V(I(G))$.  Thus, $x,y \in V(H_\ell)$ for some $\ell \in \NN_0$, where $H_0, H_1, \dots$ is the sequence of graphs used to define the $E(T)$-limit of $H_0$.  Since $T$ is a spanning tree of $H_0$, $B_\ell:=E(H_\ell) \setminus \bigcup_{i=0}^{\ell-1} E(H_i)$ are the edges of a spanning tree of $H_\ell$.   Thus, there is an $xy$-path $P$ in $H_\ell$ such that $E(P) \subseteq B_\ell$.  By the definition of $I(G)$, for all distinct $ab, cd \in E(P)$ there is an infinite collection $\mathcal{P}_{ab}$ of edge-disjoint $ab$-paths in $I(G)$ and an infinite collection $\mathcal{P}_{cd}$ of edge-disjoint $cd$-paths in $I(G)$ such that $E(P_{ab}) \cap E(P_{cd})=\emptyset$ for all $P_{ab} \in \mathcal{P}_{ab}$ and $P_{cd} \in \mathcal{P}_{cd}$.   Thus, there are infinitely many edge-disjoint $xy$-paths in $I(G)$. Hence $I(G)$ has infinite edge-connectivity, as required. 

Now assume $U$ contains all graphs in $\mathcal{G}$. Let $U'$ be the union of all subgraphs $H$ of $U$ such that $H \in \mathcal{G}$ and $H$ has infinite edge-connectivity.  By definition, $U'$ contains $I(G)$ for every $G \in \mathcal{G}$, and hence $U'$ contains every $G \in \mathcal{G}$ (since $G \subseteq I(G)$).  Moreover, each component of $U'$ has infinite edge-connectivity.  It remains to prove that some component of $U'$ contains every $G \in \mathcal{G}$. Let $U_1',U_2', \ldots$ be the components of $U'$.   For each $i \in \NN$, let $M_i \in \mathcal{G}$ such that $U_i'$ does not contain $M_i$. Let $M$ be the disjoint union of $M_1, M_2,  \ldots$. Since $\mathcal{G}$ is closed under disjoint union, $M \in \mathcal{G}$.  However, $I(M)$ is not a subgraph of any $U_i'$ and hence not a subgraph of $U'$, a contradiction which proves the lemma.
\end{proof}

%%%%%%%%%%%%%%%%%%%%%%%%%%%%%%%%%%%%%%%%%%%%%%%%%%%
\section{Final Remarks}

We conclude with several remarks and open problems.

\medskip\textbf{The Primary Question:} 
What is the simplest graph that contains all planar graphs? We have shown that $\SS_6\boxtimes \PP$ and $\SS_3\boxtimes \PP\boxtimes K_3$ contain all planar graphs. It is a tantalising open problem whether $\SS_3\boxtimes\PP$  contains all planar graphs. Note that for each $\ell\in\NN$, \citet{DJMMUW20} constructed a planar graph (with simple treewidth 3)  that is not a subgraph of $\TT_2 \boxtimes \PP \boxtimes K_{\ell}$. More generally, for all $k,\ell\in\NN$, \citet{DJMMUW20} constructed a graph with simple treewidth $k$  that is not a subgraph of $\TT_{k-1} \boxtimes \PP \boxtimes K_{\ell}$. 

\medskip\textbf{Minimality:} 
Does there exist a graph $U_0$ that contains every planar graph and is a subgraph of every graph that contains every planar graph? Note that $U_0$ must be a common subgraph of $K_{\aleph_0,\aleph_0,\aleph_0,\aleph_0}$, 
$\SS_6 \boxtimes \PP$, $\SS_3 \boxtimes \PP \boxtimes K_3$, and the graph in \cref{KtMinorFreeNoSubdiv} (with $t=5$); it must also contain $K_{\aleph_0}$ as a minor. 

\medskip\textbf{Acyclic Colourings:} 
One way to measure the `quality' of a  graph that contains every planar graph is via the acyclic chromatic number. The result of \citet{KY03} mentioned in \cref{GeneralisedColouringNumbers} along with \cref{GenColourProduct} implies 
 \begin{equation}
 \label{AcyclicColouringProduct}
 \chi_\text{a}( \TT_k \boxtimes P ) 
 \leq 
 \col_2( \TT_k \boxtimes P ) 
 \leq 
 5(k+1).
 \end{equation}
With \cref{InfinitePlanarStructure6} this implies that $\SS_6 \boxtimes \PP$ contains every planar graph and is acyclically $35$-colourable.  What is the minimum $c$ such that there exists an acyclically $c$-colourable graph that contains every planar graph? \citet{Borodin-DM79} proved that every finite planar graph has an acyclic $5$-colouring. An easy application of \cref{Konig} shows that acyclic chromatic number is extendable. Thus every planar graph has an acyclic $5$-colouring. More generally, \cref{InfiniteSurfaceProduct} and \eqref{AcyclicColouringProduct} imply that $(\SS_6+K_{2g})\boxtimes \PP$ contains every graph of Euler genus $g$ and has an acyclic $(10g+35)$-colouring. \citet{AMS96} proved that graphs with Euler genus $g$ are acyclically $O(g^{4/7})$-colourable (and this bound is tight up to a polylogarithmic factor). Is there a graph that contains every graph of Euler genus $g$ and has acyclic chromatic number $o(g)$? The analogous questions for game chromatic number or oriented chromatic number are also of interest.

\medskip\textbf{List Colouring:} Given $d\in\NN$, is there a universal graph for the class of $d$-list-colourable graphs? For $d=5$, any such graph must contain every planar graph.

\medskip\textbf{Bounded Degree Planar Graphs:} If a graph $U$ contains all planar graphs with maximum degree 3, must $U$ have infinite maximum degree? Or is there a graph with bounded maximum degree that contains all planar graphs with maximum degree 3? More generally, is there is a graph with bounded maximum degree that contains all graphs with maximum degree 3 (regardless of planarity)? Of course, the above questions are of interest with 3 replaced by any constant. 

\medskip\textbf{Minor-Closed Classes:} 
\citet{DHV85} first addressed the question of which minor-closed classes have a universal element. \cref{InfiniteCliqueMinor}(a) implies that every proper minor-closed class that contains every planar graph has no universal element. It is open whether every proper minor-closed class excluding some finite planar graph has a universal element. The following result is in this direction:

\begin{thm}
The following are equivalent for a minor-closed class $\GG$:
\begin{enumerate}[label=(\alph*)]
\item $\GG$ has bounded treewidth, 
\item some graph with bounded treewidth contains every graph in $\GG$,
\item some finite planar graph is not in $\GG$.
\end{enumerate}
\end{thm}
\begin{proof}
First, (a) implies (b) since for every $k\in\NN$ there is a universal graph for the class of treewidth $k$ graphs (\cref{TreewidthUniversal}). 

To see that (b) implies (c), assume that some graph with treewidth $n\in\NN$ contains every graph in $\GG$. Since treewidth is subgraph-closed, every graph in $\GG$ has treewidth at most $n$. Thus, the $(n+1)\times (n+1)$ grid graph, which has treewidth $n+1$, is not in $\GG$. That is, some finite planar graph is not in $\GG$, as desired. 

Finally, (c) implies (a) by the Grid-Minor Theorem of \citet{RS-V} and since treewidth is extendable (\cref{TreewidthExtendable}). 
\end{proof}

\section*{Acknowledgements} 
Thanks to Kevin Hendrey, Daniel Mathews, and Alex Scott for helpful comments.  Huge thanks to the referee for numerous helpful comments, including identifying some errors in our first submission. The example after \cref{FindSpanningTree} is due to the referee. Subsequent to this work, several papers on infinite universal graphs have appeared~\citep{Lehner23,Lehner24,Krill23,Krill25,GH23,Georgakopoulos25}.

{\fontsize{11pt}{12pt}
\selectfont
%\bibliographystyle{DavidNatbibStyle}
%\bibliography{myBibliography}
\def\soft#1{\leavevmode\setbox0=\hbox{h}\dimen7=\ht0\advance \dimen7
  by-1ex\relax\if t#1\relax\rlap{\raise.6\dimen7
  \hbox{\kern.3ex\char'47}}#1\relax\else\if T#1\relax
  \rlap{\raise.5\dimen7\hbox{\kern1.3ex\char'47}}#1\relax \else\if
  d#1\relax\rlap{\raise.5\dimen7\hbox{\kern.9ex \char'47}}#1\relax\else\if
  D#1\relax\rlap{\raise.5\dimen7 \hbox{\kern1.4ex\char'47}}#1\relax\else\if
  l#1\relax \rlap{\raise.5\dimen7\hbox{\kern.4ex\char'47}}#1\relax \else\if
  L#1\relax\rlap{\raise.5\dimen7\hbox{\kern.7ex
  \char'47}}#1\relax\else\message{accent \string\soft \space #1 not
  defined!}#1\relax\fi\fi\fi\fi\fi\fi}

}
\printindex
\end{document}